\documentclass[10pt,a4paper,american]{amsart}
\usepackage[left=3.5cm, right=3.5cm]{geometry}
\usepackage[american]{babel}
\usepackage{amsthm}
\usepackage{amsbsy}
\usepackage{amstext}
\usepackage{amsfonts,amsmath,amsthm,latexsym,amssymb,marvosym,esint,xfrac,bbm,amssymb}
\usepackage{graphics, epsfig}
\usepackage{color}
\usepackage{tikz}
\tikzstyle{mybox} = [draw=black, very thick, rectangle, rounded corners, inner ysep=5pt, inner xsep=5pt]
\usepackage{pgfplots}
\usepackage[shortlabels]{enumitem}
\usepackage{appendix}
\usepackage{amsfonts}
\usepackage{mathcomp}
\usepackage{dsfont}
\usepackage{latexsym}
\usepackage{amssymb}
\usepackage{esint}
\usepackage{stmaryrd}
\usepackage{comment}
\usepackage[colorlinks,pdfpagelabels,pdfstartview = FitH,bookmarksopen = true,bookmarksnumbered = true,linkcolor = blue,plainpages = false, hypertexnames = true,citecolor = red,pagebackref=true]{hyperref}
\usepackage{cite}
\usepackage{etoolbox}

\newtheorem{theo}{Theorem}[section]

\newtheorem{lem}[theo]{Lemma}
\theoremstyle{definition}
\newtheorem{defin}[theo]{Definition}
\newtheorem*{lem*}{Lemma}

\newtheorem*{cor*}{Corollary}
\newtheorem*{theo*}{Theorem}

\newcommand{\norm}[1]{\lVert#1\rVert}

\DeclareMathOperator*{\supp}{supp}
\DeclareMathOperator*{\sgn}{sgn}
\def\Xint#1{\mathchoice
{\XXint\displaystyle\textstyle{#1}}%
{\XXint\textstyle\scriptstyle{#1}}%
{\XXint\scriptstyle\scriptscriptstyle{#1}}%
{\XXint\scriptscriptstyle\scriptscriptstyle{#1}}%
\!\int}
\def\XXint#1#2#3{{\setbox0=\hbox{$#1{#2#3}{\int}$ }
\vcenter{\hbox{$#2#3$ }}\kern-.6\wd0}}

\def\dashint{\Xint-}

\def\XXiint#1#2#3{{\setbox0=\hbox{$#1{#2#3}{\iint}$ }
\vcenter{\hbox{$#2#3$ }}\kern-.55\wd0}}

\renewcommand{\d}{\:\:\!\!\mathrm{d}}
\newcommand{\N}{\ensuremath{\mathbb{N}}}
\newcommand{\R}{\ensuremath{\mathbb{R}}}

\renewcommand{\b}{\mathfrak{b}}

\numberwithin{equation}{section}
\allowdisplaybreaks

 \makeatletter
\patchcmd{\@setaddresses}{\indent}{\noindent}{}{}
\patchcmd{\@setaddresses}{\indent}{\noindent}{}{}
\patchcmd{\@setaddresses}{\indent}{\noindent}{}{}
\patchcmd{\@setaddresses}{\indent}{\noindent}{}{}
\makeatother

\begin{document}
\renewcommand{\refname}{References}
\renewcommand{\abstractname}{Abstract}

\title[Existence, comparison principle and uniqueness]{Existence, comparison principle and uniqueness for doubly nonlinear anisotropic evolution equations}
\date{\today}

\author[M. Vestberg]{Matias Vestberg}
\address{Matias Vestberg,
Department of Mathematics, Uppsala University,
P.~O.~Box 480, 751 06, Uppsala, Sweden}
\email{matias.vestberg@math.uu.se}

\begin{abstract}
We prove the existence of solutions to the Cauchy-Dirichlet problem associated with a class of doubly nonlinear anisotropic evolution equations. We also demonstrate the existence of solutions to the corresponding Cauchy problem on $\R^N\times(0,T)$. Under some assumptions on the Caratheodory vector field we prove a comparison principle and utilize it to obtain a uniqueness result for the Cauchy-Dirichlet problem.
\end{abstract}
\maketitle

\setcounter{tocdepth}{1}

\begin{center}
\begin{minipage}{12cm}
  \small
  \tableofcontents
\end{minipage}
\end{center}

\vskip0.5cm \noindent 
 
\section{Introduction}\label{sec: intro}
\noindent This paper is concerned with existence, comparison principle and uniqueness results for a class of doubly nonlinear anisotropic evolution equations of the form
\begin{align}\label{eq:diffusion}
\partial_t (|u|^{\alpha-1}u)   - \nabla\cdot A(x,t,u,\nabla u) = f \quad \text{ in } \quad \Omega_T:=\Omega\times (0,T),
\end{align}
where $\Omega \subset \R^N$ is either open and bounded or the whole space $\R^N$. Here, $A$ is a vector field satisfying the structure conditions
\begin{align}
\label{cond:structure1} A(x,t,u, \xi)\cdot \xi &\geq \Lambda^{-1} \sum^N_{i=1}|\xi_i|^{p_i} - \tilde a(x,t),
 \\
\label{cond:structure2} |A_i(x,t,u, \xi)| &\leq \Lambda \big(\sum^N_{k=1} |\xi_k|^{p_k} + \tilde b(x,t) \big)^\frac{p_i-1}{p_i},\hspace{5mm}i \in \{1,\dots, N\},
\end{align}
where $\tilde a, \tilde b \in L^1(\Omega_T)$ are nonnegative functions and $\Lambda > 0$ is a constant. In order to obtain existence results, we also require the strict monotonicity condition
\begin{align}\label{cond:strict-monot}
 (A(x,t,u,\xi) - A(x,t,u,\eta))\cdot (\xi - \eta) > 0, \quad \xi \neq \eta.
\end{align}
To guarantee the measurability of $A(x,t,u,\nabla u)$ we assume the Caratheodory condition, i.e. that $(u,\xi) \mapsto A(x,t,u,\xi)$ is continuous for a.e. $(x,t)\in \Omega_T$. The continuity is also used explicitly when proving the existence of solutions. For the right-hand side, we assume
\begin{align*}
 f \in L^\frac{\alpha+1}{\alpha}(\Omega_T).
\end{align*}
The parameter range considered is 
\begin{align*}
 \alpha > 0, \quad p_j > 1.
\end{align*}

We prove two different existence results. One concerns the Cauchy-Dirichlet problem on a space-time cylinder $\Omega\times (0,T)$ where $\Omega \subset \R^N$ is open and bounded and the other result concerns solutions to the Cauchy problem in $\R^N\times (0,T)$.

For vector fields $A(x,u,\xi)$ that are independent of time and that exhibit a suitable Lipschitz continuity w.r.t. $u$ we are able to prove a comparison principle. In order for the comparison principle to hold, one of the solutions must be zero on the lateral boundary $\partial \Omega \times (0,T)$. We refer to Section \ref{sec:setting} for the precise conditions. 

The comparison principle allows us to obtain a uniqueness result for the Cauchy-Dirichlet problem with with vanishing boundary data on $\Omega\times(0,T)$. Another consequence of the comparison principle is a set of sufficient conditions for the nonnegativity of solutions to the Cauchy-Dirichlet problem.
We remark that all of the previously mentioned conditions and results are valid for the model case
\begin{align*}
 \partial_t(|u|^{\alpha-1}u) - \sum^N_{j=1} \partial_j(|\partial_j u|^{p_j-2} \partial_j u) = 0.
\end{align*}

Previously, an existence result for the Cauchy-Dirichlet problem for the model case with vanishing boundary condition but with a fairly general right-hand side was obtained by Sango in \cite{Sa} using a semidiscretization procedure. In \cite{DeTe}, Degtyarev and Tedeev treat the Cauchy problem in $\R^N\times (0,T)$ for the model case with an additional $u$-dependent term and a Radon measure as initial data. In doing so, they also outline the proof of existence in the model case for bounded spatial domains and vanishing boundary data on $\Omega\times (0,T)$. Recently, an existence result for anisotropic equations without double nonlinearity but with a rather weak coercivity condition was obtained in \cite{ZoReWa}. A Barenblatt solution for the model case of anisotropic diffusion without double nonlinearity (i.e. $\alpha=1$) was constructed in the fast diffusion case in \cite{FeVaVo}, and the case of slow diffusion was treated in \cite{CiaMoVe}. A very singular solution for an anisotropic fast diffusion equation of porous medium type was constructed in \cite{Va}. An elliptic doubly anisotropic equation was considered in \cite{LeEl}.
In the doubly nonlinear isotropic case, the literature on existence is vast. Techniques relevant for our case were used by Laptev in \cite{La} and \cite{La2}. Other works include \cite{AlLu, Be, Iv, Ish} to name a few.

Apart from the general structure conditions, one major difference between our existence result for the Cauchy-Dirichlet problem and the previously mentioned works on doubly nonlinear anisotropic diffusion is that we consider the case of nonvanishing values on the lateral boundary $\partial \Omega \times (0,T)$. The function describing the boundary values is allowed to be fairly irregular, and we have carefully explored which assumptions are needed for the argument to work. Our proof is based on Galerkin's method, and the argument is somewhat different in the two parameter ranges $\alpha\in(0,1)$ and $\alpha\geq 1$. Our aim is to provide clear and detailed proofs, exploring how the different parameter ranges, and assumptions on the boundary data as well as the structure conditions affect the arguments.

Comparison principles have a long history in the theory of diffusion equations. For the porous medium equation, various comparison principles have been obtained in \cite{LiPe, KiLiLu, AvLu}. See also the book \cite{Va2}. For Trudinger's equation, a special case of doubly nonlinear diffusion, a comparison principle was obtained in \cite{LiLi}. Under extra assumptions on the time derivative, comparison principles for isotropic doubly nonlinear equations were obtained in \cite{AlLu, Ba, Di}, and under other assumptions in \cite{Ot, IvMkJae, Iv}. A comparison principle for nonnegative solutions where one solution vanishes on the lateral boundary was proved in \cite{BoeDuGiaLiSche} for general doubly nonlinear isotropic equations in the case of time-independent vector fields. A comparison principle for the model case of doubly nonlinear equations was proved in \cite{BoeStru} under the assumption that the larger (super)solution has a positive lower bound. In the model case for anisotropic diffusion without double nonlinearity, some comparison principles were proved in \cite{CiaMoVe}. 

The  techniques employed in the proof of the comparison principle in this work are similar to those appearing in \cite{BoeDuGiaLiSche}, but we remark that we do not need to assume that the solutions are nonnegative on the whole domain. This observation appears to be new even in the isotropic case.

\vspace{2mm}
\noindent{\bf Acknowledgments.} This work was supported by the Wallenberg AI, Autonomous Systems and Software Program (WASP) funded by the Knut and Alice Wallenberg Foundation. The author wishes to express his gratitude towards Simone Ciani and Vincenzo Vespri for valuable discussions regarding the subject.

\section{Setting and main results}\label{sec:setting}
In this section we define the concept of weak solutions to the Cauchy-Dirichlet problem on $\Omega_T := \Omega \times (0,T)$, and the Cauchy problem on $\R^N \times (0,T)$. We also present our main results. We start by presenting the anisotropic Sobolev spaces that appear in the definitions.
 Given a vector ${\bf p} = (p_1,\dots, p_N)$ with $p_i > 1$ we set 
\begin{align*}
W^{1, {\bf p}}(\Omega) &:= \{ v \in W^{1,1}(\Omega)\,|\, \partial_i v \in L^{p_i}(\Omega)\},
\\
 W^{1, {\bf p}}_{\textnormal{o}}(\Omega) &:= \{ v \in W^{1,1}_\textrm{o}(\Omega)\,|\, \partial_i v \in L^{p_i}(\Omega)\}.
\end{align*}
The space $W^{1, {\bf p}}_{\textnormal{o}}(\Omega)$ which is mentioned above for completeness has been used extensively previously in the regularity theory, see for example \cite{CiaMoVe} and \cite{CiaVeVe}. However, when considering solutions to boundary value problems, it seems more natural to work with the space 
\begin{align*}
 \overline W^{1, {\bf p}}_{\textnormal{o}}(\Omega) := \overline{C^\infty_{\textnormal{o}}(\Omega)} \subset W^{1, {\bf p}}(\Omega),
\end{align*}
i.e. the closure of $C^\infty_{\textnormal{o}}(\Omega)$ in $W^{1, {\bf p}}(\Omega)$ w.r.t. the norm $u\mapsto \norm{u}_{L^1(\Omega)} + \sum^N_{j=1} \norm{\partial_j u}_{L^{p_j}(\Omega)}$. For bounded domains $\Omega$ the space $\overline W^{1, {\bf p}}_{\textnormal{o}}(\Omega)$ is contained in $W^{1, {\bf p}}_{\textnormal{o}}(\Omega)$. For bounded convex domains $\Omega$ the spaces coincide, however to our knowledge it is not known if $W^{1, {\bf p}}_{\textnormal{o}}(\Omega)$ and $\overline W^{1, {\bf p}}_{\textnormal{o}}(\Omega)$ coincide in general.

As we are working with evolution equations we will also need the following space involving time:
\begin{align*}
 L^{\bf p}(0,T; W^{1, {\bf p}}(\Omega)) &:= \{ v \in L^1(0,T;W^{1,1}(\Omega)) \,|\, \partial_i v \in L^{p_i}(\Omega_T)\}.
\end{align*}
Aided by the inclusion $i : \overline W^{1, {\bf p}}_{\textnormal{o}}(\Omega) \hookrightarrow L^1(\Omega)$, the continuous linear maps $\partial_k: \overline W^{1, {\bf p}}_{\textnormal{o}}(\Omega) \to L^{p_k}(\Omega)$, and the space $E$  consisting of equivalence classes of all measurable functions $(0,T) \to \overline W^{1, {\bf p}}_{\textnormal{o}}(\Omega)$ we may now introduce the Banach space 
\begin{align*}
 L^{\bf p}(0,T; \overline W^{1, {\bf p}}_{\textnormal{o}}(\Omega)) = \{ u \in E \,|\, i \circ u \in L^1(0,T; L^1(\Omega)), \, \partial_k \circ u \in L^{p_k}(0,T; L^{p_k}(\Omega)) \},
\end{align*}
which represents functions taking the value zero on the lateral boundary $\partial \Omega \times (0,T)$. This is a rather abstract looking definition relying on the Bochner integral, but as one might expect, the space is also isomorphic to a space of functions defined on $\Omega_T$.
\begin{lem}\label{lem:isomorphic-spaces}
 The space $L^{\bf p}(0,T; \overline W^{1, {\bf p}}_{\textnormal{o}}(\Omega))$ is isometrically isomorphic to the closure of $C^\infty_{\textnormal{o}}(\Omega_T)$  in $\{ f \in L^1(\Omega_T)\,| \, \partial_k f \in L^{p_k}(\Omega_T)\}$ with respect to the norm
 \begin{align*}
  v \mapsto \norm{v}_{L^1(\Omega_T)} + \sum^N_{k=1} \norm{\partial_k v}_{L^{p_k}(\Omega_T)}.
 \end{align*}
\end{lem}
\noindent For the proof we refer to Appendix \ref{app:equiv-sol}.
In the study of solutions to the Cauchy problem we will need the space
\begin{align*}
 U^{1,{\bf p}}_{\alpha+1} := \{ u \in L^{\alpha+1}(S_T)\,|\, \partial_k u \in L^{p_k}(S_T),\, k \in \{1\dots N\}\} 
\end{align*}
endowed with the norm 
\begin{align*}
 \norm{u}_{U^{1,{\bf p}}_{\alpha+1}} := \norm{u}_{L^{\alpha+1}(S_T)} + \sum^N_{k=1} \norm{\partial_k u}_{L^{p_k}(S_T)},
\end{align*}
and the closed subspace 
\begin{align*}
 \mathring{U}^{1,{\bf p}}_{\alpha+1} := \textnormal{cl} \{ \varphi\in C^\infty(\R^N \times [0,T])\,|\, \supp \varphi \subset K\times[0,T] \textnormal{ for some compact } K \subset \R^N \}.
\end{align*}

We use the following notion of weak solutions of the equation. 
\begin{defin}[Weak solutions]\label{def:weaksol}
A function $u \in L^{\bf p}(0,T; W^{1, {\bf p}}(\Omega)) \cap L^{\alpha+1}(\Omega_T)$ for bounded $\Omega$ is a solution of \eqref{eq:diffusion} if 
\begin{align}\label{eq:weak_form}
&\iint_{\Omega_T} A(x,t,u,\nabla u)\cdot \nabla \varphi- |u|^{\alpha-1} u\partial_t \varphi\d x\d t = \iint_{\Omega_T} f \varphi \d x \d t,
\end{align}
for all $\varphi \in C^\infty_{\textnormal{o}}(\Omega_T)$. Weak solutions on $S_T:= \R^N\times(0,T)$ are defined in an analogous way except that the integrals are taken over $S_T$ and we require $u$ to be in $U^{1,{\bf p}}_{\alpha+1}$. 
\end{defin}
For the domain $S_T$ we have chosen a different function space simply because it is not natural to assume $L^1$-integrability of the function and its derivatives in this setting. A formulation which is consistent for both type of domains is to require that the solution is in $L^{\alpha+1}$ and each derivative $\partial_k u$ is in $L^{p_k}$ on the full domain of the solution. It would also be possible to consider solutions with local integrability, but in this paper we will only work with globally integrable solutions.
We are now ready to introduce the notion of a weak solution to the Cauchy-Dirichlet problem.

\begin{defin}\label{def:prob-CD}
  Let $g\in L^{\bf p}(0,T; W^{1, {\bf p}}(\Omega))$. Suppose that $u_0 \in L^{\alpha+1}(\Omega)$ and $f \in L^\frac{\alpha+1}{\alpha}(\Omega_T)$. We say that $u \in L^{\bf p}(0,T; W^{1, {\bf p}}(\Omega)) \cap C([0,T];L^{\alpha+1}(\Omega))$ is a solution to the Cauchy-Dirichlet problem
 \begin{align}\label{prob:C-D}
   \left\{
\begin{array}{ll}
\partial_t \big( |u|^{\alpha -1} u\big)  - \nabla\cdot A(x,t,u,\nabla u) = f, & \quad \text{in } \Omega_T 
\\[5pt]
 u(x,0) = u_0(x),  & \quad x \in \Omega,
 \\
 u = g,  & \quad \text{on } \partial \Omega \times (0,T).
\end{array}
\right.
 \end{align}
if the following conditions hold:
\begin{enumerate}
 \item $u$ is a solution to the PDE in the sense of Definition \ref{def:weaksol}.
 \item\label{u(t)-convg-to-u_0} $u(t) \to u_0$ in $L^{\alpha+1}(\Omega)$ as $t\to 0$.
 \item $u \in g + L^{\bf p}(0,T; \overline W^{1, {\bf p}}_{\textrm{o}}(\Omega))$.
\end{enumerate}
\end{defin}
Note that since in the previous definition $u$ is assumed to belong to $C([0,T]; L^{\alpha+1}(\Omega))$, the function $u(t) \in L^{\alpha+1}(\Omega)$ is well-defined for every $t\in [0,T]$, and thus condition \eqref{u(t)-convg-to-u_0} makes sense. 
For solutions to the Cauchy problem on $S_T := \R^N\times(0,T)$ we make the following definition.
\begin{defin}\label{def:prob-Cauchy}
 Let $u_0\in L^{\alpha+1}(\R^N)$  and let $f\in L^\frac{\alpha+1}{\alpha}(S_T)$. We say that
 \\
 \noindent $u \in C([0,T]; L^{\alpha+1}(\R^N)) \cap U^{1,{\bf p}}_{\alpha+1}$ is a solution to the Cauchy Problem 
 \begin{align}\label{prob:Cauchy}
   \left\{
\begin{array}{ll}
\partial_t \big( |u|^{\alpha -1} u\big)  - \nabla\cdot A(x,t,u,\nabla u) = f, & \quad \text{in } S_T 
\\[5pt]
 u(x,0) = u_0(x),  & \quad x \in \R^N,
\end{array}
\right.
 \end{align}
 if the following conditions hold:
 \begin{enumerate}
 \item $u$ is a solution to the PDE on $S_T$ in the sense of Definition \ref{def:weaksol}.
 \item $u(t) \to u_0$ in $L^{\alpha+1}(\R^N)$ as $t\to 0$.
 \end{enumerate}
\end{defin} 

\subsection{Existence}
For the Cauchy-Dirichlet problem we prove the following existence result.
 \begin{theo}\label{thm:existence}
  Let $\Omega$ be open and bounded. Suppose that $A(x,t,u,\xi)$ is a Caratheodory vector field satisfying \eqref{cond:structure1}, \eqref{cond:structure2} and \eqref{cond:strict-monot}. Suppose that the following conditions hold for the initial and boundary values and right-hand side:
  \begin{align*}
   u_0 \in L^{\alpha+1}(\Omega),\quad 
   g \in L^\infty(\Omega_T) \cap L^{\bf p}(0,T; W^{1, {\bf p}}(\Omega)), \quad  \partial_t g \in L^{\alpha+1}(\Omega_T), \quad f \in L^\frac{\alpha+1}{\alpha}(\Omega_T).
  \end{align*}
  Then there is a solution $u \in L^{\bf p}(0,T; W^{1, {\bf p}}(\Omega)) \cap C([0,T];L^{\alpha+1}(\Omega))$ to the Cauchy-Dirichlet problem \eqref{prob:C-D} in the sense of Definition \ref{def:prob-CD}.
 \end{theo}
For the Cauchy problem in $\R^N$ we prove the following.
\begin{theo}\label{thm:existence-S_T}
 Let $u_0 \in L^{\alpha+1}(\R^N)$ and $f \in L^\frac{\alpha+1}{\alpha}(S_T)$. Let $A(x,t,u,\xi)$ be a Caratheodory vector field satisfying the structure conditions \eqref{cond:structure1} and \eqref{cond:structure2} and  \eqref{cond:strict-monot} for $(x,t) \in S_T$. Then there is a solution $u\in \mathring{U}^{1,{\bf p}}_{\alpha+1}$ to the Cauchy problem \ref{prob:Cauchy} in the sense of Definition \ref{def:prob-Cauchy}. 
\end{theo}

\subsection{Comparison principle and uniqueness}

With some additional assumptions on the vector field we can prove a comparison principle. However note that some assumptions on $A$ that were used to obtain existence can in fact be weakened or dropped completely.
To be precise, we assume that $A(x,u,\xi)$ is a Caratheodory vector field which has \textit{no explicit dependence on time}. We also assume the structure condition 
\begin{align}\label{cond:structure2-for-comp}
 |A_i(x,u, \xi)| &\leq \Lambda \big(\sum^N_{k=1} |\xi_k|^{p_k} + \tilde b(x) \big)^\frac{p_i-1}{p_i},\hspace{5mm}i \in \{1,\dots, N\},
\end{align}
where $\tilde b \in L^1(\Omega)$. This is a version \eqref{cond:structure2} adapted for time-independent vector fields. Note however, that no variant of \eqref{cond:structure1} is needed. Furthermore, we need to assume that the components of the vector field satisfy the following type of Lipschitz-continuity condition w.r.t. $u$:
\begin{align}\label{cond:lip-cont}
 |A_i(x,u_1, \xi) - A_i(x, u_2, \xi)| \leq C|u_1 - u_2|\Big(\tilde c(x) + \sum^N_{k=1} |\xi_k|^{p_k}\Big)^\frac{p_i-1}{p_i},
\end{align}
for some $\tilde c \in L^1(\Omega)$. Instead of the strict monotonicity condition \eqref{cond:strict-monot} it is sufficient to work with the following weaker variant:
\begin{align}\label{cond:monot}
 (A(x,u,\xi) - A(x,u,\eta)\cdot (\xi - \eta) \geq 0.
\end{align}
We also need to assume that the right-hand side is independent of time. 
The comparison principle which we prove below is the following.
\begin{theo}\label{thm:comparison}
 Let $A(x,u,\xi)$ be a vector field satisfying the assumptions \eqref{cond:structure2-for-comp}, \eqref{cond:lip-cont} and \eqref{cond:monot}. Let $w$ be a solution to the Cauchy-Dirichlet problem with initial data $w_0$, nonnegative boundary data $g$ and right-hand side $f_2\in L^\frac{\alpha+1}{\alpha}(\Omega)$. Let $v$ be a solution to the Cauchy-Dirichlet problem with initial data $v_0 \leq w_0$, vanishing boundary data, i.e. $v \in L^{\bf p}(0,T; \overline W^{1, {\bf p}}_{\textnormal{o}}(\Omega))$, and with right-hand side $f_1\in L^\frac{\alpha+1}{\alpha}(\Omega)$, $f_1 \leq f_2$. Then $v\leq w$ on $\Omega_T$. 
\end{theo}
By symmetry, one can also compare a solution with negative values on the lateral boundary to a solution with vanishing boundary condition. As expected, the comparison principle gives rise to a uniqueness result. Note however, that since we need one of the functions to vanish on the boundary, we obtain the uniqueness result specifically for solutions vanishing on the lateral boundary.
\begin{theo}
Let $A(x,u,\xi)$ be a time independent vector field satisfying the structure conditions \eqref{cond:structure2-for-comp} and \eqref{cond:structure1} (with time-independent $\tilde a$). Suppose also that $A$ satisfies the Lipschitz continuity condition \eqref{cond:lip-cont} and the strict monotonicity condition \eqref{cond:strict-monot}. Let $f\in L^\frac{\alpha+1}{\alpha}(\Omega)$ and $u_0\in L^{\alpha+1}(\Omega)$. Then there is a unique solution to the Cauchy-Dirichlet problem 
\begin{align*}
   \left\{
\begin{array}{ll}
\partial_t \big( |u|^{\alpha -1} u\big)  - \nabla\cdot A(x,u,\nabla u) = f, & \quad \text{in } \Omega_T 
\\[5pt]
 u(x,0) = u_0(x),  & \quad x \in \Omega,
 \\
 u = 0,  & \quad \text{on } \partial \Omega \times (0,T).
\end{array}
\right.
 \end{align*}
\end{theo}
\begin{proof}{}
 Existence follows from Theorem \ref{thm:existence}. Uniqueness follows from Theorem \ref{thm:comparison}.
\end{proof}
An interesting question is under which circumstances one can can ensure that a solution to the Cauchy-Dirichlet problem is nonnegative. The following theorem provides some sufficient conditions.
\begin{theo}
 Let $A(x,u,\xi)$ be a vector field satisfying the assumptions of Theorem \ref{thm:comparison}. Suppose also that $A(x,0, \bar 0) = \bar 0$. Suppose that $u$ solves the problem 
 \begin{align*}
   \left\{
\begin{array}{ll}
\partial_t \big( |u|^{\alpha -1} u\big)  - \nabla\cdot A(x,u,\nabla u) = f, & \quad \text{in } \Omega_T 
\\[5pt]
 u(x,0) = u_0(x),  & \quad x \in \Omega,
 \\
 u = g,  & \quad \text{on } \partial \Omega.
\end{array}
\right.
 \end{align*}
where $u_0$, $f$ and $g$ are nonnegative. Then $u$ is nonnegative.
 \end{theo}
\begin{proof}{}
 Apply Theorem \ref{thm:comparison} with $v=0$, $f_1 = 0$, and $w=u$, $f_2 = f$.
\end{proof}

\section{Preliminaries}
In this section we introduce some notation and various lemmas that will be used in the subsequent arguments.
\subsection{Algebraic quantities and estimates}
\begin{lem}\label{lem:alpha-est}
 For $u, v \in \R$ and $\alpha\in(0,1)$ we have
 \begin{align*}
  |u - v| \leq c_\alpha \big||u|^{\alpha-1}u - |v|^{\alpha-1}v\big| \big(|u|^{(1-\alpha)} + |v|^{(1-\alpha)}\big).
 \end{align*}
\end{lem}
\begin{proof}{} Using the fact that $f(t):= |t|^{\frac1\alpha-1}t$ is a $C^1$-function we have
 \begin{align*}
  |u - v| &= \big|f(|u|^{\alpha-1}u) - f(|v|^{\alpha-1}v) \big| = \Big| \int_0^1 \frac{d}{d s} f\big( |v|^{\alpha-1}v + s(|u|^{\alpha-1}u - |v|^{\alpha-1}v)\big) \d s\Big|
  \\
  &= \frac1\alpha \big||u|^{\alpha-1}u - |v|^{\alpha-1}v\big| \int_0^1 \big|s |u|^{\alpha-1}u + (1-s) |v|^{\alpha-1}v \big|^{\frac1\alpha - 1} \d s 
  \\
  &\leq c_\alpha \big||u|^{\alpha-1}u - |v|^{\alpha-1}v\big| \big(|u|^{(1-\alpha)} + |v|^{(1-\alpha)}\big).
 \end{align*}
\end{proof}
\begin{lem}[Lemme 1.1 of \cite{Raviart}, page 302]
    Let $\alpha \in (0,\infty)$. Then for $a, b \in \R$ we have
    \begin{equation}\label{mono-convex}
        ( |a|^{\alpha-1} a - |b|^{\alpha-1} b ) a \ge \frac{\alpha}{\alpha+1} ( |a|^{\alpha+1}-|b|^{\alpha+1}).   \end{equation}
\end{lem}
\noindent The following result was proved in \cite{CiaVeVe}.
\begin{lem}\label{lem:elementary_real}
 Let $\gamma > 1$. For all $a, b \in \R$ we have 
 \begin{align}\label{est:exponent_inside}
  |a-b|^\gamma \leq c\big||a|^{\gamma-1} a - |b|^{\gamma-1}b\big|.
 \end{align}
for a constant $c=c(\gamma)$. 
\end{lem}
\noindent For $\alpha > 0$ and $v,w\in \R$ we define the quantity
\begin{align*}
 \b_\alpha[v,w] :=& \tfrac{\alpha}{\alpha+1} ( |v|^{\alpha+1} - |w|^{\alpha+1}) - w (|v|^{\alpha-1}v - |w|^{\alpha-1}w) 
 \\
 =& \tfrac{1}{\alpha+1}(|w|^{\alpha+1} - |v|^{\alpha+1}) + |v|^{\alpha-1}v(v-w).
\end{align*}
The second expression for $\b_\alpha[v,w]$ shows that the quantity is always nonnegative, since it may be expressed using the convex function $F(v) = \tfrac{1}{\alpha+1}|v|^{\alpha+1}$ as 
\begin{align*}
 \b_\alpha[v,w] ¨= F(w) - (F(v) + F'(v)(w-v)) \geq 0.
\end{align*}
In order to find suitable estimates for $\b_\alpha[v,w]$, we need to distinguish between the parameter ranges $\alpha \in (0,1)$ and $\alpha \geq 1$. The following lemma, however, is valid for all $\alpha > 0$.
\begin{lem}\label{lem:b-property-all-alpha}
 Let $\alpha > 0$. Then for all $v, w \in \R$,
 \begin{align}\label{b1}
  \big| |w|^{\frac{\alpha-1}{2}}w - |v|^{\frac{\alpha-1}{2}}v \big|^2 \leq c \b_\alpha[v,w], 
 \end{align}
 for some $c=c(\alpha)$, and
 \begin{align}\label{b5}
  \b_\alpha[v,w] \leq (|v|^{\alpha-1}v - |w|^{\alpha-1}w)(v - w).
 \end{align}
\end{lem}
\begin{proof}{}
 In the case that $\alpha \in (0,1)$ the estimate \eqref{b1} follows from property (i) of the quantity $\b$ appearing in \cite[Lemma 2.3]{BoeDuKoSc}, with the choice $m:=\tfrac1\alpha > 1$, $u= |v|^{\alpha-1}v$ and $a = |w|^{\alpha-1} w$. In the case that $\alpha \geq 1$ we can use the same argument  as in the proof of property (i) in \cite[Lemma 2.3]{BoeDuKoSc}. The estimate \ref{b5} is true since a direct calculation shows that 
\begin{align}\label{bvw-bwv}
(|v|^{\alpha-1}v - |w|^{\alpha-1}w)(v - w) - \b_\alpha[v,w] = \b_\alpha[w,v] \geq 0.
\end{align}
\end{proof}
The next Lemma contains some useful estimates for $\b_\alpha[v,w]$ in the case $\alpha \in (0,1)$. 
\begin{lem}\label{lem:b_properties}
Let $v,w \in \R$ and $\alpha \in(0,1)$. Then there exists a constant $c$ depending only on $\alpha$ such that:
\begin{enumerate}[(i)]
\item\label{b2} $\tfrac 1 c (|w|+|v|)^{\alpha-1}|w - v|^2 \leq  \b_\alpha[v,w] \leq c (|w|+|v|)^{\alpha-1} |w - v|^2 $, 
\item\label{b3} $\b_\alpha[v,w] \leq c |v - w|^{1+\alpha}$.
\item\label{b4} $\big||v|^{\alpha-1}v - |w|^{\alpha-1}w\big|^\frac{\alpha+1}{\alpha} \leq c \b_\alpha[v,w]$,
\end{enumerate}
\end{lem}
\begin{proof}{}
The properties \ref{b2} and \ref{b3} can be obtained from properties (ii) and (iii) for the related quantity $\b$ which were proved in \cite[Lemma 2.3]{BoeDuKoSc} with the choice $m:=\tfrac1\alpha > 1$, $u= |v|^{\alpha-1}v$ and $a = |w|^{\alpha-1} w$. Property \ref{b4} is obtained by combining \ref{b1} with the estimate
\begin{align*}
||v|^{\alpha -1}v - |w|^{\alpha -1}w|^\frac{\alpha+1}{\alpha} = \big(||v|^{\alpha -1}v - |w|^{\alpha -1}w|^{\frac{(\alpha + 1)}{2\alpha}}\big)^2 \leq c\big||v|^\frac{\alpha-1}{2}v - |w|^\frac{\alpha-1}{2}w |^2,
\end{align*}
where we have used Lemma \ref{lem:elementary_real} and the act that $\gamma:= (\alpha+1)/(2\alpha)> 1$ for $\alpha\in(0,1)$. 
\end{proof}
Finally, we have the following useful estimates in the case $\alpha \geq 1$.
\begin{lem}\label{lem:b_alpha_geq1}
 Let $\alpha \geq 1$. Then for all $v,w\in \R$,
  \begin{enumerate}[(i)]
  \item\label{b7} $\b_\alpha[v,w] \leq c\big||v|^{\alpha-1}v - |w|^{\alpha-1}w\big|^\frac{\alpha+1}{\alpha}$,
  \item\label{b8} $|v - w|^{\alpha+1} \leq c \b_\alpha[v,w]$.
  \item\label{alpha-geq1-stuff} $|v - w|^{\alpha+1} \leq c (|v|^{\alpha-1}v - |w|^{\alpha-1}w)(v-w)$.
 \end{enumerate}
\end{lem}
\begin{proof}{}
 To prove \ref{b7}, combine \eqref{b5} and \eqref{est:exponent_inside} and we have 
 \begin{align*}
  \b_\alpha[v,w] \leq \big||v|^{\alpha-1}v - |w|^{\alpha-1}w\big| |v - w|^{\alpha \frac1\alpha} \leq c \big||v|^{\alpha-1} - |w|^{\alpha-1}w\big|^\frac{\alpha+1}{\alpha}.
 \end{align*}
 To prove \ref{b8}, use \eqref{est:exponent_inside} and \eqref{bvw-bwv}
 \begin{align*}
  |v - w|^{\alpha+1} = |v-w|^\alpha |v-w| &\leq c\big| |v|^{\alpha-1}v - |w|^{\alpha-1}w \big| |v - w| 
  \\
  &= c(|v|^{\alpha-1}v - |w|^{\alpha-1}w)(v - w) 
  \\
  &= c\b_\alpha[v,w] + c\b_\alpha[w,v] \leq c\b_\alpha[v,w],
 \end{align*}
where we also used the fact that $\b_\alpha[w,v]$ is nonnegative in the last step. To prove \ref{alpha-geq1-stuff} we simply write use Lemma \ref{lem:elementary_real} with $\gamma=\alpha$ as follows:
\begin{align*}
 |v - w|^{\alpha+1} \hspace{-0.5mm}=\hspace{-0.5mm} |v - w|^\alpha |v-w| \leq c \big||v|^{\alpha-1}v - |w|^{\alpha-1}w\big||v-w| \hspace{-0.5mm}=\hspace{-0.5mm} c (|v|^{\alpha-1}v - |w|^{\alpha-1}w)(v-w).
\end{align*}
\end{proof}

\vspace{3mm}
 Given $\delta>0$ we define the Lipschitz function $\mathcal{H}_\delta :\R\to \R$, 
\begin{align}\label{def:H_delta}
 \mathcal{H}_\delta(s) = \left\{
\begin{array}{ll}
0, & \quad s \leq 0,
\\[5pt]
 \frac{s}{\delta},  &\quad 0< s < \delta,
 \\
1,  & \quad s \geq \delta.
\end{array}
\right.
\end{align}
Using $\mathcal{H}_\delta$ we define the quantity
\begin{align*}
 \mathfrak{h}_\delta(z,z_o) := \int_{z_o}^z \mathcal{H}_\delta(s - z_o) \alpha |s|^{\alpha - 1 } \d s, \quad z, z_o \in \R.
\end{align*}
Using the odd reflection  $\widehat{\mathcal{H}}_\delta(s) := - \mathcal{H}_\delta(-s)$ we define the quantity
\begin{align*}
 \widehat{\mathfrak{h}}_\delta(z,z_o) := \int_{z_o}^z \widehat{\mathcal{H}}_\delta(s - z_o) \alpha |s|^{\alpha - 1 } \d s, \quad z, z_o \in \R.
\end{align*}
Note that both $\mathfrak{h}_\delta$ and $\widehat{\mathfrak{h}}_\delta$ are both nonnegative. We also need the related quantities
\begin{align*}
       \mathfrak{h}_\delta^\varepsilon(z,z_o) &:= \int_{z_o}^z \mathcal{H}_\delta(s - z_o) \big[(\alpha-1)(|s|+\varepsilon)^{\alpha - 2}|s| + (|s|+\varepsilon)^{\alpha - 1}\big] \d s, \quad z, z_o \in \R,
       \\
         \widehat{\mathfrak{h}}_\delta^\varepsilon(z,z_o) &:= \int_{z_o}^z \widehat{\mathcal{H}}_\delta(s - z_o) \big[(\alpha-1)(|s|+\varepsilon)^{\alpha - 2}|s| + (|s|+\varepsilon)^{\alpha - 1}\big] \d s, \quad z, z_o \in \R,
\end{align*}
where $\varepsilon > 0$. 
Counterparts of the quantities $\mathfrak{h}_\delta(z,z_o)$ and $\widehat{\mathfrak{h}}_\delta(z,z_o)$ were used previously in \cite{BoeDuGiaLiSche} to obtain a comparison principle for nonnegative solutions in the doubly nonlinear isotropic case. For us it will be important that
\begin{align}
\label{h-delta-upperbnd} 0 \leq &\mathfrak{h}_\delta(z,z_o) \leq \chi_{(z_o,\infty)}(z) \int_{z_o}^z  \alpha |s|^{\alpha - 1 } \d s = (|z|^{\alpha-1}z - |z_o|^{\alpha-1} z_o)_+,
\\
\label{h-delta-limit} \lim_{\delta \to 0} &\mathfrak{h}_\delta(z,z_o) = (|z|^{\alpha-1}z - |z_o|^{\alpha-1} z_o)_+.
\end{align}
Similarly, we note that 
\begin{align}
\label{hat-h-delta-upperbnd} 0 \leq& \widehat{\mathfrak{h}}_\delta(z,z_o) = -\int_{z_o}^z \mathcal{H}_\delta(z_o-s) \alpha |s|^{\alpha - 1 } \d s = \chi_{(z,\infty)}(z_o) \int^{z_o}_z \mathcal{H}_\delta(z_o-s) \alpha |s|^{\alpha - 1 } \d s
 \\
 \notag \leq& \chi_{(z,\infty)}(z_o) \int^{z_o}_z \alpha |s|^{\alpha - 1 } \d s = (|z_o|^{\alpha-1}z_o - |z|^{\alpha-1} z)_+,
 \\
\label{hat-h-delta-limit} \lim_{\delta \to 0}& \widehat{\mathfrak{h}}_\delta(z,z_o) = (|z_o|^{\alpha-1}z_o - |z|^{\alpha-1} z)_+.
\end{align}

\subsection{Exponential time mollification and time continuity of solutions}
We recall the definition of the exponential time mollification, utilized previously in \cite{KiLi}, see also \cite{BoeDuMa}. For $T>0$, $t\in [0,T]$, $h\in (0,T)$ and $v\in L^1(\Omega_T)$ we set
\begin{align}
\label{def:moll}
v_h(x,t):=\frac{1}{h}\int^t_0 e^\frac{s-t}{h}v(x,s)\d s.
\end{align}
Moreover, we define the reversed analogue by
\begin{align*}
v_{\overline h}(x,t) :=\frac{1}{h}\int^T_t e^\frac{t-s}{h}v(x,s)\d s.
\end{align*}
For details regarding the properties of the exponential mollification we refer to \cite[Lemma 2.2]{KiLi}, \cite[Lemma 2.2]{BoeDuMa} and \cite[Lemma 2.9]{St}. The properties of the mollification that we will use have been collected for convenience into the following lemma:
\begin{lem}\label{lem:expmolproperties} 
Suppose that $v \in L^1(\Omega_T)$, and let $p\in[1,\infty)$. Then the mollification $v_h$ defined in \eqref{def:moll} has the following properties:
\begin{enumerate}[(i)]
\item\label{expmol1}
If $v\in L^p(\Omega_T)$ then $v_h\in L^p(\Omega_T)$,
$$
\norm{v_h}_{L^p(\Omega_T)}\leq \norm{v}_{L^p(\Omega_T)},
$$
and $v_h\to v$ in $L^p(\Omega_T)$. A similar estimate also holds with $v_{\bar h}$ on the left-hand side.
\item\label{expmol2}
In the above situation, $v_h$ has a weak time derivative $\partial_t v_h$ on $\Omega_T$ given by
\begin{align*}
\partial_t v_h=\tfrac{1}{h}(v-v_h),
\end{align*}
whereas for $v_{\overline h}$ we have
\begin{align*}
\partial_t v_{\overline h}=\tfrac{1}{h}(v_{\overline h}-v).
\end{align*}
\item\label{expmol3}
If $v$ has a weak partial derivative in space then so does $v_h$ and $v_{\bar h}$ and 
\begin{align*}
 \partial_j (v_h) = (\partial_j v)_h, \hspace{5mm} \partial_j (v_{\bar h}) = (\partial_j v)_{\bar h}.
\end{align*}
\item\label{expmol4} If $v\in L^p(0,T;L^{p}(\Omega))$ then $v_h, v_{\bar h} \in C([0,T];L^{p}(\Omega))$.
\item\label{expmol5} If $v \in U^{1,{\bf p}}_{\alpha+1}$ then $v_h \in U^{1,{\bf p}}_{\alpha+1}$. If $v \in \mathring{U}^{1,{\bf p}}_{\alpha+1}$ then $v_h \in \mathring{U}^{1,{\bf p}}_{\alpha+1}$.
\end{enumerate}
\end{lem}
\noindent 
The first claim in Property \ref{expmol5} follows from properties \ref{expmol1} and \ref{expmol3} which remain valid in the case $\Omega=\R^N$. The last claim of \ref{expmol5} follows almost immediately from the definition of the space $\mathring{U}^{1,{\bf p}}_{\alpha+1}$.

\begin{lem}\label{lem:integral-alpha-ineq}
 Let $u,v \in L^{\alpha+1}(\Omega)$ and $\alpha\in(0,1)$. Then
 \begin{align*}
  \int_\Omega |u - v|^{\alpha+1} \d x \leq c \Big(\int_\Omega |v|^{\alpha+1} + |w|^{\alpha+1}\d x \Big)^\frac{1-\alpha}{2} \Big(\int_\Omega (|v|^{\alpha-1}v - |w|^{\alpha-1}w)(v - w) \d x \Big)^\frac{\alpha+1}{2},
 \end{align*}
\end{lem}
for a constant $c$ only depending on $\alpha$.
\begin{proof}{}
 Introducing the parameters $\varepsilon>0$ and $\nu = (1-\alpha)(1+\alpha)/2$, using H\"older's inequality with the exponents $2/(1-\alpha)$ and $2/(\alpha+1)$, the estimate \eqref{b5} and Property \ref{b2} of Lemma \ref{lem:b_properties} we have 
 \begin{align*}
  \int_\Omega |u - v&|^{\alpha +1} \d x = \int_\Omega (|v| + |w| + \varepsilon)^\nu (|v| + |w| + \varepsilon)^{-\nu}|v - w|^{\alpha+1} \d x 
  \\
  &\leq \Big(\int_\Omega (|v| + |w| + \varepsilon)^{\alpha+1} \d x \Big)^\frac{1-\alpha}{2} \Big(\int_\Omega (|v| + |w| + \varepsilon)^{\alpha - 1}|v-w|^2 \d x \Big)^\frac{\alpha+1}{2}
  \\
  &\leq c \Big(\int_\Omega |v|^{\alpha+1} + |w|^{\alpha+1} + \varepsilon^{\alpha+1} \d x \Big)^\frac{1-\alpha}{2}
   \Big(\int_\Omega (|v|^{\alpha-1}v - |w|^{\alpha-1}w)(v - w) \d x \Big)^\frac{\alpha+1}{2}.
 \end{align*}
Passing to the limit $\varepsilon \to 0$ we have the result.
\end{proof}

The following Lemma will be useful in the proof of existence of solutions to the Cauchy-Dirichlet problem and the Cauchy problem on $\R^N$.

\begin{lem}\label{lem:equiv-timecont} 
Let $E \subset \R^N$ be measurable. Then $u \in C([0,T]; L^{\alpha+1}(E))$ if and only if
$|u|^{\alpha - 1}u \in C([0,T]; L^{\frac{\alpha+1}{\alpha}}(E))$. 
\end{lem}
\begin{proof}{}
 Consider first the case $\alpha > 1$. Then we can use Lemma \ref{lem:elementary_real} to estimate
\begin{align*}
 |u(x,t_1) - u(x,t_2)|^{\alpha+1}\leq \big| |u|^{\alpha-1}u(x,t_1) - |u|^{\alpha-1}u(x,t_2)\big|^\frac{\alpha+1}{\alpha},
\end{align*}
from which we see that $|u|^{\alpha - 1}u \in C([0,T]; L^{\frac{\alpha+1}{\alpha}}(E))$ implies $u \in C([0,T]; L^{\alpha+1}(E))$. Conversely, suppose that $C([0,T]; L^{\alpha+1}(E))$. Since $s\mapsto |s|^{\alpha-1}s$ is a $C^1$-function for $\alpha > 1$ we may write 
\begin{align*}
  \big| |u|^{\alpha-1}u(x,t_1) - &|u|^{\alpha-1}u(x,t_2)\big|^\frac{\alpha+1}{\alpha} \leq  c\big[ (|u(x,t_1)|^{\alpha-1} + |u(x,t_2)|^{\alpha-1})|u(x,t_1) - u(x,t_2)|\big]^\frac{\alpha+1}{\alpha}
  \\
  &\leq c\big(|u(x,t_1)|^\frac{(\alpha-1)(\alpha+1)}{\alpha} + |u(x,t_2)|^\frac{(\alpha-1)(\alpha+1)}{\alpha}\big)|u(x,t_1) - u(x,t_2)|^\frac{\alpha+1}{\alpha}.
\end{align*}
Integrating over $E$ and applying H\"older's inequality with the exponents $\alpha/(\alpha - 1)$ and $\alpha$ we have
\begin{align*}
  \int_E \big| |u|^{\alpha-1}u(x,t_1) - |u|^{\alpha-1}u(x,t_2)\big|^\frac{\alpha+1}{\alpha} \d x &\leq c \Big[ \int_E |u(x,t_1)|^{\alpha+1} + |u(x,t_2)|^{\alpha+1} \d x \Big]^\frac{\alpha-1}{\alpha} 
  \\
  & \quad \times \Big[ \int_E |u(x,t_1) - u(x,t_2)|^{\alpha+1} \d x\Big]^\frac1\alpha.
\end{align*}
Since $C([0,T]; L^{\alpha+1}(E)) \subset L^\infty(0,T; L^{\alpha+1}(E))$ we have that the expression in the first square brackets on the right-hand side is bounded for $t_1, t_2 \in [0,T]$. Moreover, the expression in the second square brackets vanishes as $t_2 \to t_1$ due to the assumed time-continuity. Thus, we have that $|u|^{\alpha-1}u \in C([0,T]; L^{\frac{\alpha+1}{\alpha}}(\Omega))$, which completes the proof in the case $\alpha>1$. In the case $\alpha \in (0,1)$ we obtain the result from the previous case applied to the function $v:= |u|^{\alpha-1}u$ and $\alpha$ replaced by $1/\alpha > 1$. The case $\alpha=1$ is trivial.
\end{proof}

The end of this section is devoted to obtaining useful properties of functions $u$ satisfying a condition similar to the weak formulation, but where the Caratheodory vector field $A$ is replaced by a general vector valued map $\mathcal{A}$ of sufficiently high integrability. These results are used in the proof of the existence of solutions below.
\begin{lem}\label{lem:before-time-cont}
 Suppose that 
 \begin{align*}
  g &\in L^{\alpha+1}(\Omega_T) \cap L^{\bf p}(0,T; W^{1, {\bf p}}(\Omega)), \quad  \partial_t g \in L^{\alpha+1}(\Omega_T),
  \\
  u &\in g + L^{\bf p}(0,T; \overline W^{1, {\bf p}}_{\textnormal{o}}(\Omega)) \textrm{ and } u \in L^{\alpha+1}(\Omega_T), 
  \\
  \mathcal{A} &: \Omega_T \to \R^N, \quad \mathcal{A}_j \in L^{p_j'}(\Omega_T), \quad f \in L^\frac{\alpha+1}{\alpha}(\Omega_T),
 \end{align*}
and that 
\begin{align}\label{weak:general}
&\iint_{\Omega_T} \mathcal{A}\cdot \nabla \varphi - |u|^{\alpha-1} u\partial_t \varphi \d x\d t = \iint_{\Omega_T} f \varphi \d x \d t,
\end{align}
holds for all $\varphi \in C^\infty_{\textnormal{o}}(\Omega_T)$. Define the class of functions
\begin{align*}
 \mathcal{V} := \{ w \in L^{\alpha+1}(\Omega_T)\,|\, w \in L^{\bf p}(0,T; \overline W^{1, {\bf p}}_{\textnormal{o}}(\Omega)), \, \partial_t w\in L^{\alpha+1}(\Omega_T)\}.
\end{align*}
Then for every $\zeta \in C^\infty_{\textnormal{o}}((0,T);[0,\infty))$ and $w \in \mathcal{V}$ we have 
\begin{align}\label{towards-time-cont}
 \iint_{\Omega_T} \b_\alpha[u, w+g]\zeta'(t) \d x \d t &= \iint_{\Omega_T}\big(|u|^{\alpha-1}u - |w+g|^{\alpha-1}(w+g)\big)\partial_t (w+g)\zeta \d x \d t
 \\
 \notag &\quad +\iint_{\Omega_T} \mathcal{A}\cdot \nabla (u-w-g)\zeta - f \zeta(u - w - g) \d x \d t.
\end{align}
\end{lem}
\begin{proof}{}
Lemma \ref{lem:weaksol2} and Lemma \ref{lem:expmolproperties} show that we may use \eqref{weak:general} with the choice
\begin{align*}
 \varphi := \zeta(w - [u-g]_h). 
\end{align*}
The goal is to pass to the limit $h\to 0$. Due to matching H\"older exponents and \ref{expmol1} and \ref{expmol3} of Lemma \ref{lem:expmolproperties} we have
\begin{align}\label{mathcalA-convg}
 \iint_{\Omega_T} \mathcal{A}\cdot \nabla (\zeta(w - [u-g]_h))\d x \d t \xrightarrow[h\to 0]{} \iint_{\Omega_T} \mathcal{A}\cdot \nabla (w + g - u)\zeta\d x \d t.
\end{align}
Similarly, we see for the right-hand side of \eqref{weak:general} that 
\begin{align}\label{f-limit}
 \iint_{\Omega_T} f \zeta(w - [u-g]_h) \d x \d t \xrightarrow[h\to0]{} \iint_{\Omega_T} f \zeta(w + g - u) \d x \d t.
\end{align}
We bound the first term on the left-hand side of \eqref{weak:general} from above as follows: 
\begin{align}\label{diffterm-est}
 |u|^{\alpha-1} u\partial_t \varphi &= |u|^{\alpha-1} u (w-[u-g]_h) \zeta'  + |u|^{\alpha-1} u \zeta  \partial_t(w + g_h)
 -|u|^{\alpha-1}u  \partial_t u_h \zeta
 \\
 \notag &= |u|^{\alpha-1} u (w-[u-g]_h) \zeta'  + |u|^{\alpha-1} u \zeta  \partial_t(w + g_h) - |u_h|^{\alpha-1}u_h \partial_t u_h \zeta
 \\
 \notag &\quad + \big(|u_h|^{\alpha-1}u_h - |u|^{\alpha-1}u \big) \partial_t u_h \zeta 
 \\
 \notag &\leq |u|^{\alpha-1} u (w-[u-g]_h) \zeta'  + |u|^{\alpha-1} u \zeta ( \partial_t w + \partial_t g_h) - \tfrac{1}{\alpha+1}\partial_t |u_h|^{\alpha+1} \zeta,
\end{align}
where in the last step we have used \ref{expmol2} of Lemma \ref{lem:expmolproperties} to conclude that the term on the second last row is nonpositive. Integrating \eqref{diffterm-est} over $\Omega_T$ we have 
\begin{align}\label{difftermint-est}
 \iint_{\Omega_T} |u|^{\alpha-1} u\partial_t \varphi \d x \d t & \leq  \iint_{\Omega_T} |u|^{\alpha-1} u (w-[u-g]_h) \zeta'  + |u|^{\alpha-1} u \zeta ( \partial_t w + \partial_t g_h) \d x \d t
 \\
\notag & \quad + \iint_{\Omega_T} \tfrac{1}{\alpha+1}|u_h|^{\alpha+1} \zeta' \d x \d t,
\end{align}
where we have integrated by parts in the last term. Passing to the limit $h\to 0$ to obtain the corresponding quantities without the time mollification is straightforward except for the term $\partial_t g_h$ which requires some care. Note that 
\begin{align}\label{time-der-of-exp-mol}
 \partial_t g_h(x,t) = [\partial_t g]_h(x,t) + \frac1h g(x,0) e^{-\frac{t}{h}}, 
\end{align}
where the last term makes sense since $g$ has a time-continuous representative due to the fact that $g \in W^{1,\alpha+1}(0,T;L^{\alpha+1}(\Omega))$. Note that for $t\geq \delta>0$, we have
\begin{align*}
 \Big|\frac1h g(x,0) e^{-\frac{t}{h}}\Big|^{\alpha+1} \leq \delta^{-(\alpha+1)}\Big|\frac{t}{h} e^{-\frac{t}{h}}\Big|^{\alpha+1} |g(x,0)|^{\alpha+1} \leq c_\delta |g(x,0)|^{\alpha+1},
\end{align*}
where the last expression is integrable over $\Omega_T$ so that by the dominated convergence theorem 
\begin{align*}
 \int^T_\delta\int_\Omega \Big|\frac1h g(x,0) e^{-\frac{t}{h}}\Big|^{\alpha+1} \d x \d t \xrightarrow[h\to 0]{} 0.
\end{align*}
Furthermore, since $\partial_t g \in L^{\alpha+1}(\Omega_T)$ we have that $[\partial_t g]_h \to \partial_t g$ in $L^{\alpha+1}(\Omega_T)$ as $h\to 0$ according to \ref{expmol1} of Lemma \ref{lem:expmolproperties}. These observations and \eqref{time-der-of-exp-mol} show that $\partial_t g_h$ converges to $\partial_t g$ in $L^p(\Omega\times (\delta,T))$ as $h\to 0$ which is sufficient for our purposes due to the presence of the compactly supported function $\zeta$ in \eqref{difftermint-est}. Thus, we can pass to the limit on the right-hand side of \eqref{difftermint-est}. Combining this with \eqref{mathcalA-convg} and \eqref{f-limit} we end up with 
\begin{align*}
 0 &\leq \iint_{\Omega_T} |u|^{\alpha-1} u (w + g - u) \zeta'  + |u|^{\alpha-1} u \partial_t(w + g)\zeta  + \tfrac{1}{\alpha+1}|u|^{\alpha+1} \zeta' \d x \d t 
 \\
 &\quad -\iint_{\Omega_T} \mathcal{A}\cdot \nabla (w + g - u)\zeta \d x \d t + \iint_{\Omega_T} f \zeta(w + g - u) \d x \d t
 \\
 &= \iint_{\Omega_T} |u|^{\alpha-1} u (w + g - u) \zeta'  + \big[|u|^{\alpha-1} u - |w+g|^{\alpha-1}(w+g)\big] \partial_t(w + g)\zeta \d x \d t
 \\
 & \quad + \iint_{\Omega_T} \tfrac{1}{\alpha+1}|u|^{\alpha+1} \zeta' - \tfrac{1}{\alpha+1} |w+g|^{\alpha+1} \zeta'\d x \d t + \iint_{\Omega_T} \mathcal{A}\cdot \nabla (u - w - g)\zeta \d x \d t 
 \\
 &\quad + \iint_{\Omega_T} f \zeta(w + g - u) \d x \d t
 \\
 &= \iint_{\Omega_T} \big[|u|^{\alpha-1} u - |w+g|^{\alpha-1}(w+g)\big] \partial_t(w + g)\zeta \d x \d t -\iint_{\Omega_T} \b_\alpha[u,w+g]\zeta' \d x \d t
 \\
 &\quad  + \iint_{\Omega_T} \mathcal{A}\cdot \nabla (u- w - g)\zeta \d x \d t - \iint_{\Omega_T} f \zeta(u - w - g) \d x \d t.
\end{align*}
This proves ``$\leq$'' in \eqref{towards-time-cont}. For the other direction, choose the test function $\varphi = \zeta(w - [u-g]_{\bar h})$ and proceed as before.
\end{proof}

\begin{lem}\label{lem:time-cont}
 Suppose that $g$, $u$, $\mathcal{A}$ and $f$ satisfy the assumptions in Lemma \ref{lem:before-time-cont}. Then $u$ has a representative in $C([0,T];L^{\alpha+1}(\Omega))$. 
\end{lem}
\begin{proof}{}
Let $\psi \in C^\infty([0,T])$ such that $\psi=1$ on $[0,3T/4]$, $\supp \psi \subset [0, 7T/8]$ and $|\psi'| \leq 16/T$. For $\tau \in [0,T/2]$ and $\varepsilon \in (0, T/4)$ define the function 
\begin{align*}
 \chi^\tau_\varepsilon(t) := \left\{
\begin{array}{ll}
0, & \quad t < \tau
\\[5pt]
 \varepsilon^{-1}(t - \tau),  & \quad t \in [\tau, \tau + \varepsilon]
 \\
1,  & \quad t > \tau.
\end{array}
\right.
\end{align*}
We use \eqref{towards-time-cont} with $\zeta = \chi^\tau_\varepsilon \psi$ and $w=[u-g]_{\bar h}$ and obtain
\begin{align}\label{pineapple}
 &\varepsilon^{-1} \int^{\tau+\varepsilon}_\tau \int_\Omega \b_\alpha[u, [u-g]_{\bar h} + g] \d x \d t = -\iint_{\Omega_T} \b_\alpha[u, [u-g]_{\bar h} + g] \chi^\tau_\varepsilon \psi' \d x \d t
 \\
 \notag & +\iint_{\Omega_T}\big(|u|^{\alpha-1}u - |[u-g]_{\bar h}+g|^{\alpha-1}([u-g]_{\bar h} + g)\big)\partial_t ([u-g]_{\bar h} + g)\chi^\tau_\varepsilon \psi \d x \d t
 \\
 \notag & -\iint_{\Omega_T} \mathcal{A}\cdot \nabla (u-g - [u-g]_{\bar h})\chi^\tau_\varepsilon \psi + f \chi^\tau_\varepsilon \psi (u - g - [u-g]_{\bar h}) \d x \d t.
\end{align}
Note that by \ref{expmol2} of \ref{lem:expmolproperties}, we have
\begin{align*}
 &\big(|u|^{\alpha-1}u - |[u-g]_{\bar h}+g|^{\alpha-1}([u-g]_{\bar h} + g)\big)\partial_t [u-g]_{\bar h} 
 \\
 &\quad = \big(|u|^{\alpha-1}u - |[u-g]_{\bar h}+g|^{\alpha-1}([u-g]_{\bar h} + g)\big)\frac{\big([u-g]_{\bar h} + g - u\big)}{h} \leq 0.
\end{align*}
At this point we follow somewhat different approaches depending on the value of $\alpha$. If $\alpha\in(0,1)$ we use the previous estimate in \eqref{pineapple}, together with \ref{b3} and \ref{b4} of Lemma \ref{lem:b_properties} to obtain
\begin{align}
&c \varepsilon^{-1} \int^{\tau+\varepsilon}_\tau \int_\Omega \big| |u|^{\alpha-1}u - |[u-g]_{\bar h} + g|^{\alpha-1}([u-g]_{\bar h} + g)\big|^\frac{\alpha+1}{\alpha} \d x \d t
\\
 \notag \leq &\varepsilon^{-1} \int^{\tau+\varepsilon}_\tau \int_\Omega \b_\alpha[u, [u-g]_{\bar h} + g] \d x \d t = -\iint_{\Omega_T} \b_\alpha[u, [u-g]_{\bar h} + g] \chi^\tau_\varepsilon \psi' \d x \d t
 \\
 \notag & +\iint_{\Omega_T}\big(|u|^{\alpha-1}u - |[u-g]_{\bar h}+g|^{\alpha-1}([u-g]_{\bar h} + g)\big)\partial_t g \chi^\tau_\varepsilon \psi \d x \d t
 \\
 \notag & +\iint_{\Omega_T} \mathcal{A}\cdot \nabla (u-g - [u-g]_{\bar h}) \chi^\tau_\varepsilon \psi + f \chi^\tau_\varepsilon \psi (u - g - [u-g]_{\bar h}) \d x \d t
 \\
 \notag & \leq c\iint_{\Omega_T} \b_\alpha[u, [u-g]_{\bar h} + g] \d x \d t + \sum^N_{j=1} \iint_{\Omega_T} |\mathcal{A}_j| |\partial_j(u-g) - [\partial_j(u-g)]_{\bar h}\d x \d t
 \\
 \notag & \quad + \iint_{\Omega_T}\big||u|^{\alpha-1}u - |[u-g]_{\bar h}+g|^{\alpha-1}([u-g]_{\bar h} + g)\big|  |\partial_t g | \d x \d t 
 \\
\notag & \quad + \iint_{\Omega_T}|f||u-g - [u-g]_{\bar h}|\d x \d t 
 \\
 \notag & \leq c \iint_{\Omega_T} \big|u - g - [u-g]_{\bar h}\big|^{\alpha+1} \d x \d t + \sum^N_{j=1} \norm{\mathcal{A}_j}_{L^{p'_j}(\Omega_T)} \norm{\partial_j(u-g) - [\partial_j(u-g)]_{\bar h}}_{L^{p_j}(\Omega_T)}
 \\
 \notag &\quad  + \norm{\partial_t g}_{L^{\alpha+1}(\Omega_T)}  \norm{|u|^{\alpha-1}u - |[u-g]_{\bar h}+g|^{\alpha-1}([u-g]_{\bar h} + g)}_{L^\frac{\alpha+1}{\alpha}(\Omega_T)}
 \\
 \notag & \quad + \norm{f}_{L^\frac{\alpha+1}{\alpha}(\Omega_T)} \norm{u - g - [u-g]_{\bar h}}_{L^{\alpha+1}(\Omega_T)}.
\end{align}
Since the integrand on the first row is integrable we have by taking $\varepsilon \to 0$ and using the Lebesgue differentiation theorem that
\begin{align*}
 &\int_\Omega \big| |u|^{\alpha-1}u - |[u-g]_{\bar h} + g|^{\alpha-1}([u-g]_{\bar h} + g)\big|^\frac{\alpha+1}{\alpha}(x,\tau) \d x 
 \\
 &\leq c \norm{u - g - [u-g]_{\bar h}}_{L^{\alpha+1}(\Omega_T)}^{\alpha+1}  + c\sum^N_{j=1} \norm{\mathcal{A}_j}_{L^{p'_j}(\Omega_T)} \norm{\partial_j(u-g) - [\partial_j(u-g)]_{\bar h}}_{L^{p_j}(\Omega_T)}
 \\
 & \quad + c\norm{\partial_t g}_{L^{\alpha+1}(\Omega_T)}  \norm{|u|^{\alpha-1}u - |[u-g]_{\bar h}+g|^{\alpha-1}([u-g]_{\bar h} + g)}_{L^\frac{\alpha+1}{\alpha}(\Omega_T)}
 \\
 &\quad + c\norm{f}_{L^\frac{\alpha+1}{\alpha}(\Omega_T)} \norm{u - g - [u-g]_{\bar h}}_{L^{\alpha+1}(\Omega_T)}, 
\end{align*}
for all $\tau \in [0,T/2]\setminus N_h$, where $N_h$ is a set of measure zero. Take a sequence $h_j \to 0$ and denote $N=\cup_{j\in \N} N_{h_j}$ which still has measure zero. Due to \ref{expmol1} of Lemma \ref{lem:expmolproperties}, we conclude from the previous estimate that 
\begin{align}\label{unif_convg_ae}
 \lim_{j\to\infty} \sup_{\tau \in [0,T/2]\setminus N} \int_\Omega \big| |u|^{\alpha-1}u - |[u-g]_{\bar h_j} + g|^{\alpha-1}([u-g]_{\bar h_j} + g)\big|^\frac{\alpha+1}{\alpha}(x,\tau) \d x = 0.
\end{align}
Note that each function $u^\alpha_j:= |[u-g]_{\bar h_j} + g|^{\alpha-1}([u-g]_{\bar h_j} + g)$ is in $C([0,T];L^\frac{\alpha+1}{\alpha}(\Omega))$ due to \ref{expmol4} of Lemma \ref{lem:expmolproperties} and the fact that $g$ is in $C([0,T];L^{\alpha+1}(\Omega))$. This observation combined with the uniform limit \eqref{unif_convg_ae} on the dense set $[0,T/2]\setminus N$ and the completeness of $L^\frac{\alpha+1}{\alpha}(\Omega)$ shows that $u^\alpha_j$ converges uniformly on $[0,T/2]$ to a representative of $|u|^{\alpha-1}u$ which is therefore continuous on $[0,T/2]$ into $L^\frac{\alpha+1}{\alpha}(\Omega)$. The interval $[T/2, T]$ can be treated in a similar fashion so we conclude that $|u|^{\alpha-1}u \in C([0,T];L^\frac{\alpha+1}{\alpha}(\Omega))$. Due to Lemma \ref{lem:equiv-timecont}, this means that $u \in C([0,T];L^{\alpha+1}(\Omega))$. If instead $\alpha \geq 1$, we use property \ref{b8} of Lemma \ref{lem:b_alpha_geq1} to estimate the left-hand side of \eqref{pineapple} downwards. The term involving $\b_\alpha[v,w]$ on the right-hand side of \eqref{pineapple} can be estimated upwards by property \ref{b7} of Lemma \ref{lem:b_alpha_geq1}. The rest of the argument is analogous to the previous case.
\end{proof}
We also have the following counterpart of Lemma \ref{lem:time-cont} in the case where the bounded domain $\Omega$ is replaced by $\R^N$.
\begin{lem}\label{lem:time-cont_RN}
 Let $u \in \mathring{U}^{1,{\bf p}}_{\alpha+1}$ and suppose that $\mathcal{A}_j \in L^{p_j'}(S_T)$ for $j \in \{1,\dots,N\}$. Let $f\in L^\frac{\alpha+1}{\alpha}(S_T)$ and suppose that 
 \begin{align}\label{weak:general-RN}
&\iint_{S_T} \mathcal{A}\cdot \nabla \varphi - |u|^{\alpha-1} u\partial_t \varphi \d x\d t = \iint_{S_T} f \varphi \d x \d t,
\end{align}
holds for all $\varphi \in C^\infty_{\textnormal{o}}(S_T)$. Then $u \in C([0,T]; L^{\alpha+1}(\R^N))$.
\end{lem}
\begin{proof}{}
 The argument is similar to the previous case and will only be outlined. We introduce the space 
 \begin{align*}
  \mathcal{V} := \big\{ w \,|\, w = v_h \textnormal{ or } w = v_{\bar h} \textnormal{ for some } v \in \mathring{U}^{1,{\bf p}}_{\alpha+1}\big\},
 \end{align*}
and note that $\varphi = \zeta ( w - u_h)$ and $\varphi = \zeta (w - u_{\bar h})$ are valid choices in \eqref{weak:general-RN} for $\zeta\in C^\infty_{\textnormal{o}}((0,T);[0,\infty))$ and $w \in \mathcal{V}$, since $\varphi$ and its first order derivatives can be approximated in suitable $L^p$-spaces by smooth compactly supported functions due to properties \ref{expmol2} and \ref{expmol5} of Lemma \ref{lem:expmolproperties}. This leads, as in the proof of Lemma \ref{lem:before-time-cont} to 
\begin{align*}
 \iint_{S_T} \b_\alpha[u,w]\zeta' \d x \d t = \iint_{S_T}\mathcal{A}\cdot (\nabla u - \nabla w) + (|u|^{\alpha-1}u - |w|^{\alpha-1}w)\partial_t w  - f(u-w) \d x \d t.
\end{align*}
We can use the previous equation with the choices $w = u_h$ and $w = u_{\bar h}$ as in the proof of Lemma \ref{lem:time-cont} to arrive at the conclusion.
\end{proof}

The following result is a variant of Lemme 1.2 of \cite{Raviart} adapted to our setting.
\begin{lem}\label{lem:initial-val} Let $u$, $g$, $\mathcal{A}$ and $f$ satisfy the assumptions of Lemma \ref{lem:before-time-cont}. Then we have
\begin{align}\label{TH-IVL}
 \iint_{\Omega_T} \mathcal{A} \cdot \nabla (u - g) \d x \d t =  \tfrac{\alpha}{\alpha+1}\norm{u(0)}_{L^{\alpha+1}(\Omega)}^{\alpha+1}  - \tfrac{\alpha}{\alpha+1}\norm{u(T)}_{L^{\alpha+1}(\Omega)}^{\alpha+1}
 \\
 \notag + \iint_{\Omega_T} f(u-g)\d x \d t - \int_\Omega |u|^{\alpha-1}ug(x,0)\d x + \int_\Omega |u|^{\alpha-1}ug(x,T)\d x - \iint_{\Omega_T} |u|^{\alpha-1}u \partial_t g \d x \d t.
\end{align}
\end{lem}
\begin{proof}{} \noindent We extend $u$ to the time interval $[-T, 2T]$ by reflection, i.e. 
\begin{equation}
    \tilde{u}(t)= \begin{cases} u(-t),& -T \leq t \leq 0,\\
    u(t), & 0\leq t \leq T,\\
    u(2T-t), & T \leq t \leq 2T.        
    \end{cases}
\end{equation} 
Similarly consider extensions $\tilde g$, $\tilde{\mathcal{A}}$ and $\tilde f$ of the functions $\mathcal{A}$ and $g$ to the set $\Omega \times (-T,2T)$. In the following we will denote also the extended maps simply by $u$, $g$, $\mathcal{A}$ and $f$. Due to the time-continuity of $u$, equation \eqref{weak:general} is true also for $\varphi$ compactly supported in $\Omega \times (-T,2T)$ provided that the integrals are taken over $\Omega \times (-T,2T)$. Due to Lemma \ref{lem:time-cont} and the equivalent weak formulation of Lemma \ref{lem:3rd-equiv-formulation} we have for $-T<t_1 < t_2 < 2T$ and all $v\in \overline W^{1, {\bf p}}_{\textnormal{o}}(\Omega) \cap L^{\alpha+1}(\Omega)$ that
\begin{align*}
 \int_\Omega \big( |u|^{\alpha-1}u(x,t_2) - |u|^{\alpha-1}u(x,t_1) \big) v(x) \d x = \int^{t_2}_{t_1} \int_\Omega f(x,s)v(x) - \mathcal{A}(x,s) \cdot \nabla v(x) \d x \d s.
\end{align*}
Especially, taking $t \in (0,T)$, $h>0$ sufficiently small, $t_1=t-h$, $t_2 = t$, noting that $v= u(\cdot,t) - g(\cdot, t)$ is a valid choice for a.e. $t$ and integrating over $t$ we see that
\begin{align}\label{grgr}
  &\iint_{\Omega_T} \big( |u|^{\alpha-1}u(x,t) - |u|^{\alpha-1}u(x,t-h) \big) u(x,t) \d x \d t 
 \\
 \notag &=  \int^T_0 \int^t_{t-h} \int_\Omega  f(x,s)(u(x,t) - g(x,t)) - \mathcal{A}(x,s) \cdot \nabla (u(x,t) - g(x,t))\d x \d s \d t
 \\
 \notag &\quad + \iint_{\Omega_T} \big( |u|^{\alpha-1}u(x,t) - |u|^{\alpha-1}u(x,t-h) \big) g(x,t) \d x \d t.
\end{align}
The left-hand side of \eqref{grgr} can be bounded from below using \eqref{mono-convex} as follows:
\begin{align*}
 \frac1h \int^T_{T-h} & \frac{\alpha}{\alpha+1}\int_\Omega |u(x,t)|^{\alpha+1} \d x \d t -  \frac1h \int^0_{-h} \int_\Omega \frac{\alpha}{\alpha+1}|u(x,t)|^{\alpha+1} \d x \d t 
 \\
 &= \frac1h \int^T_0\int_\Omega \frac{\alpha}{\alpha+1}\big(|u(x,t)|^{\alpha+1} - |u(x,t-h)|^{\alpha+1} \big) \d x \d t 
 \\
 &\leq \frac1h \int^T_0\int_\Omega   \big( |u|^{\alpha-1}u(x,t) - |u|^{\alpha-1}u(x,t-h) \big) u(x,t) \d x \d t.
\end{align*}
The last term on the right-hand side of \eqref{grgr} multiplied by $\tfrac1h$ can be expressed as
\begin{align*}
& \frac1h \iint_{\Omega_T} \big( |u|^{\alpha-1}u(x,t) - |u|^{\alpha-1}u(x,t-h) \big) g(x,t) \d x \d t = \frac1h \int^T_{T-h}\int_\Omega |u|^{\alpha-1}u g \d x \d t
 \\
 & - \frac1h \int^0_{-h}\int_\Omega |u|^{\alpha-1}u(x,t) g(x, t+h)\d x \d t + \hspace{-0.5mm} \int^{T-h}_0 \hspace{-2mm}\int_\Omega |u|^{\alpha-1}u(x,t)\frac{\big(g(x,t) - g(x,t+h)\big)}{h} \d x \d t.
\end{align*}
Combining the last three calculations we obtain
\begin{align}\label{rtrwv}
 &\frac1h \int^T_{T-h}  \frac{\alpha}{\alpha+1}\int_\Omega |u(x,t)|^{\alpha+1} \d x \d t -  \frac1h \int^0_{-h} \int_\Omega \frac{\alpha}{\alpha+1}|u(x,t)|^{\alpha+1} \d x \d t 
 \\
 \notag & \leq  \frac1h \int^T_0 \int^t_{t-h} \int_\Omega f(x,s)(u(x,t) - g(x,t)) - \mathcal{A}(x,s) \cdot \nabla (u(x,t) - g(x,t))\d x \d s \d t 
 \\
 \notag & \quad + \frac1h \int^T_{T-h}\int_\Omega |u|^{\alpha-1}u g \d x \d t  - \frac1h \int^0_{-h}\int_\Omega |u|^{\alpha-1}u(x,t) g(x, t+h)\d x \d t 
 \\
 \notag & \quad    - \int^{T-h}_0 \int_\Omega |u|^{\alpha-1}u(x,t)\frac{\big(g(x,t+h) - g(x,t)\big)}{h} \d x \d t.
\end{align}
We want to pass to the limit $h\to 0$. For the terms involving integrals over $[-h,0]$ or $[T-h,T]$ this is straightforward due to the appropriate time continuity of $u$ and $g$. Also for the integral containing $f$ and $\mathcal{A}$ passing to the limit works as expected due to the integrability properties of $f$, $\mathcal{A}$ and $\nabla (u-g)$. The nontrivial limit is the last term on the right-hand side. Here we note that 
\begin{align*}
 \int^{T-h}_0\int_\Omega \Big| \frac{\big(g(x,t+h) - g(x,t)\big)}{h} - \partial_t g(x,t) \Big|^{\alpha+1} \d x \d t \xrightarrow[h\to 0]{} 0,
\end{align*}
since $\partial_t g \in L^{\alpha+1}(\Omega_T)$. This can be verified by using the fundamental theorem of calculus together with the absolute continuity of $g$ along almost every line segment in the time direction, or alternatively by first approximating $g$ with a smooth function and passing to the limit. Hence, taking $h\to 0$ in \eqref{rtrwv} yields
\begin{align*}
 \frac{\alpha}{\alpha+1}\norm{u(T)}_{L^{\alpha+1}(\Omega)}^{\alpha+1} - \frac{\alpha}{\alpha+1}\norm{u(0)}_{L^{\alpha+1}(\Omega)}^{\alpha+1} \leq  \iint_{\Omega_T} f(u-g) - \mathcal{A} \cdot \nabla (u-g) \d x \d t
 \\
 -\int_\Omega |u|^{\alpha-1}ug(x,0)\d x + \int_\Omega |u|^{\alpha-1}ug(x,T)\d x - \iint_{\Omega_T} |u|^{\alpha-1}u \partial_t g \d x \d t.
\end{align*}
An estimate in the reverse direction can be proved by analyzing the quantity 
\begin{align*}
 \frac1h \int^T_0\int_\Omega   \big( |u|^{\alpha-1}u(x,t+h) - |u|^{\alpha-1}u(x,t) \big) u(x,t) \d x \d t 
\end{align*}
in an analogous manner.
\end{proof}
We have the following counterpart when $\Omega$ is replaced by $\R^N$.
\begin{lem}\label{lem:Raviart-RN}
 Let $u$, $\mathcal{A}$ and $f$ satisfy the assumptions of Lemma \ref{lem:time-cont_RN}. Then 
 \begin{align*}
  \iint_{S_T} \mathcal{A}\cdot \nabla u \d x \d t =  \tfrac{\alpha}{\alpha+1}\norm{u(0)}_{L^{\alpha+1}(\R^N)}^{\alpha+1}  - \tfrac{\alpha}{\alpha+1}\norm{u(T)}_{L^{\alpha+1}(\R^N)}^{\alpha+1}
 + \iint_{S_T} f u \d x.
 \end{align*}
\end{lem}
The proof is similar to the proof of Lemma \ref{lem:initial-val} except that we take $v=u(\cdot,t)$ without any boundary function $g$. The function space of $u$ is of critical importance as it guarantees that the choice of $v$ is valid.

 \section{Existence of Solutions}
In this section we prove the two existence results presented in Section \ref{sec:setting}.

\noindent \textbf{Proof of Theorem \ref{thm:existence}.}
We start by noting that since $g\in L^{\alpha+1}(\Omega_T)$ and since $\partial_t g \in L^{\alpha+1}(\Omega_T)$, $g$ has a representative in $C([0,T];L^{\alpha+1}(\Omega))$. In the following we will always work with the time-continuous representative of $g$. Note also that the time-continuous representative of $g$ satisfies $\norm{g(\cdot,t)}_{L^\infty(\Omega)} \leq \norm{g}_{L^\infty(\Omega_T)}$ for every $t\in [0,T]$. For the reader's convenience, the rest of the proof is divided into several steps. In order to make the presentation somewhat less technical, we first focus on the case $\alpha \in (0,1)$ and at the end of the proof we explain the modifications that are needed in order for the argument to work also in the case $\alpha \geq 1$.

\vspace{1mm}
\noindent \textbf{Step 1: Galerkin's method and local existence for the approximating system of ODEs.} Pick a sequence $(v_j)^\infty_{j=1}$ in $C^\infty_{\textnormal{o}}(\Omega)$ which is orthonormal in $L^2(\Omega)$ and whose linear combinations can approximate any element of $C^\infty_{\textnormal{o}}(\Omega)$ in the norm of $C^1(\bar \Omega)$. Then $(v_j)$ is also an orthonormal basis of $L^2(\Omega)$.  
Take a sequence $(\varepsilon_n)^\infty_{n=1} \subset (0,1)$ converging to zero. We will show that for each $n\in \N$, there are absolutely continuous functions $\xi^n_k \in C([0,T])$, $k\in \{1,\dots, n\}$ such that the function
 \begin{align*}
  u_n(x,t) = g(x,t) + \sum^n_{k=1} \xi^n_k(t) v_k(x) =: g(x,t) + w_n(x,t).
 \end{align*}
satisfies
\begin{align}\label{prob:finite-dim} 
 \left\{
\begin{array}{ll}
\alpha \int_\Omega (|u_n| + \varepsilon_n)^{\alpha - 1} \partial_t u_n v_m +  A(x, t, u_n, \nabla u_n) \cdot \nabla v_m \d x = \int_\Omega f(x,t)v_m \d x, 
\\
\quad \textnormal{for all } m\in \{1, \dots, n\}, \textnormal{ and for a.e. } t \in [0,T],
 \\
 u_n(x,0) = g(x,0) + \sum^n_{k=1} (\xi^n_0)^k v_k(x), \qquad  x \in \Omega,
\end{array}
\right.
\end{align}
where $\xi^n_0 \in \R^n$ in the last condition is chosen so that 
\begin{align*}
 u_n(\cdot,0) \xrightarrow[n\to\infty]{} u_0, \quad \textnormal{ in } L^{\alpha+1}(\Omega).
\end{align*}
Such $\xi^n_0$ exists since $C^\infty_{\textnormal{o}}(\Omega)$ is dense in $L^{\alpha+1}(\Omega)$ and since linear combinations of the vectors $v_j$ are dense in $C^\infty_{\textnormal{o}}(\Omega)$ w.r.t. the norm of $C^1(\bar \Omega)$ and hence also w.r.t the norm of $L^{\alpha+1}(\Omega)$. By introducing the maps
\begin{align*}
 &F: \R^n \times [0,T] \to \R^{n\times n}, 
 \\
 & F_{m, k}(\xi,t) := \alpha \int_\Omega \Big(|g(x,t) + \sum^n_{j=1} \xi_j v_j(x)| + \varepsilon_n \Big)^{\alpha -1} v_k(x) v_m(x) \d x,
 \\
 &K: \R^n\times[0,T]  \to \R^n, \quad K_m(\xi,t) := \alpha \int_\Omega \Big(|g(x,t) + \sum^n_{j=1} \xi_j v_j(x)| + \varepsilon_n \Big)^{\alpha -1} \partial_t g(x,t) v_m(x) \d x
 \\
 &G: \R^n\times[0,T]  \to \R^n, 
 \\
 & G_m(\xi, t) := \int_\Omega A\Big(x, t, g(x,t) + \sum^n_{j=1}\xi_j v_j(x), \nabla g(x,t) + \sum^n_{j=1} \xi_j \nabla v_j(x)\Big)\cdot \nabla v_m(x) \d x,
 \\
 &J: [0,T]\to \R^n, \quad J_m(t):= \int_\Omega f(x,t)v_m(x)\d x,
\end{align*}
we can rewrite \eqref{prob:finite-dim} as an initial value problem for a system of ODEs for the coefficient functions $\xi^n_k$:
\begin{align*} 
 \left\{
\begin{array}{ll}
F(\xi(t),t)\xi'(t) + K(\xi(t),t) + G(\xi(t),t) = J(t), \quad  \textnormal{for a.e. } t \in [0,T],
 \\
 \xi(0) = \xi^n_0.
\end{array}
\right.
\end{align*}
We first show that this system is solvable at least on some interval $[0,\delta]$. In order to do so, introduce the functions
\begin{align*}
 \theta_{\xi,t}:\Omega \to (0,\infty), \quad \theta_{\xi,t}(x) = \Big[\big|g(x,t) + \sum^n_{j=1}\xi_j v_j(x)\big| + \varepsilon_n\Big]^\frac{\alpha-1}{2}
\end{align*}
and note that $F_{m,k}$ can be written as
\begin{align*}
 F_{m, k}(\xi,t) = \alpha \big(  \theta_{\xi,t} v_k,  \theta_{\xi,t}v_m \big)_{L^2(\Omega)}.
\end{align*}
Since the functions $v_j$ are linearly independent and since the function $ \theta_{\xi,t}$ is strictly positive and bounded, the functions $ \theta_{\xi,t} v_k$ are also linearly independent functions in $L^2(\Omega)$ for a.e. $t$. Thus, $F(\xi,t)$ is invertible for a.e. $t$, and the system of ODEs may be written in the form
\begin{align}\label{syst:ODE}
 \left\{
\begin{array}{ll}
\xi'(t) = H(\xi(t),t), \quad \textnormal{for a.e. } t \in [0,T],
 \\
 \xi(0) = \xi^n_0, 
\end{array}
\right.
\end{align}
where
\begin{align}\label{def:H}
 H(\xi,t) := F(\xi,t)^{-1}(J(t) - K(\xi,t) - G(\xi,t)).
\end{align}
The system of ODEs is equivalent to the following integral equation:
\begin{align}\label{integ-eq}
 \xi(t) = \xi^n_0 + \int^t_0 H(\xi(s),s) \d s.
\end{align}
To see that this equation has a solution on some small time interval $[0,\delta]$ we will use the Schauder fixed point theorem. First, note that by the boundedness of $g$ and $v_j$, we have 
\begin{align}\label{est:theta-below}
 \theta_{\xi,t}(x) \geq c(1 + |\xi|)^\frac{\alpha-1}{2}, 
\end{align}
for a.e. $(x,t)$.
Note that for a.e. $t$, $F(\xi,t)$ is a symmetric and positive definite, and thus diagonalizable, with strictly positive eigenvalues. If $\lambda$ is an eigenvalue for $F(\xi,t)$, let $y\in \mathbb{S}^{n-1}$ be a corresponding eigenvector and note that we can use \eqref{est:theta-below} to estimate
\begin{align*}
 \lambda = y \cdot (\lambda y) = y \cdot F(\xi,t)y = \sum^n_{m,k=1} y_m \alpha \big(  \theta_{\xi,t} v_k,  \theta_{\xi,t}v_m \big)_{L^2(\Omega)} y_k = \alpha \int_\Omega \theta_{\xi,t}^2 \Big| \sum^n_{k=1} y_k v_k \Big|^2 \d x 
 \\
 \geq c(1+|\xi|)^{\alpha-1} \int_\Omega  \Big| \sum^n_{k=1} y_k v_k \Big|^2 \d x =  c(1+|\xi|)^{\alpha-1}|y|^2 = c(1+|\xi|)^{\alpha-1},
\end{align*}
where the penultimate step follows from the orthonormality of the functions $v_j$. Let $\Lambda$ be the greatest eigenvalue of the symmetric positive definite matrix $F(\xi,t)^{-1}$. Then $\Lambda>0$ and $\lambda := \Lambda^{-1}$ is an eigenvalue of $F(\xi,t)$ so we can estimate
\begin{align}\label{est:F-invers-op-norm}
 \norm{F(\xi,t)^{-1}}_{\textnormal{op}} = \Lambda = \frac{1}{\lambda}\leq c(1+|\xi|)^{1-\alpha}.
\end{align}
The definition of $K$ and the space of $\partial_t g$ show that 
\begin{align}\label{est:K-xi-t}
 |K(\xi, t)| \leq \mathcal{K}(t), \textnormal{ for some } \mathcal{K} \in L^{\alpha+1}([0,T]).
\end{align}
The structure condition \eqref{cond:structure2}, Young's inequality and the integrability properties of the spatial derivatives $\partial_k g$ show that 
\begin{align}\label{est:G-xi-t}
 |G(\xi,t)| \leq \mathcal{R}(t) + c(1+|\xi|)^{p_N}, \textnormal{ for some } \mathcal{R} \in L^1([0,T]).
\end{align}
Since also $J \in L^\frac{\alpha+1}{\alpha}([0,T])$ we see from the definition \eqref{def:H} and the estimates \eqref{est:F-invers-op-norm}, \eqref{est:K-xi-t} and \eqref{est:G-xi-t} that
\begin{align}\label{est:H}
 |H(\xi,t)| \leq (1 + |\xi|)^{1-\alpha}\mathcal{N}(t) + c(1 + |\xi|)^{p_N + 1 - \alpha}, \textnormal{ for some } \mathcal{N} \in L^1([0,T]).
\end{align}
Fix $r>0$, and set $\bar r := r + |\xi^n_0|$. Define 
\begin{align*}
 \mathcal{M}(t) := (1 + \bar r)^{1-\alpha}\mathcal{N}(t) + c(1 + \bar r)^{p_N + 1 - \alpha},
\end{align*}
where $c$ is as in \eqref{est:H}. Then $\mathcal{M}\in L^1([0,T])$ and 
\begin{align}\label{est:H-for-xi-in-ball}
 |H(\xi,t)| \leq \mathcal{M}(t) \textnormal{ for all } (\xi,t) \in \bar B(\xi^n_0,r) \times [0,T].
\end{align}
For $\delta \in (0,T)$ we define the set for functions
\begin{align*}
 E_\delta := \Big\{\gamma: [0,\delta]\to \bar B(\xi^n_0, r) \,|\, |\gamma(t_2) - \gamma(t_1)| \leq \int^{t_2}_{t_1} \mathcal{M}(t)\d t, \quad 0\leq t_1 < t_2 \leq \delta \Big\}.
\end{align*}
The set $E_\delta$ is evidently convex and closed $C([0,\delta])$. By the uniform equicontinuity of its members one can see using the Ascoli-Arzel\`a theorem that $E_\delta$ is also compact in $C([0,\delta])$. From \eqref{est:H-for-xi-in-ball} it follows that for sufficiently small $\delta$, the map 
\begin{align*}
 T: E_\delta \to E_\delta, \quad (T\gamma)(t) = \xi^n_0 + \int^t_0 H(\gamma(s),s) \d s,
\end{align*}
is well-defined. In order to see that $T$ is continuous with respect to the metric inherited from $C([0,\delta])$, suppose that $(\gamma^j)$ is a sequence in $E_\delta$ converging in the norm of $C([0,\delta])$ to $\gamma$. Especially the sequence converges pointwise. Note that 
\begin{align*}
 \norm{T\gamma^j - T\gamma}_{C([0,\delta])} = \sup_{t\in[0,\delta]} |T\gamma^j(t) - T\gamma(t)| \leq \int^t_0|H(\gamma^j(s),s) - H(\gamma(s),s)| \d s
 \\
 \leq \int^\delta_0|H(\gamma^j(s),s) - H(\gamma(s),s)| \d s.
\end{align*}
Note that the integrand in the last expression has the integrable majorant $2 \mathcal{M}$. Since $F$, $K$ and $G$ are continuous in $\xi$ for almost every $t$, the same is true for $H$ so by the pointwise convergence of $\gamma^j$ we may use the dominated convergence theorem to conclude that 
\begin{align*}
 \lim_{j\to\infty} \norm{T\gamma^j - T\gamma}_{C([0,\delta])} = 0,
\end{align*}
i.e. $T$ is continuous. Thus, by the Schauder fixed point theorem $T$ has a fixed point $\xi^n \in E_\delta$ which is therefore a solution to the integral equation  
\eqref{integ-eq}.

\vspace{3mm}
\noindent \textbf{Step 2: Extension and energy estimates for the approximating solutions.}
 In order to show that $\xi^n$ can be extended to all of $[0,T]$, it is sufficient to show that there is a fixed bounded set which contains the values of any solution to the system of ODEs. Note that for $a \in [0,\delta]$ we can write
\begin{align*}
 |\xi^n(a)|^2 = \int_\Omega |w_n(x,a)|^2 \d x &\leq \norm{w_n(\cdot, a)}_{L^\infty(\Omega)}^{1-\alpha} \int_\Omega |w_n(x,a)|^{\alpha +1 } \d x 
 \\
 &\leq c_n |\xi^n(a)|^{1-\alpha} \int_\Omega |u_n(x,a) - g(x,a)|^{\alpha +1 } \d x
 \\
 &\leq \tilde c_{n,\alpha} |\xi^n(a)|^{1-\alpha} \int_\Omega |u_n(x,a)|^{\alpha+1} + |g(x,a)|^{\alpha+1} \d x \d t,
\end{align*}
where the penultimate step follows from the fact that the functions $v_j$ are bounded. It follows that 
\begin{align}\label{est:xi-of-a}
 |\xi^n(a)| \leq c_{n,\alpha} \Big( \int_\Omega |u_n(x,a)|^{\alpha +1 } + |g(x,a)|^{\alpha+1} \d x \Big)^\frac{1}{\alpha+1},
\end{align}
for some constant $c_{n,\alpha}$. In order to find an upper bound for the integral on the right-hand side, we multiply the equation in \eqref{prob:finite-dim} with $\xi^n_m(t)$, sum over $m$ and integrate w.r.t. time from $0$ to $a$ to conclude
\begin{align}\label{eq-summed-integrated}
&\int^a_0 \int_\Omega \alpha (|u_n| + \varepsilon_n)^{\alpha - 1}u_n \partial_t u_n \d x \d t + \int^a_0\int_\Omega A(x, t, u_n, \nabla u_n) \cdot \nabla u_n \d x \d t 
\\
\notag &= \int^a_0 \int_\Omega \alpha (|u_n| + \varepsilon_n)^{\alpha - 1} \partial_t u_n g \d x \d t + \int^a_0\int_\Omega A(x, t, u_n, \nabla u_n) \cdot \nabla g \d x \d t
\\
\notag & \quad + \int^a_0 \int_\Omega f(u_n-g)\d x \d t.
\end{align}
Introduce the function
\begin{align*}
 \Gamma_n(u):= \alpha\int^u_0 (|s| + \varepsilon_n)^{\alpha -1}s \d s. 
\end{align*}
The boundedness of $u_n$ and the fact that $\partial_t u_n$ is integrable allow us to apply the chain rule for Sobolev functions and re-write the integrand in the first integral on the left-hand side of \eqref{eq-summed-integrated} as $\partial_t \Gamma_n(u_n)$. Therefore, using the fact that $g$ and hence also $u_n$ is time-continuous into $L^{\alpha+1}(\Omega)$ one can see that 
\begin{align}\label{gamma-der-integrated}
 \int^a_0 \int_\Omega \alpha (|u_n| + \varepsilon_n)^{\alpha - 1}u_n \partial_t u_n \d x \d t = \int_\Omega \Gamma_n(u_n(x,a)) \d x - \int_\Omega \Gamma_n(u_n(x,0)) \d x.
\end{align}
Combining \eqref{gamma-der-integrated} with \eqref{eq-summed-integrated} and estimating the second integral on the left-hand side of \eqref{eq-summed-integrated} from below by using the structure condition \eqref{cond:structure1} we end up with
\begin{align}\label{beetcheeuutchee}
 \notag \int_\Omega  \Gamma_n(u_n(x, a))  \d x &+  c \int^a_0 \int_\Omega \sum^N_{j=1}  |\partial_j u_n|^{p_j} \d x \d t   
  \leq \int_\Omega \Gamma_n(u_n(x, 0)) \d x + \int^a_0 \int_\Omega \tilde a \d x\d t 
\\
 & + \int^a_0\int_\Omega f(u_n - g)\d x \d t + \int^a_0 \int_\Omega \alpha (|u_n| + \varepsilon_n)^{\alpha - 1} \partial_t u_n g \d x \d t 
 \\
 \notag & + \sum^N_{j=1} \int^a_0 \int_\Omega |A_j(x,t,u_n,\nabla u_n)||\partial_j g| \d x \d t.
\end{align}
Using Young's inequality with $\varepsilon$ and the structure condition \eqref{cond:structure2} on the integrands in the sum on the last row, we have 
\begin{align*}
 |A_j(x,t,u_n,\nabla u_n)||\partial_j g| &\leq \varepsilon |A_j(x,t,u_n,\nabla u_n)|^{p_j'} + c_\varepsilon |\partial_j g|^{p_j} 
 \\
 &\leq c\varepsilon \Big( \sum^N_{k=1} |\partial_k u_n|^{p_k} + \tilde b \Big) + c_\varepsilon |\partial_j g|^{p_j}.
\end{align*}
Choosing $\varepsilon$ small, the previous estimate allows us to include all the terms involving spatial derivatives of $u_n$ in the corresponding term on the left-hand side of \eqref{beetcheeuutchee} to obtain 
\begin{align}\label{beetcheeuutchee2}
 &\int_\Omega \Gamma_n(u_n(x,a)) \d x + c \int^a_0 \int_\Omega \sum^N_{j=1} |\partial_j u_n|^{p_j} \d x \d t   
  \leq  \int_\Omega \Gamma_n(u_n(x, 0)) \d x
\\
\notag & + \int^a_0 \int_\Omega \alpha (|u_n| + \varepsilon_n)^{\alpha - 1} \partial_t u_n g \d x \d t + c\sum^N_{j=1} \int^a_0 \int_\Omega |\partial_j g|^{p_j} \d x \d t + c|\Omega|
\\
\notag & + \int^a_0\int_\Omega f(u_n - g)\d x \d t + \int^a_0 \int_\Omega \tilde a + \tilde b \d x\d t.
\end{align}
Since $\alpha \in (0,1)$ we can estimate
\begin{align}\label{est:Gamma-lower}
 \Gamma_n(u) \geq \alpha\int^u_0 (|u| + \varepsilon_n)^{\alpha -1}s \d s = \frac\alpha2 (|u| + \varepsilon_n)^{\alpha -1} u^2 \geq c\chi_{\{|u| \geq 1\}} |u|^{\alpha+1}.
\end{align}
Moreover, we have 
\begin{align}\label{est:Gamma-upper}
 \Gamma_n(u) &= \alpha\int^u_0 (|s| + \varepsilon_n)^{\alpha -1}s \d s = \alpha\int^u_0 (|s| + \varepsilon_n)^{\alpha} \frac{s}{|s| + \varepsilon_n} \d s \leq \alpha\int^{|u|}_0 (s + \varepsilon_n)^{\alpha} \d s 
 \\
 \notag &= \tfrac{\alpha}{\alpha + 1}\big( (|u| + \varepsilon_n)^{\alpha+1} - \varepsilon_n^{\alpha+1}\big)\leq c(|u|^{\alpha+1} + 1).
\end{align}
Using \eqref{est:Gamma-lower} we see that 
\begin{align}\label{Lapha_plus_one-at-a}
 \int_\Omega |u_n(x,a)|^{\alpha+1} \d x \leq |\Omega| + \int_{\Omega\cap \{|u_n|\geq 1\}}|u_n(x,a)|^{\alpha+1} \d x \leq |\Omega| + c\int_\Omega \Gamma_n(u_n(x,a)) \d x.
\end{align}
From \eqref{est:Gamma-upper} obtain the estimate
\begin{align}\label{Gamma-at-0}
  \int_\Omega \Gamma_n(u_n(x, 0)) \d x \leq c (\norm{u_n(\cdot, 0)}_{L^{\alpha+1}(\Omega)}^{\alpha+1}  + |\Omega|) \leq C(u_0) + c|\Omega| < \infty,
\end{align}
where we have used the fact that the sequence of the norms $\norm{u_n(\cdot, 0)}_{L^{\alpha+1}(\Omega)}$ remains bounded by a constant due to our choice of the coefficients $\xi^n_0$.
Combining \eqref{beetcheeuutchee2}, \eqref{Lapha_plus_one-at-a} and \eqref{Gamma-at-0} we end up with
\begin{align}\label{est:almost-energy}
 c\int_\Omega & |u_n(x,a)|^{\alpha+1} \d x + c\int^a_0 \int_\Omega \sum^N_{j=1} |\partial_j u_n|^{p_j} \d x \d t \leq C(u_0) + C|\Omega|
 \\
 \notag &+ \int^a_0 \int_\Omega \alpha (|u_n| + \varepsilon_n)^{\alpha - 1} \partial_t u_n g \d x \d t
  +  C\sum^N_{j=1} \int^a_0 \int_\Omega |\partial_j g|^{p_j} \d x \d t 
  \\
  \notag & +\int^a_0\int_\Omega f(u_n - g)\d x \d t + \int^a_0 \int_\Omega \tilde a + \tilde b \d x\d t.
\end{align}
In order to estimate the first integral on the right-hand side, we introduce the function 
\begin{align*}
 B_n(u) := \alpha\int^u_0 (|s| + \varepsilon_n)^{\alpha-1}\d s.
\end{align*}
Using the boundedness of $u_n$, the integrability of $\partial_t u_n$ and the time-continuity of $g$ one can justify the chain rule and integration by parts in the second term on the right-hand side of \eqref{est:almost-energy} as follows:
\begin{align}\label{derB-g-integrated}
 \notag \int^a_0 \int_\Omega \alpha (|u_n| + \varepsilon_n)^{\alpha - 1} \partial_t u_n g \d x \d t =& \int^a_0 \int_\Omega \partial_t B_n(u_n) g \d x \d t = \int_\Omega B_n(u_n(x,a))g(x,a)\d x 
 \\
 &- \int_\Omega B_n(u_n(x,0))g(x,0)\d x - \int^a_0 \int_\Omega B_n(u_n) \partial_t g \d x \d t.
\end{align}
In order to proceed we show some algebraic properties of $B_n$ which will also be used later in the proof. Note that for $\alpha \in (0,1)$ we have by Lemma \ref{lem:elementary_real} that
\begin{align}\label{B_n-ab-diff}
 |B_n(b) - B_n(a)| = \Big| \alpha \int^b_a (|s|+ \varepsilon_n)^{\alpha-1} \d s\Big| \leq \Big| \int^b_a \alpha |s|^{\alpha-1}\d s\Big| &= \big||b|^{\alpha-1}b - |a|^{\alpha-1}a\big|
 \\
 \notag &\leq c|b - a|^\alpha.
\end{align}
Especially, taking $b=0$ we have
\begin{align}\label{B_n-upperbound}
 |B_n(a)| \leq |a|^\alpha.
\end{align}
This estimate and Young's inequality show that
\begin{align}\label{shsh}
 \Big|\int_\Omega B_n(u_n(x,a))g(x,a)\d x \Big| &\leq \int_\Omega |u_n(x,a)|^\alpha |g(x,a)| \d x 
 \\
 \notag &\leq \varepsilon \int_\Omega |u_n(x,a)|^{\alpha+1} \d x + c_\varepsilon \int_\Omega |g(x,a)|^{\alpha+1}\d x,
\end{align}
and a similar estimate (without $\varepsilon$) holds for the integral involving $B_n(u_n(x,0))$. Similarly, we have
\begin{align}\label{tttt}
 \Big|\int^a_0 \int_\Omega B_n(u_n) \partial_t g \d x \d t\Big| &\leq \int^a_0 \int_\Omega |u_n|^\alpha |\partial_t g| \d x \d t 
 \\
 \notag &\leq \tilde \varepsilon \int^a_0 \int_\Omega |u_n|^{\alpha + 1} \d x \d t + c_{\tilde \varepsilon} \int^a_0 \int_\Omega  |\partial_t g|^{\alpha+1} \d x \d t
 \\
 \notag &\leq \tilde \varepsilon T \sup_{t\in[0,\delta]}\int_\Omega |u_n(x,t)|^{\alpha+1} \d x + c_{\tilde \varepsilon} \int^a_0 \int_\Omega  |\partial_t g|^{\alpha+1} \d x \d t.
\end{align}
Similarly, for the last term on the right-hand side of \eqref{est:almost-energy} we have by Young's inequality
\begin{align}\label{ung_olikhet}
 \notag \int^a_0\int_\Omega & f(u_n - g)\d x \d t \leq \int^a_0\int_\Omega |f||u_n| \d x \d t + \iint_{\Omega_T}|f g| \d x \d t 
 \\
 &\leq \tilde \varepsilon T \sup_{t\in[0,\delta]}\int_\Omega |u_n(x,t)|^{\alpha+1} \d x + c_{\tilde \varepsilon} \iint_{\Omega_T} |f|^\frac{\alpha+1}{\alpha}\d x \d t + \iint_{\Omega_T}|f g| \d x \d t 
 \\
 \notag & \leq \tilde \varepsilon T \sup_{t\in[0,\delta]}\int_\Omega |u_n(x,t)|^{\alpha+1} \d x + \hat c_{\tilde \varepsilon} \iint_{\Omega_T} |f|^\frac{\alpha+1}{\alpha}\d x \d t + T \sup_{t\in [0,T]}\int_\Omega|g(x,t)|^{\alpha+1} \d x
\end{align}
We can thus combine \eqref{shsh}, \eqref{tttt}, \eqref{ung_olikhet} and \eqref{derB-g-integrated} with \eqref{est:almost-energy} to conclude
\begin{align}\label{est:nearly-energy}
 c\int_\Omega & |u_n(x,a)|^{\alpha+1} \d x + c\int^a_0 \int_\Omega \sum^N_{j=1} |\partial_j u_n|^{p_j} \d x \d t \leq  C\sup_{t\in[0,T]} \int_\Omega |g(x,t)|^{\alpha+1}\d x
 \\
 \notag &+ 2\tilde \varepsilon T \sup_{t\in[0,\delta]}\int_\Omega |u_n(x,t)|^{\alpha+1} \d x + c_{\tilde \varepsilon} \int^a_0 \int_\Omega  |\partial_t g|^{\alpha+1} \d x \d t
  +  C\sum^N_{j=1} \int^a_0 \int_\Omega |\partial_j g|^{p_j} \d x \d t
  \\
  \notag & + c_{\tilde \varepsilon} \iint_{\Omega_T} |f|^\frac{\alpha+1}{\alpha}\d x \d t + C\iint_{\Omega_T} \tilde a + \tilde b \,\d x\d t + C(u_0) + C|\Omega|.
\end{align}
where we have chosen $\varepsilon$ so small that the first term on the last row of \eqref{shsh} can be absorbed to the left-hand side of \eqref{est:almost-energy}. Since both terms on the left-hand side are nonnegative, we can take the supremum over $a \in [0,\delta]$ individually to obtain
\begin{align*}
 c\sup_{a\in[0,\delta]}\int_\Omega & |u_n(x,a)|^{\alpha+1} \d x + c\int^\delta_0 \int_\Omega \sum^N_{j=1} |\partial_j u_n|^{p_j} \d x \d t \leq  C \sup_{t\in[0, T]} \int_\Omega |g(x,t)|^{\alpha+1}\d x
 \\
 \notag &+ 2 \tilde \varepsilon T \sup_{t\in[0,\delta]}\int_\Omega |u_n(x,t)|^{\alpha+1} \d x + c_{\tilde \varepsilon} \iint_{\Omega_T}  |\partial_t g|^{\alpha+1} \d x \d t
  +  C\sum^N_{j=1} \iint_{\Omega_T} |\partial_j g|^{p_j} \d x \d t
   \\
  \notag & + c_{\tilde \varepsilon} \iint_{\Omega_T} |f|^\frac{\alpha+1}{\alpha}\d x \d t + C\iint_{\Omega_T} \tilde a + \tilde b \,\d x\d t +C(u_0) + C|\Omega|.
\end{align*}
Taking $\tilde \varepsilon$ sufficiently small, the term where it appears as a factor can be absorbed into the left-hand side  so that 
\begin{align}\label{est:basically-energy}
 \notag c\sup_{a\in[0,\delta]}\int_\Omega & |u_n(x,a)|^{\alpha+1} \d x + c\int^\delta_0 \int_\Omega \sum^N_{j=1} |\partial_j u_n|^{p_j} \d x \d t \leq  C\sup_{t\in[0, T]} \int_\Omega |g(x,t)|^{\alpha+1}\d x
 \\
  &  + C \iint_{\Omega_T} |f|^\frac{\alpha+1}{\alpha} + |\partial_t g|^{\alpha+1}
  +  \sum^N_{j=1}  |\partial_j g|^{p_j} + \tilde a + \tilde b \,\d x\d t + C(u_0) + C|\Omega|.
\end{align}
As the second integral on the left-hand side is nonnegative we can combine \eqref{est:basically-energy} with \eqref{est:xi-of-a} to obtain
\begin{align*}
 |\xi(a)| \leq C(n, g, \partial_t g, \nabla g, u_0,f, |\Omega|) < \infty, \quad a \in [0,\delta].
\end{align*}
This is the desired boundedness estimate which confirms that $\xi^n$ can indeed be extended to $[0,T]$. Thus, we may take $\delta=T$ in \eqref{est:basically-energy} and write
\begin{align}\label{est:u_n-energy}
 \notag \sup_{t\in[0,T]} &\int_\Omega |u_n(x,t)|^{\alpha+1}\d x + \iint_{\Omega_T} \sum^N_{j=1} |\partial_j u_n|^{p_j} \d x \d t \leq C(u_0) + C|\Omega| + C\iint_{\Omega_T} |f|^\frac{\alpha+1}{\alpha}\d x \d t
 \\
  & + C\Big[\sup_{t\in[0, T]} \int_\Omega |g(x,t)|^{\alpha+1}\d x +  \iint_{\Omega_T}  |\partial_t g|^{\alpha+1} 
  +  \sum^N_{j=1}  |\partial_j g|^{p_j} + \tilde a + \tilde b \,\d x\d t\Big],
\end{align}
which is the desired energy estimate.

\vspace{3mm}
\noindent \textbf{Step 3: Weak and strong convergence of the approximating solutions.} 
The energy estimate \eqref{est:u_n-energy} shows that $(u_n)$ is a bounded sequence in $L^{\alpha+1}(\Omega_T)$ and that for each $j\in \{1,\dots, N\}$ the sequence $(\partial_j u_n)$ is bounded in $L^{p_j}(\Omega_T)$. Moreover, this fact combined with the structure condition \eqref{cond:structure2} shows that $(A_j(x, u_n, \nabla u_n))$ is a bounded sequence in $L^{p_j'}(\Omega_T)$ for each $j\in \{1,\dots, N\}$. Therefore, by passing to a subsequence (still labelled $u_n$) we may assume that
\begin{enumerate}
 \item The sequence $(u_n)^\infty_{n=1}$ converges weakly in $L^{\alpha+1}(\Omega_T)$ to a limit function $u$.
 \item For each $j\in \{1, \dots, N\}$ the sequence $(\partial_j u_n)^\infty_{n=1}$ converges weakly in $L^{p_j}(\Omega_T)$ to the weak derivative $\partial_j u$. 
 \item For each $j\in \{1, \dots, N\}$ the sequence $\big(A_j(x, t, u_n, \nabla u_n)\big)^\infty_{n=1}$ converges weakly in $L^{p_j'}(\Omega_T)$ to a limit function $\mathcal{A}_j$.
\end{enumerate}
We now intend to pass to yet another subsequence in order to obtain strong convergence of the sequence $(u_n)$ in $L^q(\Omega_T)$ for some $q>1$, and also pointwise convergence a.e. To do this, we follow the method of Laptev \cite{La} which is based on a compactness result proved by Sobolev, see Theorem 1.4.3 in \cite{So}. Examining the theorem, we see that since $\Omega_T$ is bounded and since $(u_n)$ is a bounded sequence in $L^{\alpha+1}(\Omega_T)$, it is sufficient to show that
\begin{align}\label{cond:compactness}
 \lim_{|(y,h)|\to 0} \sup_{n\in \N} \iint_{\Omega_T} |u_n(x+y, t+h) - u_n(x,t)|^q \d x \d t = 0,
\end{align}
where $u_n$ is extended as zero outside $\Omega_T$. It is sufficient to consider translations in time and space separately. 
To treat translations in time we first note that for any $a,b\in \R$, recalling that $\varepsilon_n<1$, we have
\begin{align}\label{est:B}
 (B_n(b) - B_n(a))(b-a) &= \int^1_0 \frac{\d }{\d s}B_n(a + s(b-a)) \d s (b-a) 
 \\
 \notag &= \int^1_0 \alpha (|a+s(b-a)|+\varepsilon_n)^{\alpha-1} \d s (b-a)^2 
 \\
 \notag &\geq c(|a| + |b| + 1)^{\alpha-1}(b-a)^2.
\end{align}
Thus, by using H\"older's inequality with the exponents $2/(\alpha+1)$ and $2/(1-\alpha)$ and setting $\nu= (1-\alpha)(1+\alpha)/2$ we have 
\begin{align}\label{est:Lalphaplusone-diff}
 \int_\Omega & |u_n(x,t+h) - u_n(x,t)|^{\alpha+1} \d x 
 \\
\notag  &= \int_\Omega (|u_n(x,t+h)| + |u_n(x,t)| + 1)^\nu 
 \\
 \notag & \qquad \times (|u_n(x,t+h)| + |u_n(x,t)| + 1)^{-\nu} |u_n(x,t+h) - u_n(x,t)|^{\alpha+1} \d x 
 \\
 \notag &\leq  \Big[ \int_\Omega (|u_n(x,t+h)| + |u_n(x,t)| + 1)^{\alpha-1} |u_n(x,t+h) - u_n(x,t)|^2  \d x\Big]^\frac{\alpha+1}{2}
 \\
 \notag & \quad \times \Big[ \int_\Omega (|u_n(x,t+h)| + |u_n(x,t)| + 1)^{\alpha+1} \d x\Big]^\frac{1-\alpha}{2} 
 \\
 \notag & \leq c\Big[ \int_\Omega \big(B_n(u_n(x,t+h)) - B_n(u_n(x,t))\big) \big(u_n(x,t+h) - u_n(x,t)\big)  \d x\Big]^\frac{\alpha+1}{2}
 \\
 \notag &= c\Big[ \int_\Omega \big(B_n(u_n(x,t+h)) - B_n(u_n(x,t))\big) \big(w_n(x,t+h) - w_n(x,t)\big)  \d x
\\
\notag & \qquad \qquad + \int_\Omega \big(B_n(u_n(x,t+h)) - B_n(u_n(x,t))\big) \big(g(x,t+h) - g(x,t) \d x\Big]^\frac{\alpha+1}{2}
\end{align}
where in the penultimate step we have used \eqref{est:B} and the boundedness of $u_n$ on $[0,T]$ with respect to the norm of $L^{\alpha+1}(\Omega)$, which follows from \eqref{est:u_n-energy}. In order to treat the first integral inside the square brackets, we note that by \eqref{prob:finite-dim} we have for $v\in V_n$ and $t\in (0, T-h)$ that
\begin{align*}
\int_\Omega \big(B_n(u_n(x,t+h)) - &B_n(u_n(x,t))\big) v(x)\d x = \int_\Omega \int^h_0 \frac{\d }{\d s}B(u_n(x, t + s)) \d s v(x) \d x
 \\
 & =  \int^h_0 \alpha \int_\Omega (|u_n(x, t+s)| + \varepsilon_n)^{\alpha-1}\partial_t u_n(x, t+s) v(x) \d x \d s
 \\
 & = - \int^h_0 \int_\Omega A(x, t + s, u_n(x,t+s), \nabla u_n(x, t + s)) \cdot \nabla v(x)\d x \d s
 \\
 & \quad + \int^h_0 \int_\Omega f(x, t + s)v(x) \d x.
\end{align*}
Especially, the last equation holds for the choice $v= w_n(x,t+h) - w_n(x,t)$.
This allows us to combine the last equation with \eqref{est:Lalphaplusone-diff} to conclude that
\begin{align}\label{aaa}
\notag \int^{T-h}_0 \int_\Omega &|u_n(x ,t+h)  - u_n(x,t)|^{\alpha+1} \d x \d t \leq 
\\
\notag c \Big[\int^h_0 \hspace{-1mm}\iint_{\Omega_{T-h}} \hspace{-3.5mm}&|A(x, t + s, u_n(x,t+s), \nabla u_n(x, t + s)) \cdot (\nabla w_n(x,t+h) - \nabla w_n(x,t))|\d x \d t \d s\Big]^\frac{\alpha+1}{2}
\\
& + c\int^{T-h}_0\Big[\int_\Omega \big|[B_n(u_n(x,t+h)) - B_n(u_n(x,t))] [g(x,t+h) - g(x,t)]\big| \d x \Big]^\frac{\alpha+1}{2}\d t
\\
\notag & + c \Big[\int^h_0 \iint_{\Omega_{T-h}} |f(x,t+s)||w_n(x,t+h) - w_n(x,t)|\d x \d t    \Big]^\frac{\alpha+1}{2}.
\end{align}
We need to prove that the three integrals on the right-hand side converge to zero as $h\to 0$ uniformly in $n$. To treat the first integral it is sufficient to note that the sequence of partial derivatives $(\partial_j w_n)$ and the functions $(A_j(x, t, u_n, \nabla u_n))$ are bounded in H\"older spaces of matching exponents which means that the inner integral is bounded by a constant independent of $n$. The integral over $[0,h]$ produces a factor of $h$ so that all in all,
\begin{align*}
 \int^h_0 \hspace{-1mm}\iint_{\Omega_{T-h}} \hspace{-3mm}|A(x, t + s, u_n(x,t+s), \nabla u_n(x, t + s)) \cdot (\nabla w_n(x,t+h) - \nabla w_n(x,t))|\d x \d t \d s \leq c h,
\end{align*}
where $C$ is independent of $n$. Similarly, due to matching H\"older exponents in the last integral on the right-hand side and the boundedness of the sequence $(w_n)$ in $L^{\alpha+1}(\Omega_T)$ we have
\begin{align*}
 \int^h_0 \iint_{\Omega_{T-h}} |f(x,t+s)||w_n(x,t+h) - w_n(x,t)|\d x \d t \leq ch.
\end{align*}
To treat the second integral on the right-hand side of \eqref{aaa} we use \eqref{B_n-upperbound} and H\"older's inequality as follows:
\begin{align*}
 &\int^{T-h}_0 \Big[\int_\Omega \big|[B_n(u_n(x,t+h)) - B_n(u_n(x,t))] [g(x,t+h) - g(x,t)]\big| \d x \Big]^\frac{\alpha+1}{2}\d t
 \\
 &\leq c \Big[ \iint_{\Omega_{T-h}}\hspace{-2mm} |u_n(x,t+h)|^{\alpha+1} + |u_n(x,t)|^{\alpha+1} \d x \d t \Big]^\frac{\alpha}{2}   \Big[ \iint_{\Omega_{T-h}}\hspace{-2mm} |g(x,t+h) - g(x,t)|^{\alpha+1} \d x \d t \Big]^\frac12
 \\
 &\leq C  \Big[ \iint_{\Omega_{T-h}} |g(x,t+h) - g(x,t)|^{\alpha+1} \d x \d t \Big]^\frac12 \xrightarrow[h\to 0]{} 0.
\end{align*}
where the penultimate step follows from the fact that $(u_n)$ is a bounded sequence in $L^{\alpha+1}(\Omega_T)$, so that the constant $C$ is independent of $n$. The last step follows from the fact that $g$ is in $C([0,T];L^{\alpha+1}(\Omega))$ and since $[0,T]$ is compact, $g$ is also uniformly continuous $[0,T]\to L^{\alpha+1}(\Omega)$. Thus, we have confirmed that 
\begin{align}\label{time-translation}
 \lim_{h\to 0} \sup_{n\in \N} \int^{T-h}_0 \int_\Omega |u_n(x ,t+h)  - u_n(x,t)|^{\alpha+1} \d x \d t = 0.
\end{align}
For translations in space, it is sufficient to consider each coordinate direction separately. Recalling that $w_n$ vanishes outside $\Omega_T$ we can see that extending $u_n$ as zero outside of $\Omega_T$ means extending $g$ as zero outside $\Omega_T$ so that for all $(x,t)\in \Omega_T$, 
\begin{align*}
|u_n(x+ h e_j,t) - u_n(x,t)|^{p_j} \leq c|g(x+h e_j,t) - g(x,t)|^{p_j} + c|w_n(x+h e_j,t) - w_n(x,t)|^{p_j}.  
\end{align*}
By the regularity of the functions $v_j$ we have that we can use the classical fundamental theorem of calculus to estimate
\begin{align*}
 |w_n(x+h e_j,t) - w_n(x,t)|^{p_j} = \Big| \int^h_0 \partial_j w_n(x + se_j, t)\d s\Big|^{p_j} \hspace{-0.5mm} \leq h^{p_j-1} \hspace{-1mm}\int^h_0 |\partial_j w_n(x, t+ s e_j)|^{p_j}\d s.
\end{align*}
Integrating over $\Omega_T$ we have 
\begin{align*}
 \iint_{\Omega_T} |w_n(x+h e_j,t) - w_n(x,t)|^{p_j} \d x \d t \leq h^{p_j-1}\int^h_0 \iint_{\Omega_T}|\partial_j w_n(x, t+ s e_j)|^{p_j}\d x \d t \d s \leq ch^{p_j},
\end{align*}
where the last step follows from the fact that for every $s \in [0,h]$ the inner integral can be bounded by the $L^{p_j}(\Omega_T)$-norm of $\partial_j w_n$, which is bounded by a constant independent of $n$ since $\partial_j w_n = \partial_j u_n - \partial_j g$ and since $(\partial_j u_n)$ is a bounded sequence in $L^{p_j}(\Omega_T)$. Thus, we have seen that
\begin{align*}
 \iint_{\Omega_T}|u_n(x+ h e_j,t) - u_n(x,t)|^{p_j}\d x \d t \leq \iint_{\Omega_T} |g(x+h e_j,t) - g(x,t)|^{p_j} \d x \d t + ch^{p_j}.
\end{align*}
Since $g$ is bounded and especially $L^{p_j}$-integrable, we have 
\begin{align*}
 \lim_{h\to 0} \iint_{\Omega_T} |g(x+h e_j,t) - g(x,t)|^{p_j} \d x \d t = 0,
\end{align*}
so the last two estimates confirm that
\begin{align}\label{space-translation}
 \lim_{h\to 0} \sup_{n\in \N} \iint_{\Omega_T} |u_n(x + h e_j ,t)  - u_n(x,t)|^{p_j} \d x \d t = 0.
\end{align}
Combining \eqref{time-translation} and \eqref{space-translation} we have \eqref{cond:compactness} with $q= \min\{\alpha+1, p_j\}$, which means that a subsequence, still labelled $(u_n)$, converges strongly in $L^q(\Omega_T)$ and also pointwise a.e. after passing to yet another subsequence.

\vspace{3mm}
\noindent \textbf{Step 4: Equations satisfied by the limit function.}
For $\varphi \in C^\infty_\textnormal{o}((0,T))$ and $v \in \cup^\infty_{n=1} V_n$ we  have by \eqref{prob:finite-dim} for sufficiently large $n$ that 
\begin{align*}
 \iint_{\Omega_T} \alpha (|u_n| + \varepsilon_n)^{\alpha - 1}\partial_t u_n \varphi(t) v(x) \d x \d t + \iint_{\Omega_T} A(x, t, u_n, \nabla u_n) \cdot \nabla v(x) \varphi(t) \d x \d t
 \\
 = \iint_{\Omega_T} f(x,t) v(x)\varphi(t) \d x \d t.
\end{align*}
Using integration by parts in the first integral, we end up with
\begin{align}\label{chemistryclass}
 -\iint_{\Omega_T} B_n(u_n(x,t))  \varphi'(t) v(x) \d x \d t + \iint_{\Omega_T} A(x, t, u_n, \nabla u_n) \cdot \nabla v(x) \varphi(t) \d x \d t 
 \\
 \notag = \iint_{\Omega_T} f(x,t) v(x)\varphi(t) \d x \d t.
\end{align}
We intend to pass to the limit $n \to \infty$. In order to treat the first integral we estimate
\begin{align}\label{B-diff1}
|B_n(b) - |a|^{\alpha-1}a| \leq |B_n(b) - B_n(a)| + |B_n(a) - |a|^{\alpha-1}a|.
\end{align}
The first expression on the right-hand side can be estimated upwards using \eqref{B_n-ab-diff} whereas the second term can be estimated as follows:
\begin{align}\label{B_n-alpha-diff}
 |B_n(a) - |a|^{\alpha-1}a| &= | (|a| + \varepsilon_n)^\alpha \sgn(a) - \varepsilon_n^\alpha - |a|^{\alpha-1}a| 
 \\
 \notag &\leq |(|a| + \varepsilon_n)^\alpha \sgn(a) - |a|^{\alpha-1}a| + \varepsilon_n^\alpha
 \\
 \notag &= |(|a| + \varepsilon_n)^\alpha - |a|^\alpha| + \varepsilon_n^\alpha
 \\
 \notag &\leq 2 \varepsilon_n^\alpha.
\end{align}
where in the last step we have used Lemma \ref{lem:elementary_real} with $\gamma=1/\alpha > 1$. Thus, we end up with
\begin{align}\label{est:B2}
 |B_n(b) - |a|^{\alpha-1}a| \leq |b-a|^\alpha + 2 \varepsilon^\alpha_n.
\end{align}
By \eqref{B_n-upperbound} we have that $(B_n(u_n))$ is a bounded sequence in $L^\frac{\alpha+1}{\alpha}(\Omega_T)$ so by passing to a subsequence we may assume that it converges weakly in this space to some limit function $\mathcal{B} \in L^\frac{\alpha+1}{\alpha}(\Omega_T)$. The pointwise convergence of $u_n$ to $u$ and \eqref{est:B2} imply that $(B_n(u_n))$ converges pointwise to $|u|^{\alpha-1}u$, and thus we have $\mathcal{B} = |u|^{\alpha-1}u$. 
Furthermore, each function $A_j(x,t,u_n,\nabla u_n)$ converges weakly to $\mathcal{A}_j$ so passing to the limit $n\to \infty$ in \eqref{chemistryclass} we obtain
\begin{align} \label{half-limit}
 -\iint_{\Omega_T} |u|^{\alpha-1}u(x,t) \varphi'(t) v(x) \d x \d t + \iint_{\Omega_T} \mathcal{A} \cdot \nabla v(x) \varphi(t) \d x \d t
 = \iint_{\Omega_T} f(x,t) v(x)\varphi(t) \d x \d t.
\end{align}
By density, we see that the previous equation holds for all $v\in V$, and due to the properties of the sequence $(v_j)^\infty_{j=1}$ the equation holds also for all $v\in C^\infty_{\textnormal{o}}(\Omega)$. We now show that \eqref{half-limit} implies that $u$ satisfies a corresponding result with test functions in $C^\infty_{\textnormal{o}}(\Omega_T)$. For this purpose, let $\eta \in C^\infty_{\textnormal{o}}((-1,1))$ be a bump function used in the construction of $1$-dimensional mollifiers and set
\begin{align*}
 u^\alpha_\varepsilon(x,t):= \varepsilon^{-1}\int^{\varepsilon}_{-\varepsilon} |u|^{\alpha-1}u(x,t-s)\eta(\varepsilon^{-1}s)\d s,
\end{align*}
that is, the convolution is taken only w.r.t. the time variable. This function is well-defined on $\Omega\times(\varepsilon, T-\varepsilon)$. Hence, given $\psi \in C^\infty_{\textnormal{o}}(\Omega_T)$ we have 
\begin{align*}
 \iint_{\Omega_T} |u|^{\alpha-1}u(x,t) \partial_t \psi(x,t)\d x \d t = \lim_{\varepsilon\to 0} \iint_{\Omega_T} u^\alpha_\varepsilon(x,t) \partial_t \psi(x,t) \d x \d t.
\end{align*}
where the integral on the right-hand side is well-defined for small $\varepsilon$. Using Fubini's theorem and classical integration by parts we may re-write the integral as
\begin{align*}
 \iint_{\Omega_T} u^\alpha_\varepsilon(x,t) \partial_t \psi(x,t) \d x \d t &= - \int_\Omega\int^T_0 |u|^{\alpha-1}u (x, \cdot)*\eta_\varepsilon'(t) \psi(x,t)\d t \d x
 \\
 &= -\int^T_0 \int_\Omega \int^\varepsilon_{-\varepsilon} |u|^{\alpha-1}u(x,t-s)\eta_\varepsilon'(s) \psi(x,t)\d s \d x \d t
 \\
&= - \int^T_0 \int_\Omega \int^{t+\varepsilon}_{t-\varepsilon} |u|^{\alpha-1}u(x,s)\eta_\varepsilon'(t-s) \psi(x,t)\d s \d x \d t
\\
&=  \int^T_0 \int_\Omega \int^{t+\varepsilon}_{t-\varepsilon} \mathcal{A}(x,s)\cdot \nabla \psi(x,t) \eta_\varepsilon (t-s) \d s \d x \d t
\\
& \quad - \int^T_0 \int_\Omega \int^{t+\varepsilon}_{t-\varepsilon} f(x,s) \eta_\varepsilon (t-s) \psi(x,t) \d s \d x \d t
\\
&\xrightarrow[\varepsilon\to 0]{} \iint_{\Omega_T} \mathcal{A}(x,t)\cdot \nabla \psi(x,t)\d x \d t - \iint_{\Omega_T}f(x,t) \psi(x,t)\d x \d t,
\end{align*}
where in the second last step we use \eqref{half-limit} with the choice $v(x) = \psi(x,t)$. Combining the last two equations we have that 
\begin{align}\label{eq-with-mathcalA-on-OmegaT}
 -\iint_{\Omega_T} |u|^{\alpha-1}u \partial_t \psi \d x \d t + \iint_{\Omega_T} \mathcal{A}\cdot \nabla \psi \d x \d t = \iint_{\Omega_T}f \psi\d x \d t,
\end{align}
for all $\psi \in C^\infty_{\textnormal{o}}(\Omega_T)$. 

\vspace{3mm}
\noindent \textbf{Step 5: Boundary values, time continuity and initial values.} Previously we obtained the function $u$ as the weak limit in $L^{\bf p}(0,T; \overline W^{1, {\bf p}}(\Omega))$ of a sequence $(u_n) $ contained in the convex closed set $g + L^{\bf p}(0,T; \overline W^{1, {\bf p}}_{\textnormal{o}}(\Omega))$. Since for a convex set the closure and the weak closure coincide, this means that $u\in g + L^{\bf p}(0,T; \overline W^{1, {\bf p}}_{\textnormal{o}}(\Omega))$, i.e. $u$ satisfies the correct boundary conditions.

In order to see that $u$ has a representative in $C([0,T];L^{\alpha+1}(\Omega))$, it suffices to note that $u$, $g$, $\mathcal{A}$ and $f$ satisfy the assumptions of Lemma \ref{lem:time-cont}. We are now ready to show that the time-continuous representative of $u$ satisfies $u(0)=u_0$, i.e. $u$ satisfies the correct initial value. Take a function $\zeta \in C^\infty([0,T])$ with support compactly contained in $[0,T)$ and which takes the value $1$ in a neighborhood of $0$. Take $v\in \cup^\infty_{n=1} V_n$ and use \eqref{eq-with-mathcalA-on-OmegaT} with the test function 
\begin{align}\label{test-funct-for-init-val}
\psi(x,t) = \mathcal{H}_\delta(t)\zeta(t) v(x),  
\end{align}
where $\mathcal{H}_\delta$ is defined as in \eqref{def:H_delta}. For small $\delta>0$ this leads to
\begin{align*}
 \delta^{-1} &\int^\delta_0 \int_\Omega |u|^{\alpha-1}u(x,t) v(x) \d x \d t = \iint_{\Omega_T} \mathcal{A}(x,t)\cdot \nabla v(x) \mathcal{H}_\delta(t) \zeta(t) \d x \d t 
 \\
 &\quad - \iint_{\Omega_T} |u|^{\alpha-1}u(x,t) \mathcal{H}_\delta(t) \zeta'(t) v(x) \d x \d t
 - \iint_{\Omega_T} f(x,t) \mathcal{H}_\delta(t)\zeta(t) v(x) \d x \d t.
\end{align*}
Due to the time-continuity of $u$ we may pass to the limit $\delta \to 0$ and conclude that
\begin{align}\label{hehu1}
 \int_\Omega |u|^{\alpha-1}u(x,0) v(x) \d x \d t &= \iint_{\Omega_T} \mathcal{A}(x,t)\cdot \nabla v(x)  \zeta(t) \d x \d t 
 \\
 \notag &\quad - \iint_{\Omega_T} |u|^{\alpha-1}u(x,t) \zeta'(t) v(x) \d x \d t - \iint_{\Omega_T} f(x,t) \zeta(t)v(x)\d x \d t.
\end{align}
Similarly, using \eqref{chemistryclass} with the choice $\varphi = \mathcal{H}_\delta \zeta$ we obtain
\begin{align*}
 \delta^{-1} &\int^\delta_0 \int_\Omega B_n(u_n(x,t)) v(x) \d x \d t = \iint_{\Omega_T} A(x,t,u_n,\nabla u_n)\cdot \nabla v(x) \mathcal{H}_\delta(t) \zeta(t) \d x \d t 
 \\
 & - \iint_{\Omega_T} B_n(u_n(x,t)) \mathcal{H}_\delta(t) \zeta'(t) v(x) \d x \d t - \iint_{\Omega_T} f(x,t) v(x) \mathcal{H}_\delta(t) \zeta(t) \d x \d t.
\end{align*}
Since $u_n \in C([0,T];L^{\alpha+1}(\Omega))$, passing to the limit $\delta\to 0$ poses no problem, and we end up with
\begin{align*}
 \int_\Omega B_n(u_n(x,0)) v(x) \d x &= \iint_{\Omega_T} A(x,t,u_n,\nabla u_n)\cdot \nabla v(x)  \zeta(t) \d x \d t
 \\
 &\quad - \iint_{\Omega_T} B_n(u_n(x,t)) \zeta'(t) v(x) \d x \d t - \iint_{\Omega_T} f(x,t) v(x) \zeta(t) \d x \d t.
\end{align*}
We now intend to pass to the limit $n\to \infty$. Due to \eqref{est:B2} and the strong convergence of $u_n(x,0)$ to $u_0$ in $L^{\alpha+1}(\Omega)$, the left-hand side poses no problem. Similarly, the weak convergence of $B_n(u_n)$ to $|u|^{\alpha-1}u$ in $L^\frac{\alpha+1}{\alpha}(\Omega_T)$ allows us to pass to the limit in the second integral on the right-hand side. Finally, taking into account also the weak convergence of $A(x,u_n,\nabla u_n)$ to $\mathcal{A}$ we have by passing to $n\to\infty$ that 
\begin{align*}
 \int_\Omega |u_0|^{\alpha-1}u_0 v \d x &= \iint_{\Omega_T} \mathcal{A}(x,t) \cdot \nabla v(x)  \zeta(t) \d x \d t
 - \iint_{\Omega_T} |u|^{\alpha-1}u(x,t)\zeta'(t) v(x) \d x \d t
 \\
 &\quad - \iint_{\Omega_T} f(x,t) \zeta(t) v(x)\d x \d t.
\end{align*}
Comparing this with \eqref{hehu1} we conclude that 
\begin{align}\label{integs-equal}
 \int_\Omega |u|^{\alpha-1}u(x,0) v(x) \d x = \int_\Omega |u_0|^{\alpha-1}u_0 v \d x,
\end{align}
for all $v\in \cup^\infty_{n=1} V_n$, and by density arguments \eqref{integs-equal} holds for all $v\in C^\infty_{\textnormal{o}}(\Omega)$ and even for all $v\in L^{\alpha+1}(\Omega)$ so that in fact $u(x,0) = u_0$, i.e. $u$ satisfies the correct initial value.

\vspace{3mm}
\noindent \textbf{Step 6: Identification of the weak limit.}
To complete the proof, in light of \eqref{eq-with-mathcalA-on-OmegaT} it only remains to show that $\mathcal{A} = A(x,u,\nabla u)$. For this identification, we proceed through the celebrated Minty's trick, using the monotonicity condition \eqref{cond:monot}. Indeed, this assumption implies that for all $v \in L^{{\bf p}}(0,T; W^{1,\bf{p}}(\Omega))$ the number
\begin{equation}\label{positive xn}
    X_n:= \iint_{\Omega_T} \big( A(x,t,u_n, \nabla u_n)- A(x,t, u_n, \nabla v)\big) \cdot \big( \nabla u_n- \nabla v \big) \d x \d t
\end{equation} is non-negative, and it factorizes as
\begin{equation}\label{est:X_n}
    \begin{aligned}
        0 \leq X_n &= \iint_{\Omega_T} A(x, t, u_n, \nabla u_n) \cdot \nabla( u_n - g) \d x \d t 
        \\
        &  - \iint_{\Omega_T} A(x, t, u_n, \nabla u_n) \cdot \nabla (v - g)\, \d x \d t - \iint_{\Omega_T} A(x, t, u_n,\nabla v) \cdot ( \nabla u_n-\nabla v) \d x \d t
        \\
        &= I_1(n)-I_2(n)-I_3(n).
    \end{aligned}
\end{equation} We estimate each member of the right-hand side as $n$ grows. By \eqref{eq-summed-integrated}, \eqref{gamma-der-integrated} and \eqref{derB-g-integrated} we have  
\begin{equation}\label{calc:I_1}
\begin{aligned}
 I_1(n) &= \iint_{\Omega_T} \alpha (|u_n| + \varepsilon_n)^{\alpha - 1} \partial_t u_n g \d x \d t - \iint_{\Omega_T} \alpha (|u_n| + \varepsilon_n)^{\alpha - 1}u_n \partial_t u_n \d x \d t
   \\
   &= \int_{\Omega} \Gamma_n (u_n(x,0)) \d x - \int_{\Omega} \Gamma_n (u_n(x,T)) \d x + \int_\Omega B_n(u_n(x,T))g(x,T)\d x 
 \\
 &\quad - \int_\Omega B_n(u_n(x,0))g(x,0)\d x - \iint_{\Omega_T} B_n(u_n) \partial_t g \d x \d t + \iint_{\Omega_T} f(u_n-g)\d x \d t.
   \end{aligned}
\end{equation}
In order to analyze the behavior of the integrals involving $\Gamma_n$ as $n\to \infty$, we note that for any $u \in \R$ we can write
\begin{align*}
 D_n(u) &:= \Gamma_n(u) - \frac{\alpha}{\alpha+1}|u|^{\alpha+1} = \alpha \int^u_0 (|s| + \varepsilon_n)^{\alpha-1}s - |s|^{\alpha-1}s \d s 
 \\
 &= \alpha \int_{[0,u]\cap [-\varepsilon_n, \varepsilon_n]} \hspace{-3mm}(|s| + \varepsilon_n)^{\alpha-1}s - |s|^{\alpha-1}s \d s + \alpha \int_{[0,u]\setminus [-\varepsilon_n, \varepsilon_n]}\hspace{-3mm} s[ (|s| + \varepsilon_n)^{\alpha-1} - |s|^{\alpha-1}] \d s
 \\
 &=: D^1_n(u) + D^2_n(u),
\end{align*}
where we remark that $[0,u]$ should be replaced by $[u,0]$ (along with a change of sign) in the case that $u<0$. One immediately sees that 
\begin{align*}
 |D^1_n(u)| \leq c\varepsilon_n^{\alpha+1}.
\end{align*}
In order to estimate $D^2_n(u)$ we simply use the mean value theorem to conclude that for some $t_n^s\in (0,1)$ we have 
\begin{align*}
|D^2_n(u)| \leq  \alpha \int_{[0,u]\setminus [-\varepsilon_n, \varepsilon_n]} |s|(1-\alpha) (|s| + t_n^s \varepsilon_n)^{\alpha-2}\varepsilon_n \d s &\leq c \int_{[0,u]\setminus [-\varepsilon_n, \varepsilon_n]} |s|^{\alpha-1} \varepsilon_n \d s 
\\
&\leq c\int_{[0,u]\setminus [-\varepsilon_n, \varepsilon_n]} \varepsilon_n^\alpha \d s \leq c \varepsilon^\alpha_n |u|.
\end{align*}
Hence, for $a\in \{0,T\}$ we have 
\begin{align}\label{Gamma_n-D_n-stuff}
 \int_{\Omega} \Gamma_n (u_n(x,a)) \d x = \int_\Omega D_n(u_n(x,a)) \d x + \frac{\alpha}{\alpha+1}\int_\Omega |u_n(x,a)|^{\alpha+1}\d x.
\end{align}
By the previous estimates for $D^1_n(u)$ and $D^2_n(u)$ we conclude that 
\begin{align}\label{lim:D_n-int-abs}
 \Big| \int_\Omega D_n(u_n(x,a)) \d x \Big| &\leq c \int_\Omega \varepsilon^{\alpha+1}_n  + \varepsilon_n^\alpha |u_n(x,a)| \d x 
 \\
\notag &\leq c|\Omega| \varepsilon^{\alpha+1}_n + c\varepsilon_n^{\alpha}|\Omega|^\frac{\alpha}{\alpha+1}\norm{u_n(\cdot, a)}_{L^{\alpha+1}(\Omega)}
 \\
\notag &\xrightarrow[n\to\infty]{} 0,
\end{align}
where in calculating the limit we use the fact that the $L^{\alpha+1}$-norm of $u_n(\cdot, a)$ stays bounded from above by a constant due to \eqref{est:almost-energy}. In order to treat the second integral in \eqref{Gamma_n-D_n-stuff} we need to distinguish between the two values of $a\in\{0,T\}$. For $a=0$ we note that $u_n(\cdot,0) = u_0^n$ converges to $u_0$ strongly in $L^2(\Omega)$, which implies strong convergence also in $L^{\alpha+1}(\Omega)$. From this it follows immediately that 
\begin{align}\label{a=0-strongnormlimit}
 \lim_{n\to\infty} \frac{\alpha}{\alpha+1}\int_\Omega |u_n(x,0)|^{\alpha+1}\d x = \frac{\alpha}{\alpha+1}\int_\Omega |u_0|^{\alpha+1}\d x.
\end{align}
Combining \eqref{a=0-strongnormlimit}, \eqref{lim:D_n-int-abs} and \eqref{Gamma_n-D_n-stuff}, we thus have that 
\begin{align}\label{lim-Gamma_n-u_n}
 \lim_{n\to\infty} \int_{\Omega} \Gamma_n (u_n(x,0)) \d x = \frac{\alpha}{\alpha+1}\int_\Omega |u_0|^{\alpha+1}\d x.
\end{align}
For $a=T$ on the other hand we use the fact that by the energy estimate \eqref{est:u_n-energy}, the sequence $(|u_n|^{\alpha-1}u_n(\cdot,T))$ is bounded in $L^\frac{\alpha+1}{\alpha}(\Omega)$, which means that a subsequence converges weakly in $L^\frac{\alpha+1}{\alpha}(\Omega)$ to a limit function $w$. By the weak lower semicontinuity of norms we thus have 
\begin{align*}
 \liminf_{n\to\infty} \int_\Omega |u_n(x,T)|^{\alpha+1}\d x = \liminf_{n\to\infty} \int_\Omega ||u_n|^{\alpha-1}u_n(x,T)|^\frac{\alpha+1}{\alpha}\d x\geq \int_\Omega |w|^{\alpha+1} \d x.
\end{align*}
Combining this fact with \eqref{lim:D_n-int-abs} and \eqref{Gamma_n-D_n-stuff} we have
\begin{align}\label{liminf-Gamma_n-T}
 \liminf_{n\to \infty} \int_{\Omega} \Gamma_n (u_n(x,T)) \d x \geq \frac{\alpha}{\alpha+1}\int_\Omega |w|^\frac{\alpha+1}{\alpha} \d x.
\end{align}
We now intend to show that in fact $w=|u|^{\alpha-1}u(\cdot,T)$ when we use the time-continuous representative of $u$. First note that by using \eqref{chemistryclass} with $\varphi$ chosen as in \eqref{func:trapets} with $t_1=0$ and $t_2=T$ and passing to the limit, we see that for all $v \in \cup^\infty_{n=1} V_n$,  
\begin{align*}
 \int_{\Omega} B_n(u_n(x,T))  v(x) \d x &= \int_{\Omega} B_n(u_n(x,0))  v(x) \d x - \iint_{\Omega_T} A(x, u_n, \nabla u_n) \cdot \nabla v(x) \d x \d t
 \\
 & \quad + \iint_{\Omega_T} f(x,t)v(x) \d x \d t,
\end{align*}
and thus
\begin{align}\label{cleveland}
 \notag \int_\Omega |u_n(x,T)|^{\alpha-1}u_n(x,T) v(x) \d x = \int_\Omega [|u_n(x,T)|^{\alpha-1}u_n(x,T) - B_n(u_n(x,T))] v(x) \d x 
 \\
 + \int_{\Omega} B_n(u_n(x,0))  v(x) \d x - \iint_{\Omega_T} A(x, u_n, \nabla u_n) \cdot \nabla v(x) \d x \d t + \iint_{\Omega_T} f(x,t)v(x) \d x \d t.
\end{align}
By \eqref{B_n-alpha-diff}, the first integral on the right-hand side vanishes in the limit $n\to \infty$. Similarly,
\begin{align*}
 \int_{\Omega} B_n(u_n(x,0))  v(x) \d x & = \int_{\Omega} [B_n(u_n(x,0)) - |u_n(x,0)|^{\alpha-1}u_n(x,0)]  v(x) \d x 
 \\
 & \hphantom{=} + \int_\Omega |u_n(x,0)|^{\alpha-1}u_n(x,0) v(x) \d x
 \\
 &\xrightarrow[n\to\infty]{} \int_\Omega |u_0|^{\alpha-1}u_0 v \d x.
\end{align*}
Combining these observations with \eqref{cleveland}, and utilizing also the weak convergence of the sequence $(A(x,u_n,\nabla u_n))$, we conclude that 
\begin{align*}
 \lim_{n\to\infty} \int_\Omega |u_n(x,T)|^{\alpha-1}u_n(x,T) v(x) \d x &= \int_\Omega |u_0|^{\alpha-1}u_0 v \d x - \iint_{\Omega_T} \mathcal{A}(x,t) \cdot \nabla v(x) \d x \d t
 \\
 & \quad +  \iint_{\Omega_T} f(x,t)v(x) \d x \d t
 \\
 &= \int_\Omega |u(x,T)|^{\alpha-1}u(x,T) v(x) \d x,
\end{align*}
where the last step follows from Lemma \ref{lem:3rd-equiv-formulation}. This limit was established under the assumption $v \in \cup^\infty_{n=1} V_n$, but by density arguments it holds also for all $v\in L^{\alpha+1}(\Omega)$, and we have thus confirmed that $w= |u|^{\alpha-1}u(\cdot, T)$. In light of this, we can restate \eqref{liminf-Gamma_n-T} as 
\begin{align}\label{liminf-aater}
 \liminf_{n\to \infty} \int_{\Omega} \Gamma_n (u_n(x,T)) \d x \geq \frac{\alpha}{\alpha+1}\int_\Omega |u(x,T)|^{\alpha+1} \d x.
\end{align}
Since also all terms in \eqref{calc:I_1} involving $B_n(u_n)$ converge to terms containing $|u|^{\alpha-1}u$, we see using \eqref{liminf-aater} and \eqref{lim-Gamma_n-u_n} that 
\begin{align}\label{limsup:I_1}
 \notag \limsup_{n\to \infty}& I_1(n) \leq \frac{\alpha}{\alpha+1}\int_\Omega |u_0|^{\alpha+1} \d x  - \frac{\alpha}{\alpha+1}\int_\Omega |u(x,T)|^{\alpha+1} \d x
 + \int_\Omega |u|^{\alpha-1}u g(x,T)\d x 
 \\
  &\quad - \int_\Omega |u(x,0)|^{\alpha-1}u(x,0)g(x,0)\d x - \iint_{\Omega_T} |u|^{\alpha-1}u \partial_t g \d x \d t + \iint_{\Omega_T} f(u - g)\d x \d t
 \\
\notag &= \iint_{\Omega_T}\mathcal{A} \cdot \nabla (u - g)\d x \d t, 
\end{align}
where the last step follows from Lemma \ref{lem:initial-val}.
For the term $I_2(n)$ in \eqref{est:X_n} we have directly by weak convergence that 
\begin{align}\label{lim:I_2}
 \lim_{n\to\infty} I_2(n) = \iint_{\Omega_T} \mathcal{A} \cdot \nabla (v - g)\, \d x \d t.
\end{align}
The term $I_3(n)$ can be treated as follows:
\begin{align*}
 I_3(n) &= \iint_{\Omega_T} A(x,t, u_n,\nabla v) \cdot ( \nabla u_n - \nabla v) \d x \d t 
 \\
 &= \iint_{\Omega_T} \big(A(x,t, u_n,\nabla v)  -  A(x,t,u,\nabla v)\big) \cdot ( \nabla u_n-\nabla v) \d x \d t 
\\
&\quad + \iint_{\Omega_T} A(x,t, u,\nabla v) \cdot ( \nabla u_n-\nabla v) \d x \d t 
\\
&= I_3^1(n) + I_3^2(n).
 \end{align*}
The term $I_3^1(n)$ will vanish in the limit due to the uniform bound of the partial derivatives $\partial_j u_n$ in $L^{p_j}(\Omega)$ combined with the pointwise a.e. convergence of $u_n$ to $u$ and the structure condition \eqref{cond:structure2}, which allows us to use the Dominated Convergence Theorem. The term $I_3^2(n)$ converges due to the weak convergence of the partial derivatives $\partial_j u_n$ to the corresponding expression without $n$, so all in all we have
\begin{align}\label{lim:I_3}
 \lim_{n\to \infty} I_3(n) = \iint_{\Omega_T} A(x,u,\nabla v) \cdot ( \nabla u - \nabla v) \d x \d t.
\end{align}
Combining \eqref{limsup:I_1}, \eqref{lim:I_2} and \eqref{lim:I_3} we end up with
\begin{align*}
 0 \leq \limsup_{n\to\infty} X_n \leq \iint_{\Omega_T}\mathcal{A} \cdot \nabla (u - g) \d x \d t - \iint_{\Omega_T} \mathcal{A} \cdot \nabla (v- g) \d x \d t 
 \\
 - \iint_{\Omega_T} A(x,u,\nabla v) \cdot ( \nabla u - \nabla v) \d x \d t.
\end{align*}
Especially, taking $v=u$ we see that 
\begin{align*}
 \lim_{n\to \infty} \iint_{\Omega_T} \big( A(x,t,u_n, \nabla u_n)- A(x,t, u_n, \nabla u)\big) \cdot \big( \nabla u_n- \nabla u \big) \d x \d t = 0.
\end{align*}
Given that the integrand is nonnegative by the monotonicity condition we have thus stated that the integrand converges to zero in $L^1(\Omega_T)$. Hence, after passing to another subsequence we have that 
\begin{align*}
  \lim_{n\to \infty}  \big( A(x,t,u_n, \nabla u_n)- A(x,t,u_n, \nabla u)\big) \cdot \big( \nabla u_n- \nabla u \big) = 0, 
\end{align*}
pointwise for almost every $(x,t)\in \Omega_T$. This means that we can reason as in Lemma 2.4 and Lemma 4.4 in Laptev, as these arguments are based only on the strict monotonicity property, to conclude that for a subsequence,
$\nabla u_n(x,t) \to \nabla u(x,t)$ pointwise for a.e. $(x,t)\in \Omega_T$. Since we already established that also $u_n$ converges pointwise a.e. to $u$ we have by the Caratheodory property that
\begin{align}\label{ptwise-convg}
 A(x,t, u_n(x,t), \nabla u_n(x,t)) \xrightarrow[n\to \infty]{} A(x,t, u(x,t),\nabla u(x,t)), \qquad \textnormal{for a.e. } (x,t)\in \Omega_T.
\end{align}
Note however that we already proved that each component $A_i(x,t,u_n,\nabla u_n)$ converges weakly in $L^{p_i'}(\Omega_T)$ to $\mathcal{A}_i$. Hence, by Mazur's lemma there is a sequence of convex combinations which converges strongly to $\mathcal{A}_i$. Passing to another subsequence we have pointwise a.e. convergence of the convex combinations to $\mathcal{A}_i$. However, these convex combinations also converge pointwise to $A_i(x,t,u,\nabla u)$ by \eqref{ptwise-convg} so we have confirmed that 
\begin{align*}
 A(x,t,u,\nabla u) = \mathcal{A},
\end{align*}
which completes the proof in the case $\alpha \in (0,1)$.

\vspace{3mm}
\noindent \textbf{Modifications in the case $\alpha \geq 1$.}
Instead of \eqref{est:theta-below} we now have 
\begin{align*}
 \theta_{\xi,t} \geq \varepsilon_n^\frac{\alpha-1}{2}.
\end{align*}
Consequently, the lower bound for any eigenvalue $\lambda$ of $F(\xi,t)$ becomes $\lambda \geq \alpha\varepsilon_n^{\alpha-1}$ which leads to the bound 
\begin{align*}
 \norm{F(\xi,t)^{-1}}_{\textnormal{op}} \leq \alpha^{-1}\varepsilon_n^{1-\alpha}.
\end{align*}
Note also that the assumptions $g,\, \partial_t g \in L^{\alpha+1}(\Omega_T)$ are sufficient for concluding that 
\begin{align*}
 |K(\xi,t)| \leq \mathcal{K}_1(t) + c|\xi|^{\alpha-1}\mathcal{K}_2(t), \qquad \textnormal{where } \mathcal{K}_1 \in L^1([0,T]) \textnormal{ and } \mathcal{K}_2 \in L^{\alpha+1}([0,T]), 
\end{align*}
which is the estimate replacing \eqref{est:K-xi-t} in this range. These modifications mean that we can replace \eqref{est:H} by an estimate of the form
\begin{align*}
 |H(\xi,t)| \leq \mathcal{N}_1(t) + \mathcal{N}_2(t)(1 + |\xi|)^{\max \{\alpha-1, p_N\}}, \qquad \textnormal{where } \mathcal{N}_1, \mathcal{N}_2 \in L^1([0,T]),
\end{align*}
which is sufficient to prove the local existence in a similar way as before. The estimate \eqref{est:xi-of-a} takes the same shape as before, but is now obtained by using H\"older's inequality on the $L^2$-norm of $w_n(\cdot,a)$ with the exponents $(\alpha+1)/2$ and $(\alpha+1)/(\alpha-1)$. We may again use the chain rule for Sobolev functions on $\Gamma_n(u_n)$ although the justifications are somewhat different in this case. The estimate \eqref{est:Gamma-lower} is replaced by
\begin{align*}
 \Gamma_n(u) = \alpha \int^{|u|}_0 (|s| + \varepsilon_n)^{\alpha-1}s \d s \geq \alpha \int^{|u|}_0 s^\alpha \d s = \frac{\alpha}{\alpha +1}|u|^{\alpha+1}.
\end{align*}
Instead of \eqref{B_n-ab-diff} we now have the estimate
\begin{align}\label{B_n-ab-diff-alpha-geq1}
 |B_n(b) - B_n(a)| = \Big| \alpha \int^b_a (|s|+ \varepsilon_n)^{\alpha-1} \d s\Big| &\leq \alpha \big| (|a|+ \varepsilon_n)^{\alpha-1} + (|b|+ \varepsilon_n)^{\alpha-1}\big| |a-b|
 \\
 \notag &\leq c\big( |a|^{\alpha-1} + |b|^{\alpha-1} + 1\big)|a-b|.
\end{align}
Instead of \eqref{B_n-upperbound} we have
\begin{align}\label{B_n-upperbound-alpha-geq1}
 |B_n(a)| = \int^{|a|}_0 \alpha(s+\varepsilon_n)^{\alpha-1} \d s \leq c \int^{|a|}_0 s^{\alpha-1} + \varepsilon_n^{\alpha-1} \d s = c |a|^\alpha + \varepsilon_n^{\alpha-1}|a| \leq c |a|^\alpha + \varepsilon_n^\alpha,
\end{align}
where we used Young's inequality in the last step. Therefore
\begin{align}\label{csg}
 |B_n(u_n) g|\leq c|u_n|^\alpha |g| + |g| \leq \varepsilon |u_n|^{\alpha+1} + c_\varepsilon |g|^{\alpha+1} + 1.
\end{align}
Integrating this estimate, we get a replacement for \eqref{shsh}. By putting $\partial_t g$ in place of $g$ in \eqref{csg} we obtain a replacement for \eqref{tttt}. All in all, this allows us to obtain an energy estimate as before, although with an additional $|\Omega_T|$ on the right-hand side. Instead of \eqref{est:B} we have due to \ref{alpha-geq1-stuff} of Lemma \ref{lem:b_alpha_geq1} the simpler estimate
\begin{align*}
 |b - &a|^{\alpha+1}  \leq c(|b|^{\alpha-1}b - |a|^{\alpha-1}a)(b-a)  = c \int^b_a |s|^{\alpha-1}\d s \,(b-a) 
 \\
 &\leq \alpha \int^b_a (|s| + \varepsilon_n )^{\alpha-1}\d s\, (b-a) = \int^b_a \frac{\d}{\d s} B_n(s) \d s \, (b - a) = (B_n(b) - B_n(a)) (b - a),
\end{align*}
and this allows us to conclude \eqref{time-translation} by a similar calculation as before, except that the use of H\"older's inequality can be avoided. The pointwise convergence of $B_n(u_n)$ to $|u|^{\alpha-1}u$ can in this range be concluded from the pointwise convergence of $u_n$ to $u$ and the continuity of 
\begin{align*}
 \R\times [0,\infty) \ni (u, \kappa) \mapsto \alpha\int^u_0 (|s| + \kappa)^{\alpha-1} \d s.
\end{align*}
Thus, since $(B_n(u_n))$ is a bounded sequence in $L^\frac{\alpha+1}{\alpha}(\Omega_T)$ by \eqref{B_n-upperbound-alpha-geq1}, we again find a subsequence which converges weakly to $|u|^{\alpha-1}u$. To conclude \eqref{a=0-strongnormlimit} we now use the fact that for $\alpha\geq 1$,
\begin{align*}
 \tfrac{\alpha}{\alpha+1} |v|^{\alpha+1} = \alpha \int^v_0 |s|^{\alpha-1} s \d s \leq \alpha \int^v_0 (|s| + \varepsilon_n)^{\alpha-1} s \d s = \Gamma_n(v)  &\leq \alpha \int^{|v|}_0 (s+\varepsilon_n)^\alpha \d s 
 \\
 &\leq \tfrac{\alpha}{\alpha+1}(|v|+\varepsilon_n)^{\alpha+1}.
\end{align*}
Thus,
\begin{align*}
 0 \leq \Gamma_n(v) - \tfrac{\alpha}{\alpha+1} |v|^{\alpha+1} \leq \tfrac{\alpha}{\alpha+1}\big((|v|+\varepsilon_n)^{\alpha+1} - |v|^{\alpha+1}\big) \leq c (|v|^\alpha + 1)\varepsilon_n.
\end{align*}
where the mean value theorem is used in the last step. From this it follows that
\begin{align*}
 \big|\Gamma_n&(u_n(x,0)) - \tfrac{\alpha}{\alpha+1} |u_0|^{\alpha+1}\big| 
 \\
 &\leq \big|\Gamma_n(u_n(x,0)) - \tfrac{\alpha}{\alpha+1} |u_n(x,0)|^{\alpha+1}\big| 
 + \tfrac{\alpha}{\alpha+1}\big| |u_n(x,0)|^{\alpha+1} - |u_0|^{\alpha+1}\big|
 \\
 &\leq c(|u_n(x,0)|^\alpha + 1)\varepsilon_n + \tfrac{\alpha}{\alpha+1}\big| |u_n(x,0)|^{\alpha+1} - |u_0|^{\alpha+1}\big|.
\end{align*}
Integrating over $\Omega$ and using the fact that $u_n(x,0)$ converges strongly in $L^{\alpha+1}(\Omega)$ we conclude that \eqref{a=0-strongnormlimit} is true. In order to identify the weak limit of $(|u_n|^{\alpha-1}u_n(x,T))$ we make use of the estimate 
\begin{align*}
 |B_n(a) - |a|^{\alpha-1}a| = \Big| \alpha \int^{|a|}_0 (s + \varepsilon_n)^{\alpha-1}\d s - |a|^\alpha\Big| = \big| (|a| + \varepsilon_n)^\alpha - \varepsilon_n^\alpha - |a|^\alpha \big| 
 \\
 \leq (|a| + \varepsilon_n)^\alpha - |a|^\alpha + \varepsilon_n^\alpha \leq c(|a|^{\alpha-1} + 1)\varepsilon_n + \varepsilon_n^\alpha,
\end{align*}
where the mean value theorem has been used in the last step. This estimate along with the boundedness of $(u_n(x,T))$ in $L^{\alpha+1}(\Omega)$ shows that the first integral on the right-hand side of \eqref{cleveland} vanishes in the limit $n\to \infty$. The limit of the second integral on the right-hand side of \eqref{cleveland} can be analyzed in a similar way, and thus we obtain the correct weak limit of the sequence $(|u_n|^{\alpha-1}u_n(x,T))$. The rest of the proof does not require any modifications.
\qed

\vspace{3mm}
We can use Theorem \ref{thm:existence} to conclude the existence of a solution also to the Cauchy problem on $S_T := \R^N \times (0,T)$, i.e. Theorem \ref{thm:existence-S_T}. Many of the steps of the proof are analogous to the corresponding steps in the proof of Theorem \ref{thm:existence} and will only be outlined briefly.

\vspace{3mm}
\noindent \textbf{Proof of theorem \ref{thm:existence-S_T}.}
 Theorem \ref{thm:existence} implies that for every $n\in \N$, there is a solution $u_n$ to the problem
\begin{align}
   \left\{
\begin{array}{ll}
\partial_t (|u_n|^{\alpha-1}u_n)  - \nabla\cdot A(x,t,u_n,\nabla u_n) = f, & \quad \text{in } B_n\times(0,T) 
\\[5pt]
 u_n(x,0) = u_0(x),  & \quad x \in B_n,
 \\
 u_n = 0,  & \quad \text{on } \partial B_n \times (0,T).
\end{array}
\right.
 \end{align} 
 where $B_n := B(\bar 0, n)$. Since $u_n \in L^{\alpha+1}(B_n\times(0,T)) \cap L^{\bf p}(0,T; W^{1, {\bf p}}_{\textnormal{o}}(B_n))$, we may extend $u_n$ as zero outside of $B_n$, obtaining an element $u_n \in L^{\alpha+1}(S_T) \cap L^{\bf p}(0,T; W^{1, {\bf p}}(\R^N))$ which is compactly supported in space.
We may utilize Lemma \ref{lem:initial-val} with $u=u_n$, $g=0$ and $\mathcal{A} = A(x,t,u_n,\nabla u_n)$ and the end point $T$ replaced by an arbitrary $a\in (0,T]$ to conclude that
\begin{align}\label{est:energy}
\sup_{t\in [0,T]} &\norm{u_n(\cdot,t)}^{\alpha+1}_{L^{\alpha+1}(\R^N)} + \iint_{S_T} \sum^N_{k=1} |\partial_k u_n|^{p_k} \d x \d t
\\
&=
\notag  \sup_{t\in [0,T]}\norm{u_n(\cdot,t)}^{\alpha+1}_{L^{\alpha+1}(B_n)} + \int^T_0 \int_{B_n} \sum^N_{k=1} |\partial_k u_n|^{p_k} \d x \d t 
 \\
\notag &\leq c \norm{u_0}_{L^{\alpha+1}(B_n)}^{\alpha+1} + c\int^T_0 \int_{B_n} |f|^\frac{\alpha+1}{\alpha} + \tilde a \d x \d t
 \\
 \notag & \leq c \norm{u_0}_{L^{\alpha+1}(\R^N)}^{\alpha+1} + c\iint_{S_T} |f|^\frac{\alpha+1}{\alpha} + \tilde a\,\d x \d t.
\end{align}
The argument makes use of the structure conditions and Young's inequality in a way which is analogous to the proof of the energy estimate for the approximating solutions in Step 2 of the proof of Theorem \ref{thm:existence}. Thus, we see that $(u_n)$ is a bounded sequence in $L^{\alpha+1}(S_T)$, the partial derivatives $(\partial_k u_n)$ form a bounded sequence in $L^{p_k}(S_T)$ and combining this fact with the structure condition \eqref{cond:structure2} we have that $(A_k(x,t,u_n,\nabla u_n))$ is a bounded sequence in $L^{p_k'}(\R^N)$. Thus, as in the proof of Theorem \ref{thm:existence} we obtain for a subsequence (still labelled as $(u_n)$) that 
\begin{enumerate}
 \item The sequence $(|u_n|^{\alpha-1}u_n)^\infty_{n=1}$ converges weakly in $L^\frac{\alpha+1}{\alpha}(S_T)$ to a limit function $u$.
 \item For each $j\in \{1, \dots, N\}$ the sequence $(\partial_j u_n)^\infty_{n=1}$ converges weakly in $L^{p_j}(S_T)$ to the weak derivative $\partial_j u$. 
 \item For each $j\in \{1, \dots, N\}$ the sequence $\big(A_j(x, t, u_n, \nabla u_n)\big)^\infty_{n=1}$ converges weakly in $L^{p_j'}(S_T)$ to a limit function $\mathcal{A}_j$.
\end{enumerate}
We also note that $u \in \mathring{U}^{1,{\bf p}}_{\alpha+1}$ To see this, note that each $u_n$ is in $\mathring{U}^{1,{\bf p}}_{\alpha+1}$ since its support is contained in $\bar B_n \times [0,T]$. Since $u$ is the weak limit of the functions $u_n$ in the space $U^{1,{\bf p}}_{\alpha+1}$ and since the subspace $\mathring{U}^{1,{\bf p}}_{\alpha+1}$ is closed we have that also $u$ belongs to $\mathring{U}^{1,{\bf p}}_{\alpha+1}$. 
As in the proof of Theorem \ref{thm:existence}, we want to prove strong convergence of the sequence $(u_n)$ in some $L^q$-space. For this purpose, note that by Lemma \ref{lem:3rd-equiv-formulation}, we may write
\begin{align} \label{aaather}
\notag &\int_{B_n} \big(|u_n|^{\alpha-1}u_n(x,t+h) - |u_n|^{\alpha-1}u_n(x,t)\big) v(x) \d x =
\\ & \int^h_0 \int_{B_n} f(x,t + s) v(x) \d x \d s 
-\int^h_0 \int_{B_n} A(x,t + s,u_n(x,t+s), \nabla u_n(x,t+s))\cdot \nabla v(x) \d x \d s,
\end{align}
for all $v\in C^\infty_{\textnormal{o}}(B_n)$. In the case $\alpha \in (0,1)$ we use the integral inequality of Lemma \ref{lem:integral-alpha-ineq}, and \eqref{aaather} with the choice  $v(x) = u_n(x,t+h) - v(x,t)$ and H\"older's inequality to obtain
\begin{align*}
 &\int^{T-h}_0 \int_{B_n} |u_n(x, t + h) - u_n(x,t)|^{\alpha+1} \d x \d t
 \\
 &\leq c\int^{T-h}_0 \Big[\int_{B_n}\big(|u_n|^{\alpha-1}u_n(x,t+h) - |u_n|^{\alpha-1}u_n(x,t)\big) (u_n(x,t+h) - u_n(x,t)) \d x \Big]^\frac{\alpha+1}{2}
 \\
&  \qquad \qquad \times \Big[ \int_{B_n} |u_n(x,t+h)|^{\alpha+1} + |u_n(x,t)|^{\alpha+1} \d x \Big]^\frac{1-\alpha}{2} \d t
\\
&\leq 
 c\Big(\Big[\int^{T-h}_0 \int^h_0 \int_{B_n} |A(x,t + s, u_n, \nabla u_n)\cdot \nabla (u_n(x,t+h) - u_n(x,t))| \d x \d s \d t \Big]^\frac{\alpha+1}{2}
\\
& \qquad + c\Big[\int^{T-h}_0   \int^h_0 \int_{B_n} f(x,t + s) (u_n(x,t+h) - u_n(x,t)) \d x \d s \d t\Big]^\frac{\alpha+1}{2}\Big) 
\\
&\quad \times \Big[\int^{T-h}_0 \int_{B_n} |u_n(x,t+h)|^{\alpha+1} + |u_n(x,t)|^{\alpha+1} \d x \d t\Big]^\frac{1-\alpha}{2}.
\end{align*}
Due to the energy estimate \eqref{est:energy}, we may reason as in the proof of Theorem \ref{thm:existence} to obtain the bound
\begin{align*}
 \int^{T-h}_0 \int_{B_n} |u_n(x, t + h) - u_n(x,t)|^{\alpha+1} \d x \d t \leq c h^\frac{\alpha+1}{2},
\end{align*}
where $c$ is independent of $n$. In the case $\alpha \geq 1$ the calculations are somewhat easier since we may use \ref{alpha-geq1-stuff} of Lemma \ref{lem:b_alpha_geq1} and avoid the application of H\"older's inequality.

The bound for translations in space is obtained as in the proof of Theorem \ref{thm:existence} with some simplifications as $g=0$ in the current case. Thus, we have that a suitable subsequence converges in $L^q_{\textnormal{loc}}(S_T)$ to $u$ for $q= \min\{\alpha+1, p_j\}$, and also pointwise a.e. after passing to another subsequence. 
The weak formulation satisfied by $u_n$ and the convergence of the sequences $(u_n)$ and $(A(x,t,u_n,\nabla u_n))$ in the sense mentioned above imply that $u$ satisfies the equation
\begin{align*}
 \iint_{S_T} |u|^{\alpha-1}u \partial_t \psi  - \mathcal{A}\cdot \nabla \psi \d x \d t = - \iint_{S_T} f \psi \d x \d t,
\end{align*}
for every $\psi \in C^\infty_{\textnormal{o}}(S_T)$. Since $u \in \mathring{U}^{1,{\bf p}}_{\alpha+1}$, this is sufficient to conclude that $u$ has a representative in $C([0,T];L^{\alpha+1}(\R^N))$ due to Lemma \ref{lem:time-cont_RN}.

In order to show that $u$ satisfies the correct initial condition, we use the weak formulations for both $u$ and $u_n$ with a test function of the form \eqref{test-funct-for-init-val} with $v\in C^\infty_{\textnormal{o}}(\Omega)$, and pass to the limit $\delta \to 0$ as before to conclude that 
\begin{align*}
 \int_{\R^N} &|u|^{\alpha-1}u(x,0)v(x)\d x = \iint_{S_T}\mathcal{A} \cdot \nabla v(x) \zeta(t) -  |u|^{\alpha-1}u v(x) \zeta'(t) - f v(x)\zeta(t) \d x \d t
 \\
 &= \lim_{n\to\infty} \int^T_0\int_{B_n} A(x,t,u_n,\nabla u_n) \cdot \nabla v(x)\zeta(t) - |u_n|^{\alpha-1}u_n v(x) \zeta'(t) - f v(x) \zeta(t) \d x \d t
 \\
 &= \lim_{n\to\infty} \int_{B_n} |u_n|^{\alpha-1}u_n(x,0) v(x) \d t 
 \\
 &= \int_{\R^N} |u_0|^{\alpha-1}u_0 v \d x.
\end{align*}
Since $v$ was arbitrary, this confirms that $u(\cdot, 0) = u_0$. It only remains to show that $\mathcal{A} = A(x,t,u,\nabla u)$. We proceed again using Minty's trick. Take $v \in L^{{\bf p}}(0,T; W^{1,\bf{p}}(\R^N))$ and consider the quantity
\begin{align*}
 X_n :=& \iint_{S_T} (A(x,t,u_n, \nabla u_n) - A(x,t,u_n, \nabla v)) \cdot (\nabla u_n - \nabla v) \d x \d t
 \\
=& \iint_{S_T} A(x, t, u_n, \nabla u_n) \cdot \nabla u_n \d x \d t 
  - \iint_{S_T} A(x, t, u_n, \nabla u_n) \cdot \nabla v \d x \d t 
  \\
  &- \iint_{S_T} A(x, t, u_n,\nabla v) \cdot ( \nabla u_n-\nabla v) \d x \d t
        \\
        &= I_1(n)-I_2(n)-I_3(n).
\end{align*}
Reasoning as in the proof of Theorem \ref{thm:existence} see that
\begin{align*}
 \lim_{n\to\infty} I_2(n) = \iint_{S_T} \mathcal{A} \cdot \nabla v \d x \d t, \qquad \lim_{n\to\infty} I_3(n) = \iint_{S_T} A(x,t, u, \nabla v)\cdot(\nabla u - \nabla v) \d x \d t.
\end{align*}
By Lemma \ref{lem:initial-val} we can calculate
\begin{align*}
 I_1(n) &= \int^T_0 \int_{B_n} A(x, t, u_n, \nabla u_n) \cdot \nabla u_n \d x \d t 
 \\
 &= \tfrac{\alpha}{\alpha+1}\int_{B_n}|u_0|^{\alpha+1} \d x - \tfrac{\alpha}{\alpha+1}\int_{B_n}|u_n(x,T)|^{\alpha+1} \d x + \int^T_0 \int_{B_n} f u_n \d x \d t.
\end{align*}
Before proceeding to analyze the limit of each term above we note that the equation satisfied by $u_n$ implies that for any $\psi\in C^\infty_{\textnormal{o}}(\R^N)$ we have 
\begin{align*}
 \lim_{n\to \infty}& \int_{\R^N} |u_n|^{\alpha-1}u_n(x,T) \psi(x) \d x = \lim_{n\to \infty} \int_{B_n} |u_n|^{\alpha-1}u_n(x,T) \psi(x) \d x
 \\
 &= \lim_{n\to \infty} \Big[ \int_{B_n} |u_n|^{\alpha-1}u_n(x,0) \psi(x) \d x - \int^T_0 \int_{B_n} A(x,t,u_n,\nabla u_n) \cdot \nabla \psi - f \psi \d x \d t  
 \\
 &= \int_{\R^N} |u_0|^{\alpha-1}u_0 \psi \d x - \iint_{S_T} \mathcal{A}\cdot \nabla \psi  -  f \psi \d x \d t
 \\
 &= \int_{\R^N} |u|^{\alpha-1}u(x,T) \psi(x) \d x,
\end{align*}
where in the arguments we have made use of the compact support of $\psi$ and the fact that $u_n(0)$ coincides with $u_0$ on $B_n$. In the last step we also use the equation satisfied by $u$. Thus we have identified $|u|^{\alpha-1}u(T)$ as the weak limit of $|u_n|^{\alpha-1}u_n(T)$. As in the proof of Theorem \ref{thm:existence} this allows us to conclude that
\begin{align*}
 \limsup_{n\to \infty} X_n &\leq \tfrac{\alpha}{\alpha+1}\int_{\R^N}|u_0|^{\alpha+1} \d x - \tfrac{\alpha}{\alpha+1}\int_{\R^N}|u(x,T)|^{\alpha+1} \d x + \iint_{S_T} f u \d x \d t 
 \\
  & \quad - \iint_{S_T} \mathcal{A} \cdot \nabla v \d x \d t - \iint_{S_T} A(x,t, u, \nabla v)\cdot(\nabla u - \nabla v) \d x \d t
  \\
  &\leq \iint_{S_T} \mathcal{A} \cdot \nabla u \d x \d t - \iint_{S_T} \mathcal{A} \cdot \nabla v \d x \d t - \iint_{S_T} A(x,t, u, \nabla v)\cdot(\nabla u - \nabla v) \d x \d t,
\end{align*}
where in the last step we apply Lemma \ref{lem:Raviart-RN} to $u$.
In particular, we can take $v=u$ and obtain
\begin{align*}
 \lim_{n\to\infty} \iint_{S_T} (A(x,t,u_n, \nabla u_n) - A(x,t,u_n, \nabla u)) \cdot (\nabla u_n - \nabla u) \d x \d t,
\end{align*}
which as seen in the proof of Theorem \ref{thm:existence} is sufficient to conclude that $\mathcal{A} = A(x,t,u,\nabla u)$.
\qed

\section{Proof of the comparison principle}\label{sec:comparison}
This section is devoted to the proof of Theorem \ref{thm:comparison}. The strategy for proving the comparison principle is similar to the proof of the corresponding result in \cite{BoeDuGiaLiSche}. 
Before proceeding with the proof, we need to verify two lemmas. For this purpose we introduce the following terminology: we say that $\tilde v \in W^{1, {\bf p}}(\Omega)$ is nonnegative on $\partial \Omega$ if $\tilde v$ is of the form
\begin{align}\label{tilde-v-repr}
 \tilde v = \tilde v_o + \tilde g, \quad \tilde v_o \in \overline W^{1, {\bf p}}_{\textnormal{o}}(\Omega), \quad  \tilde g \in W^{1, {\bf p}}(\Omega), \textnormal{ and } \tilde g \geq 0.
\end{align}

\begin{lem}\label{lem:complemma-v}
 Let $v$ be a solution to the Cauchy-Dirichlet problem with zero boundary values and right-hand side $f_1$. Let $\tilde v \in L^{\alpha+1}(\Omega)\cap W^{1, {\bf p}}(\Omega)$ be nonnegative on $\partial \Omega$. Then
 \begin{align}\label{eq:comparison-help}
  \iint_{\Omega_T} -\mathfrak{h}_\delta(v, \tilde v)\partial_t \psi + A(x,v ,\nabla v) \cdot \nabla [\mathcal{H}_\delta(v - \tilde v)]\psi \d x \d t = \iint_{\Omega_T}f_1 \mathcal{H}_\delta(v - \tilde v)\psi \d x \d t,
 \end{align}
for all $\psi \in C^\infty_{\textnormal{o}}(0, T)$. 
\end{lem}
\begin{proof}{}
We claim that $\mathcal{H}_\delta([v]_h - \tilde v)$ is in the space $ W^{1, \alpha + 1}(0, T; L^{\alpha + 1}(\Omega)) \cap L^{\bf p}(0,T; \overline W^{1, {\bf p}}_{\textrm{o}}(\Omega))$ so that we can use $\varphi_h:= \mathcal{H}_\delta([v]_h - \tilde v)\psi$, where $\psi \in C^\infty_{\textnormal{o}}(0, T)$ as test function as guaranteed by Lemma \ref{lem:weaksol2}. To see this, note that 
\begin{align*}
 [v]_h(x,t) - \tilde v(x) = ([v]_h(x,t) - \tilde v_o(x)) - \tilde g(x),
\end{align*}
where $\tilde v_o$ and $\tilde g$ are as in \eqref{tilde-v-repr}. The expression inside the brackets on the right-hand side forms an element of $L^{\bf p}(0,T; \overline W^{1, {\bf p}}_{\textrm{o}}(\Omega)) \cap L^{\alpha+1}(\Omega_T)$ which can thus be approximated by a sequence $(v_j) \subset C^\infty(\bar \Omega \times [0,T])$ such that $\supp v_j \subset K_j \times [0,T]$ for some compact $K_j \subset \Omega$. Since $\tilde g$ is nonnegative we have that $\mathcal{H}_\delta(v_j - \tilde g)$ vanishes when $x \notin K_j$. Passing to the limit $j\to \infty$ we see that also $\mathcal{H}_\delta([v]_h - \tilde v)$ is in the right space.
Since $\partial_k \big( \mathcal{H}_\delta([v]_h - \tilde v)\psi \big) \to \partial_k \big( \mathcal{H}_\delta(v - \tilde v)\psi \big)$ in $L^{p_k}(\Omega_T)$ we see using H\"older's inequality and the estimate \eqref{cond:structure2} for the components $A_k$ that 
 \begin{align}\label{h-to-zero-elliptic}
  \iint_{\Omega_T} A(x,v,\nabla v)\cdot \nabla \varphi_h \d x\d t  \xrightarrow[h\to 0]{} \iint_{\Omega_T} A(x,v,\nabla v)\cdot \nabla [ \mathcal{H}_\delta(v - \tilde v)]\psi \d x\d t.
 \end{align}
 For the right-hand side we have 
 \begin{align}\label{h-to-zero-rhs}
  \iint_{\Omega_T} f_1 \varphi_h \d x \d t \xrightarrow[h\to 0]{}  \iint_{\Omega_T}f_1 \mathcal{H}_\delta(v - \tilde v)\psi \d x \d t.
 \end{align}
 The diffusion part can be re-written as 
 \begin{align}\label{diffpart-split}
 \iint_{\Omega_T} |v|^{\alpha-1}v \partial_t \varphi_h \d x\d t = \iint_{\Omega_T} |v_h|^{\alpha-1}v_h \partial_t \varphi_h \d x\d t + \iint_{\Omega_T} \big( |v|^{\alpha-1}v - |v_h|^{\alpha-1}v_h \big) \partial_t \varphi_h \d x\d t.
 \end{align}
For the second integral on the right-hand side we can use property \ref{expmol2} of Lemma \ref{lem:expmolproperties} to estimate
\begin{align*}
   \iint_{\Omega_T} \big( |v|^{\alpha-1}v - |v_h|^{\alpha-1}v_h \big) &\partial_t \varphi_h \d x\d t =  \iint_{\Omega_T} \big( |v|^{\alpha-1}v - |v_h|^{\alpha-1}v_h \big) \mathcal{H}_\delta(v_h - \tilde v)\partial_t \psi \d x\d t 
   \\
   &\quad + \iint_{\Omega_T} \big( |v|^{\alpha-1}v - |v_h|^{\alpha-1}v_h \big)\delta^{-1}\chi_{\{ 0<[v]_h - \tilde v < \delta\}}\partial_t v_h \psi \d x\d t
   \\
   &\geq  \iint_{\Omega_T} \big( |v|^{\alpha-1}v - |v_h|^{\alpha-1}v_h \big) \mathcal{H}_\delta([v]_h - \tilde v)\partial_t \psi \d x\d t.
\end{align*}
In the first term on the right-hand side of \eqref{diffpart-split} we would like to integrate by parts and use the chain rule. In the case $\alpha < 1$ the chain rule is not directly applicable, but after introducing an extra parameter $\varepsilon>0$ we can integrate by parts as follows:
\begin{align*}
 \iint_{\Omega_T} &|v_h|^{\alpha-1}v_h \partial_t \varphi_h \d x\d t = \lim_{\varepsilon \downarrow 0} \iint_{\Omega_T} (|v|_h + \varepsilon)^{\alpha-1}v_h \partial_t \varphi_h \d x\d t 
 \\
 &=  - \lim_{\varepsilon \downarrow 0} \iint_{\Omega_T}\partial_t [(|v_h| + \varepsilon)^{\alpha-1}v_h] \varphi_h  \d x\d t
 \\
 &=  - \lim_{\varepsilon \downarrow 0} \iint_{\Omega_T}[ (\alpha-1)(|v_h|+\varepsilon)^{\alpha-2}|v_h| +  (|v_h|+\varepsilon)^{\alpha-1}](\partial_t v_h)  \mathcal{H}_\delta(v_h - \tilde v)\psi \d x\d t
 \\
 &=  - \lim_{\varepsilon \downarrow 0} \iint_{\Omega_T}\partial_t \mathfrak{h}_\delta^\varepsilon(v_h, \tilde v) \psi \d x\d t
 \\
 &= \lim_{\varepsilon \downarrow 0} \iint_{\Omega_T} \mathfrak{h}_\delta^\varepsilon(v_h, \tilde v) \partial_t\psi \d x\d t
  \\
 &=  \iint_{\Omega_T} \mathfrak{h}_\delta(v_h, \tilde v) \partial_t\psi \d x\d t.
\end{align*}
 Note that integration by parts was possible due to the fact that $\varphi_h$ can be approximated in a suitable sense by smooth compactly supported functions according to Lemma \ref{lem:approx_with_smooth_funct}. In the case $\alpha \geq 1$ we can instead use the chain rule without introducing the parameter $\varepsilon$. Combining these calculations and \eqref{diffpart-split} we have
\begin{align}\label{h-to-zero-diffusion}
 \iint_{\Omega_T} |v|^{\alpha-1}v \partial_t \varphi_h \d x\d t &\geq  \iint_{\Omega_T} \big( |v|^{\alpha-1}v - |v_h|^{\alpha-1}v_h \big) \mathcal{H}_\delta(v_h - \tilde v)\partial_t \psi \d x \d t
 \\
  & \notag \quad +  \iint_{\Omega_T} \mathfrak{h}_\delta(v_h, \tilde v) \partial_t\psi \d x\d t 
 \\
 \notag & \xrightarrow[h\to 0]{} \iint_{\Omega_T} \mathfrak{h}_\delta(v, \tilde v) \partial_t\psi \d x\d t.
\end{align}
Thus combining \eqref{h-to-zero-elliptic}, \eqref{h-to-zero-rhs} and \eqref{h-to-zero-diffusion} we see that passing to the limit $h\to 0$ with our particular test function proves the inequality ``$\geq$'' in \eqref{eq:comparison-help}. Using the test function $\mathcal{H}_\delta(v_{\bar h} - \tilde v)\psi$ and performing similar calculations proves the reverse estimate.
 \end{proof}
Using similar techniques as in the previous proof, we also have the following.
 \begin{lem}\label{lem:complemma-w}
 Let $w$ be a solution to the Cauchy-Dirichlet problem with right-hand side $f_2$ and nonnegative boundary values. Let $\tilde w \in L^{\alpha+1}(\Omega)\cap \overline W^{1, {\bf p}}_{\textnormal{o}}(\Omega)$. Then
 \begin{align}\label{eq:comparison-help2}
  \iint_{\Omega_T} -\widehat{\mathfrak{h}}_\delta(w, \tilde w)\partial_t \psi + A(x,w ,\nabla w) \cdot \nabla [\widehat{\mathcal{H}}_\delta(w - \tilde w)]\psi \d x \d t =  \iint_{\Omega_T}f_2 \widehat{\mathcal{H}}_\delta(w - \tilde w)\psi \d x \d t
 \end{align}
for all $\psi \in C^\infty_{\textnormal{o}}(0, T)$. 
\end{lem}
\begin{proof}{}
 Note that since $w$ has nonnegative boundary values we can write $w=w_o + g$ where $w_o \in L^{\bf p}(0,T; \overline W^{1, {\bf p}}_{\textnormal{o}}(\Omega)) \cap L^{\alpha+1}(\Omega_T)$, $g \in L^{\bf p}(0,T;  W^{1, {\bf p}}(\Omega)) \cap L^{\alpha+1}(\Omega_T)$, and $g\geq 0$. Thus,
 \begin{align*}
  w_h(x,t) - \tilde w(x) = ((w_o)_h(x,t) - \tilde w(x)) + g_h(x,t). 
 \end{align*}
The expression in the brackets on the right-hand side forms an element of $L^{\bf p}(0,T; \overline W^{1, {\bf p}}_{\textnormal{o}}(\Omega)) \cap L^{\alpha+1}(\Omega)$ and $g_h \geq 0$, so reasoning as in the proof of Lemma \ref{lem:complemma-v} we can justify the test function $\varphi_h:= \widehat{\mathcal{H}}_\delta([w]_h - \tilde w)\psi$. The rest of the argument is analogous to the proof of Lemma \ref{lem:complemma-v}.
\end{proof}
Now we are ready to prove the comparison principle.

\noindent \textbf{Proof of Theorem \ref{thm:comparison}.} Define the set 
\begin{align*}
 Q := \Omega \times (0,T)^2,
\end{align*}
and define extensions of $v$ and $w$ as follows:
\begin{align*}
\begin{array}{rl}
 \hat v: Q \to \R, & \quad \, \hat v(x,t_1,t_2) := v(x,t_1),
 \\
 \hat w: Q \to \R, & \quad \hat w(x,t_1,t_2) := w(x,t_2).
 \end{array}
\end{align*}
Take $\psi \in C^\infty_{\textnormal{o}}((0,T)^2; [0,\infty))$. For fixed $\delta>0$ and a.e. $t_2\in (0,T)$ the function $\tilde v(x) := w(x,t_2)=:w_{t_2}(x)$ is a valid choice in Lemma \ref{lem:complemma-v} and we can use the function $\psi_{t_2}(t):= \psi(t, t_2)$ in place of the $\psi$ appearing in Lemma \ref{lem:complemma-v}. Thus, we have
\begin{align}\label{eq:t2-fixed}
  \iint_{\Omega_T} -\mathfrak{h}_\delta(v, w_{t_2})\partial_{t_1} \psi_{t_2} + A(x,v ,\nabla v) \cdot \nabla [\mathcal{H}_\delta(v - w_{t_2})]\psi_{t_2} \d x \d t_1 
  \\
 \notag  = \iint_{\Omega_T}f_1 \mathcal{H}_\delta(v - w_{t_2})\psi_{t_2} \d x \d t_1,
 \end{align}
where $v$ is to be understood as a function of $x$ and $t_1$, $w_{t_2}$ is a function of $x$, and $\psi_{t_2}$ is a function of $t_1$. Similarly, for fixed $\delta>0$ and a.e. $t_1\in (0,T)$ we have that $\tilde w(x) := v(x,t_1) =: v_{t_1}(x)$ is a valid choice in Lemma \ref{lem:complemma-w}. Using $\psi^{t_1}(t):=\psi(t_1,t)$ in place of $\psi$ in Lemma \ref{lem:complemma-w} we end up with
\begin{align}\label{eq:t1-fixed}
  \iint_{\Omega_T} -\widehat{\mathfrak{h}}_\delta(w, v_{t_1})\partial_{t_2} \psi^{t_1} + A(x,w ,\nabla w) \cdot \nabla [\widehat{\mathcal{H}}_\delta(w - v_{t_1})]\psi^{t_1} \d x \d t_2 
  \\
  \notag 
  = \iint_{\Omega_T}f_2 \widehat{\mathcal{H}}_\delta(w - v_{t_1})\psi^{t_1} \d x \d t_2,
 \end{align}
 Integrating \eqref{eq:t2-fixed} w.r.t. $t_2$, integrating \eqref{eq:t1-fixed} w.r.t. $t_1$, adding the two resulting integrals and using the relation between $\mathcal{H}_\delta$ and $\widehat{\mathcal{H}}_\delta$, we obtain
 \begin{align}\label{eq:added-integs}
  \iiint_Q \Big[&-\big(\mathfrak{h}_\delta(\hat v, \hat w)\partial_{t_1}\psi + \widehat{\mathfrak{h}}_\delta(\hat w, \hat v)\partial_{t_2} \psi \big) 
  \\
 \notag &+ [A(x, \hat v, \nabla \hat v) - A(x, \hat w, \nabla \hat w)]\cdot \nabla \mathcal{H}_\delta(\hat v - \hat w) \psi \Big]\d x \d t_1 \d t_2 
 \\
 \notag &= \iiint_Q (f_1 - f_2) \mathcal{H}_\delta(\hat v - \hat w) \psi \d x \d t_1 \d t_2 \leq 0,
 \end{align}
 where the last estimate follows from the nonnegativity of $\psi$ combined with the relation between $f_1$ and $f_2$.
Our goal is to pass to the limit $\delta \to 0$ in this estimate. For this purpose, note that due to the monotonicity condition \eqref{cond:monot} we have
\begin{align}\label{est:use-monot}
 [A(x, \hat v, \nabla \hat v) &- A(x, \hat w, \nabla \hat w)]\cdot \nabla \mathcal{H}_\delta(\hat v - \hat w) 
 \\
 \notag &= [A(x, \hat v, \nabla \hat v) - A(x, \hat v, \nabla \hat w)]\cdot (\nabla \hat v - \nabla \hat w) \mathcal{H}_\delta'(\hat v - \hat w)
 \\
 \notag & \quad + [A(x, \hat v, \nabla \hat w) - A(x, \hat w, \nabla \hat w)]\cdot (\nabla \hat v - \nabla \hat w) \mathcal{H}_\delta'(\hat v - \hat w)
 \\
 \notag &\geq [A(x, \hat v, \nabla \hat w) - A(x, \hat w, \nabla \hat w)]\cdot (\nabla \hat v - \nabla \hat w) \mathcal{H}_\delta'(\hat v - \hat w).
\end{align}
Using the continuity condition \eqref{cond:lip-cont} we may estimate the following integral involving the last expression:
\begin{align*}
 \Big|\iiint_Q &[A(x, \hat v, \nabla \hat w) - A(x, \hat w, \nabla \hat w)]\cdot (\nabla \hat v - \nabla \hat w) \mathcal{H}_\delta'(\hat v - \hat w)\psi \d x \d t_1 \d t_2 \Big|
 \\
 &\leq C\delta^{-1}\sum^N_{j=1}\iiint_{Q\cap \{ 0 < \hat v - \hat w < \delta\}}|\hat v - \hat w|\Big(\tilde c(x) + \sum^N_{k=1}|\partial_k \hat w|^{p_k}\Big)^\frac{p_j-1}{p_j}|\partial_j \hat v - \partial_j \hat w|\psi \d x \d t_1 \d t_2
 \\
 &\leq C \sum^N_{j=1}\iiint_{Q\cap \{ 0 < \hat v - \hat w < \delta\}}\Big(\tilde c(x) + \sum^N_{k=1}|\partial_k \hat w|^{p_k}\Big)^\frac{p_j-1}{p_j}|\partial_j \hat v - \partial_j \hat w|\psi \d x \d t_1 \d t_2.
\end{align*}
The function in the integral above is integrable due to H\"older's inequality and the integrability properties of the spatial partial derivatives of $\hat v$ and $\hat w$. Thus, by the dominated convergence theorem we have that
 \begin{align}\label{lim:delta-to-zero}
  \lim_{\delta \downarrow 0} \iiint_Q &[A(x, \hat v, \nabla \hat w) - A(x, \hat w, \nabla \hat w)]\cdot (\nabla \hat v - \nabla \hat w) \mathcal{H}_\delta'(\hat v - \hat w)\psi \d x \d t_1 \d t_2 = 0.
 \end{align}
Combining \eqref{eq:added-integs}, \eqref{est:use-monot} and \eqref{lim:delta-to-zero} we end up with
\begin{align}\label{limsup:delta-to-zero}
 \limsup_{\delta \downarrow 0} \iiint_Q -\big(\mathfrak{h}_\delta(\hat v, \hat w)\partial_{t_1}\psi + \widehat{\mathfrak{h}}_\delta(\hat w, \hat v)\partial_{t_2} \psi \big) \d x \d t_1 \d t_2 \leq 0.
\end{align}
The bounds and limits \eqref{h-delta-upperbnd} - \eqref{hat-h-delta-limit} show that the dominated convergence theorem can be used in \eqref{limsup:delta-to-zero} to conclude that
\begin{align*}
 \iiint_Q -(|\hat v|^{\alpha-1}\hat v - |\hat w|^{\alpha-1}\hat w)_+(\partial_{t_1}\psi + \partial_{t_2}\psi) \d x \d t_1 \d t_2 \leq 0.
\end{align*}
As in \cite{BoeDuGiaLiSche} we now take a nonnegative functions $\phi \in C^\infty_{\textnormal{o}}(0,T)$ and $\varphi \in C^\infty_{\textnormal{o}}(\R)$ such that $\varphi$ has unit $L^1$-norm,  and choose the test function $\psi$ as 
\begin{align*}
 \psi(t_1,t_2) := \frac1\varepsilon \varphi\Big(\frac{t_1 - t_2}{\varepsilon}\Big)\phi\Big(\frac{t_1 + t_2}{2}\Big),
\end{align*}
and pass to the limit $\varepsilon\to 0$ and end up with
\begin{align*}
 -\iint_{\Omega_T}(|v|^{\alpha-1}v - |w|^{\alpha-1}w)_+ \phi'(t)\d x \d t \leq 0.
\end{align*}
Taking $\phi$ of the form \eqref{func:trapets} with $t_2=\tau \in (0,T)$ and $t_1<\tau$, we have for small $\varepsilon>0$ that
\begin{align*}
 \frac1\varepsilon \int^{\tau}_{\tau-\varepsilon} \int_\Omega (|v|^{\alpha-1}v - |w|^{\alpha-1}w)_+ \d x \d t \leq \frac1\varepsilon \int^{t_1+\varepsilon}_{t_1} \int_\Omega (|v|^{\alpha-1}v - |w|^{\alpha-1}w)_+ \d x \d t.
\end{align*}
Since $v$ and $w$ satisfy the assumptions of Lemma \ref{lem:time-cont}, both $v$ and $w$ are in $C([0,T];L^{\alpha+1}(\Omega))$, and we may pass to the limit $\varepsilon \to 0$ to obtain
\begin{align*}
 \int_{\Omega} (|v|^{\alpha-1}v - |w|^{\alpha-1}w)_+(x,\tau) \d x \leq  \int_{\Omega} (|v|^{\alpha-1}v - |w|^{\alpha-1}w)_+(x,t_1) \d x.
\end{align*}
Finally, by the aforementioned time-continuity we may take $t_1\to 0$ which yields
\begin{align*}
 \int_{\Omega} (|v|^{\alpha-1}v - |w|^{\alpha-1}w)_+(x,\tau) \d x \leq  \int_{\Omega} (|v_o|^{\alpha-1}v_o - |w_o|^{\alpha-1}w_o)_+ \d x = 0,
\end{align*}
where the last step follows from the assumption $v_o \leq w_o$. This confirms that $v\leq w$ on $\Omega_T$.
\qed

\appendix
\section{Function spaces and equivalent definitions of solutions}\label{app:equiv-sol}
\noindent \textbf{Proof of Lemma \ref{lem:isomorphic-spaces}}. We want to prove that the linear map  
\begin{align*}
 S: C^\infty_{\textnormal{o}}(\Omega_T) \to L^{\bf p}(0,T; \overline W^{1, {\bf p}}_{\textnormal{o}}(\Omega)), \quad S v(t) = v(\cdot, t)
\end{align*}
extends to an isometric isomorphism on the closure of $C^\infty_{\textnormal{o}}(\Omega_T)$ in the space
\begin{align*}
 V:= \{ v \in L^1(\Omega_T) \,|\, \partial_k v \in L^{p_k}(\Omega_T)\},
\end{align*}
i.e. the space of integrable functions on $\Omega_T$ that have weak first order spatial derivatives $\partial_k v$ that are $L^{p_k}$-integrable on $\Omega_T$, endowed with the norm
\begin{align*}
 \norm{v} = \norm{v}_{L^1(\Omega_T)} + \sum^N_{k=1} \norm{\partial_k v}_{L^{p_k}(\Omega_T)}.
\end{align*}
By definition, $S$ is isometric, so it extends to an isometry on the closure of $C^\infty_{\textnormal{o}}(\Omega_T)$. Thus $S$ is also injective. It remains to show that $S$ is surjective, and for this it suffices to show that $S[C^\infty_{\textnormal{o}}(\Omega_T)]$ is dense in $L^{\bf p}(0,T; \overline W^{1, {\bf p}}_{\textnormal{o}}(\Omega))$. It follows directly from the definitions that 
\begin{align*}
 L^{\bf p}(0,T; \overline W^{1, {\bf p}}_{\textnormal{o}}(\Omega)) \subset L^1(0,T;\overline W^{1, {\bf p}}_{\textnormal{o}}(\Omega)).
\end{align*}
From this we see that for any $u \in L^{\bf p}(0,T; \overline W^{1, {\bf p}}_{\textnormal{o}}(\Omega))$ we can define the convolution $u^\varepsilon_\delta:= \eta_\varepsilon * (u\chi_\delta) \in C^\infty_{\textnormal{o}}([0,T]; \overline W^{1, {\bf p}}_{\textnormal{o}}(\Omega))$, where $\eta_\varepsilon$ is a smooth mollifier, $\varepsilon < \delta$ and $\chi_\delta$ is the characteristic function of $(\delta, T- \delta)$. It is easy to verify that
$\partial_k \circ u^\varepsilon_\delta = \eta_\varepsilon *(\partial_k \circ (\chi_\delta u))$, so that $u^\varepsilon_\delta \in L^{\bf p}(0,T; \overline W^{1, {\bf p}}_{\textnormal{o}}(\Omega))$ and we can approximates $u$ by $u^\varepsilon_\delta$ in $L^{\bf p}(0,T; \overline W^{1, {\bf p}}_{\textnormal{o}}(\Omega))$ by taking $\delta$ and $\varepsilon$ sufficiently small. We can approximate $u^\varepsilon_\delta$ in $L^{\bf p}(0,T; \overline W^{1, {\bf p}}_{\textnormal{o}}(\Omega))$ by piecewise constant functions $\sum^M_{j=1} \chi_{\Delta_j} v_j$ where $v_j \in \overline W^{1, {\bf p}}_{\textnormal{o}}(\Omega)$ and $\Delta_j = \Big[\tfrac{T(j-1)}{M}, \tfrac{Tj}{M}\Big)$. Since $u^\varepsilon_\delta$ has compact support in $(0,T)$ we can take $v_1 = v_M = \bar 0$. Such functions can in turn be approximated in $L^{\bf p}(0,T; \overline W^{1, {\bf p}}_{\textnormal{o}}(\Omega))$ by functions of the form $\sum^{M-1}_{j=2} \chi_{\Delta_j} \varphi_j$ where $\varphi_j \in C^\infty_{\textnormal{o}}(\Omega)$. Finally, by introducing another convolution we can approximate these function by functions by functions of the form
\begin{align*}
 \sum^{M-1}_{j=2} (\eta_r * \chi_{\Delta_j}) \varphi_j.
\end{align*}
The corresponding functions on $\Omega_T$,
\begin{align*}
 (x,t) \mapsto \sum^{M-1}_{j=2} (\eta_r * \chi_{\Delta_j})(t) \varphi_j(x),
\end{align*}
are evidently in $C^\infty_{\textnormal{o}}(\Omega_T)$ for sufficiently small $r$. \qed

In order to simply terminology, we will in the following not make a clear distinction of elements $\varphi \in C^\infty_{\textnormal{o}}(\Omega_T)$ and the corresponding map $t\mapsto \varphi(\cdot,t)$. This simplifies the formulation of many results, such as the following lemma.
\begin{lem}\label{lem:approx_with_smooth_funct}
The space $C^\infty_{\textnormal{o}}(\Omega_T)$ is a dense subspace of the space
 \begin{align*}
  E := \{ v \in W^{1, \alpha + 1}(0, T; L^{\alpha + 1}(\Omega)) \cap L^{\bf p}(0,T; \overline W^{1, {\bf p}}_\textnormal{o}(\Omega)) \,|\, v(0) = 0 = v(T) \},
 \end{align*}
endowed with the norm of $W^{1, \alpha + 1}(0, T; L^{\alpha + 1}(\Omega)) \cap L^{\bf p}(0,T; \overline W^{1, {\bf p}}_\textnormal{o}(\Omega))$.
\end{lem}
\begin{proof}{}
We will prove the result by several steps. In each step we show that the functions from the previous step can be approximated by another class of functions which is more regular in some sense. In the final step we prove approximation with functions in $C^\infty_{\textnormal{o}}(\Omega_T)$.

\vspace{2mm}
 \noindent \textit{Step 1:} Let $v \in E$. We extend $v$ as zero outside $[0,T]$ and approximate $v$ using a re-scaling in time: $v^\varepsilon(x,t) := v(x, R_\varepsilon(t-\varepsilon))$ where $\varepsilon > 0 $ and $R_\varepsilon = T/(T -2 \varepsilon)$. Note that $v^\varepsilon$ vanishes outside $(\varepsilon, T - \varepsilon)$. Furthermore, $v^\varepsilon$ approximates $v$ in the norm of $W^{1, \alpha + 1}(0, T; L^{\alpha + 1}(\Omega)) \cap L^{\bf p}(0,T; \overline W^{1, {\bf p}}_{\textrm{o}} (\Omega))$. To see this, note first that for any $v \in L^q(0,T; L^q(\Omega)) \simeq L^q(\Omega_T)$ we may approximate $v$ using a continuous bounded function $w\in C(\Omega_T)$. Now also $w^\varepsilon$ approximates $v^\varepsilon$ and we have 
 \begin{align*}
  \norm{v - v^\varepsilon}_{L^q(\Omega_T)} \leq \norm{v - w}_{L^q(\Omega_T)} + \norm{w - w^\varepsilon}_{L^q(\Omega_T)} + \norm{w^\varepsilon - v^\varepsilon}_{L^q(\Omega_T)}.
 \end{align*}
By the previous discussion the first and third term on the right-hand side can be made arbitrarily small, and the approximation is stable as $\varepsilon \to 0$. Then, for a fixed $w$ we see that also the second term on the right-hand side can be made small by the DCT. Since $\partial_j (v^\varepsilon) = (\partial_j v)^\varepsilon$, a similar argument can be made also for approximation of the weak spatial derivatives. It remains to show that $v^\varepsilon \in W^{\alpha+1}(0,T; L^{\alpha + 1}(\Omega))$ and that $\partial_t v^\varepsilon$ approximates $\partial_t v$. Due to the fact that $v(0) = 0 = v(T)$ we have that the zero-extension of $v$ to $\R$ is in fact in $W^{\alpha+1}(\R; L^{\alpha+1}(\Omega))$. 
Thus, also the re-scaled function $v^\varepsilon$ is in $W^{\alpha+1}(\R; L^{\alpha+1}(\Omega))$. Moreover, weak derivatives behave as expected: 
\begin{align*}
 \partial_t v^\varepsilon(x,t) = R_\varepsilon \partial_t v(x, R_\varepsilon(t-\varepsilon)),
\end{align*}
and similar techniques as before can be used to verify approximation of the function and the first order time derivative w.r.t. the norm of $L^{\alpha+1}(0,T;L^{\alpha+1}(\Omega))$. The conclusion is that we may assume that $v$ vanishes in a neighborhood of $0$ and $T$.
 
 \vspace{2mm}
 \noindent \textit{Step 2:} Given $v\in E$ vanishing in a neighborhood of $0$ and $T$ we let $v_\varepsilon$ denote its mollification in time w.r.t. a smooth mollifier $\eta_\varepsilon : \R\to \R$: $v_\varepsilon = v * \eta_\varepsilon$. Reasoning as in the proof of Lemma \ref{lem:isomorphic-spaces}, we see that $v_\varepsilon$ is a well-defined map $[0,T]\to L^{\alpha+1}(\Omega) \cap \overline W^{1, \bf p}_{\textrm{o}}(\Omega)$. Since we assumed that $v$ vanishes in a neighborhood of $0$ and $T$, we have that 
  $v_\varepsilon \in C^\infty_0(0,T; L^{\alpha+1}(\Omega)\cap \overline W^{1, \bf p}_{\textrm{o}}(\Omega))$ when $\varepsilon$ is sufficiently small. Since $v_\varepsilon$ approximates $v$ in the appropriate norm of $E$, we may henceforth suppose that $v \in C^\infty_{\textnormal{o}}(0,T; L^{\alpha+1}(\Omega)\cap \overline W^{1, \bf p}_{\textrm{o}}(\Omega))$. 

 \vspace{2mm} 
\noindent \textit{Step 3:} For $0<\varepsilon < \tfrac14$ we define the piecewise affine function
\begin{align*}
 F_\varepsilon(t) := 
 \left\{
\begin{array}{ll}
0, & \quad t < 2\varepsilon
\\[5pt]
 (1-4\varepsilon)^{-1}(t- 2\varepsilon),  & t \in [2\varepsilon, 1 - 2\varepsilon]
 \\
1,  & \quad t > 1 -2\varepsilon,
\end{array}
\right.
\end{align*}
and introduce the smooth functions
\begin{align*}
 S_\varepsilon = \eta_\varepsilon * F_\varepsilon, \quad S^{a,b}_\varepsilon(t) := S_\varepsilon\Big(\frac{t-a}{b-a}\Big), \quad a < b.
\end{align*}
Take $N\in \N$ and consider the points 
\begin{align*}
 t_k:= kT/N, \quad k \in \{0, \dots, N\}.
\end{align*}
Define now the function
\begin{align*}
 v^N_\varepsilon(t) = v(t_k)\big(1 - S^{t_k, t_{k+1}}_\varepsilon(t)\big) + v(t_{k+1}) S^{t_k, t_{k+1}}_\varepsilon(t), \quad t \in [t_k, t_{k+1}].
\end{align*}
Note that the pointwise values $v(t_k)$ and $v(t_{k+1})$ are well defined due to the continuity of $v$ on the time interval.  We show that we can approximate $v$ w.r.t.~the norm of $W^{1, \alpha + 1}(0, T; L^{\alpha + 1}(\Omega)) \cap L^{\bf p}(0,T; \overline W^{1, {\bf p}}_{\textrm{o}} (\Omega))$ using functions of the form $v^N_\varepsilon$. First of all, these functions belong to the right space by construction. Note now that for $q\in \{\alpha+1, 1\}$,
\begin{align*}
 &\int^T_0  \norm{v^N_\varepsilon  - v}_{L^q(\Omega)}^q \d t = \sum^{N-1}_{k=0}\int^{t_{k+1}}_{t_k}\norm{v^N_\varepsilon - v}_{L^q(\Omega)}^q \d t 
 \\
 &\leq \sum^{N-1}_{k=0}\int^{t_{k+1}}_{t_k}\big( \big(1 - S^{t_k, t_{k+1}}_\varepsilon(t)\big)\norm{v(t_k) - v(t)}_{L^q(\Omega)} + S^{t_k, t_{k+1}}_\varepsilon(t) \norm{v(t_{k+1}) - v(t)}_{L^q(\Omega)} \big)^q \d t
\end{align*}
The functions $S^{t_k, t_{k+1}}_\varepsilon$ are bounded, and $v$ is uniformly continuous on the interval $[0,T]$ into $L^q(\Omega)$ so the convergence is clear as $N \to \infty$. A similar argument works for approximation in the spatial derivatives since $v$ is uniformly continuous also into $\overline W^{1,{\bf p}}_{\textnormal{o}}(\Omega)$. To show approximation in the time derivative, note that 
\begin{align*}
 {v^N_\varepsilon}'(t) = (v(t_{k+1}) - v(t_k) ){(S^{t_k, t_{k+1}}_\varepsilon)}'(t) &= \frac{(v(t_{k+1}) - v(t_k))}{t_{k+1} - t_k}{S_\varepsilon}'\Big(\frac{t-t_k}{t_{k+1} - t_k}\Big)
 \\
 & = {S_\varepsilon}'\Big(\frac{t-t_k}{t_{k+1} - t_k}\Big) \dashint^{t_{k+1}}_{t_k} v'(s) \d s , \quad t \in [t_k, t_{k+1}],
\end{align*}
where the last step follows from the smoothness of $v$ on the time interval. From the definition of $S_\varepsilon$ it follows that 
\begin{align*}
 {S_\varepsilon}'\Big(\frac{t-t_k}{t_{k+1} - t_k}\Big) = \frac{1}{1 - 4\varepsilon} = 1 + \frac{4\varepsilon}{1 - 4\varepsilon}, \quad t_k + 3\varepsilon \Delta_N < t < t_{k+1} - 3\varepsilon\Delta_N,
\end{align*}
where $\Delta_N = T/N$. Outside of this interval we have that $S_\varepsilon'$ is nonnegative and bounded from above by $1/(1-4\varepsilon)$. Putting together these observations we have that 
\begin{align}\label{est:incredible-mess}
 \int^T_0 & \norm{ {v^N_\varepsilon}'(t) - v'(t)}_{L^{\alpha +1}(\Omega)}^{\alpha+1} \d t = \sum^{N-1}_{k=0} \int^{t_{k+1}}_{t_k} \norm{ {v^N_\varepsilon}'(t) - v'(t)}_{L^{\alpha +1}(\Omega)}^{\alpha+1} \d t
 \\
 \notag & = \sum^{N-1}_{k=0} \int_{[t_k, t_k + 3\varepsilon \Delta_N]\cup[t_{k+1} - 3\varepsilon \Delta_N, t_{k+1}]} \norm{ {v^N_\varepsilon}'(t) - v'(t)}_{L^{\alpha +1}(\Omega)}^{\alpha+1} \d t
 \\
 \notag & \quad + \sum^{N-1}_{k=0} \int^{t_{k+1} - 3\varepsilon \Delta_N}_{t_k + 3\varepsilon \Delta_N} \Big|\Big| \Big( 1 + \frac{4\varepsilon}{1 - 4\varepsilon}\Big) \dashint^{t_{k+1}}_{t_k} v'(s) \d s - v'(t) \Big|\Big|_{L^{\alpha +1}(\Omega)}^{\alpha+1} \d t.
\end{align}
The second sum on the right-hand side of \eqref{est:incredible-mess} we estimate as 
\begin{align}\label{est:mess2}
 \sum^{N-1}_{k=0} & \int^{t_{k+1} - 3\varepsilon \Delta_k}_{t_k + 3\varepsilon \Delta_k} \Big|\Big| \Big( 1 + \frac{4\varepsilon}{1 - 4\varepsilon}\Big) \dashint^{t_{k+1}}_{t_k} v'(s) \d s - v'(t) \Big|\Big|_{L^{\alpha +1}(\Omega)}^{\alpha+1} \d t 
 \\
\notag &\leq  c \sum^{N-1}_{k=0} \int^{t_{k+1} - 3\varepsilon \Delta_k}_{t_k + 3\varepsilon \Delta_k}\Big|\Big| \dashint^{t_{k+1}}_{t_k} \big[ v'(s) - v'(t) \big]\d s \Big|\Big|^{\alpha+1}_{L^{\alpha +1}(\Omega)}\d t 
 \\
\notag & \quad + c \varepsilon^{\alpha+1} \sum^{N-1}_{k=0} \int^{t_{k+1} - 3\varepsilon \Delta_k}_{t_k + 3\varepsilon \Delta_k} \Big|\Big| \dashint^{t_{k+1}}_{t_k} v'(s) \d s \Big|\Big|^{\alpha+1}_{L^{\alpha +1}(\Omega)} \d t.
\end{align}
By the uniform continuity of $v'$ on the closed interval $[0,T]$ into $L^{\alpha+1}$, we can make the first sum on the right-hand side of \eqref{est:mess2} arbitrarily small by choosing $N$ large. Then, with $N$ fixed we can use the boundedness of $v'$ as a map on the interval $[0,T]$ into $L^{\alpha+1}$ to ensure that also the second sum on the right-hand side of \eqref{est:mess2} is small, by choosing $\varepsilon$ small. Finally, note that ${v^N_\varepsilon}'$ has a uniform bound as a map on the interval $[0,T]$ into $L^{\alpha+1}$ as $\varepsilon\to 0$. This shows that also the first sum on the right-hand side of \eqref{est:incredible-mess} converges to zero as $\varepsilon\to 0$. Thus, we have shown that functions of the form $v^N_\varepsilon$ approximate $v$ in the norm of $W^{1, \alpha + 1}(0, T; L^{\alpha + 1}(\Omega)) \cap L^{\bf p}(0,T; \overline W^{1, {\bf p}}_\textrm{o}(\Omega))$ as $N$ is chosen large and $\varepsilon$ is chosen small.

 \vspace{2mm} 
\noindent \textit{Step: 4:} Given a function $v^N_\varepsilon$ we can for every $k \in \{0, \dots, N\}$ choose $\varphi_k \in C^\infty_{\textnormal{o}}(\Omega)$ which approximates $v(t_k)$ in $L^{\alpha+1}(\Omega) \cap \overline W^{1, {\bf p}}_\textrm{o}(\Omega)$. Finally, we define 
\begin{align*}
 \varphi(x,t) := \varphi_k(x)\big(1 - S^{t_k, t_{k+1}}_\varepsilon(t)\big) + \varphi_{k+1}(x) S^{t_k, t_{k+1}}_\varepsilon(t), \quad t \in [t_k, t_{k+1}].
\end{align*}
Note that due to the construction of $v$, also $v^N_k$ will vanish when $t$ is near $0$ or $T$ and thus the approximating $\varphi_k$ can be taken to be $0$ when $k$ is close to $0$ or $N$. This means that $\varphi \in C^\infty_o(\Omega_T)$. Since $\varphi$ clearly approximates $v^N_\varepsilon$ in the norm of $W^{1, \alpha + 1}(0, T; L^{\alpha + 1}(\Omega)) \cap L^{\bf p}(0,T; \overline W^{1, {\bf p}}_\textrm{o}(\Omega))$, we are done.
\end{proof}

The previous approximation result is useful for showing that one can use a larger function class in the weak formulation with general vector field $\mathcal{A}$.
\begin{lem}\label{lem:weaksol2}
Let $\mathcal{A}:\Omega_T \to \R^N$, $\mathcal{A}_j \in L^{p_j}(\Omega_T)$, let $f\in L^\frac{\alpha+1}{\alpha}(\Omega_T)$ and suppose that $u \in L^{\bf p}(0,T; W^{1, {\bf p}}(\Omega)) \cap L^{\alpha+1}(\Omega_T)$. Then
\begin{align}\label{gggg}
 \iint_{\Omega_T} \mathcal{A} \cdot \nabla \varphi - |u|^{\alpha-1}u \partial_t \varphi \d x \d t = \iint_{\Omega_T} f \varphi\d x\d t,
\end{align}
for all $\varphi \in C^\infty_{\textnormal{o}}(\Omega_T)$ iff
\begin{align}\label{weakdef-new}
&\iint_{\Omega_T} \mathcal{A} \cdot \nabla v - |u|^{\alpha-1} u\partial_t v \d x\d t= \iint_{\Omega_T} f v \d x \d t,
\end{align}
holds for all $v \in W^{1, \alpha + 1}(0, T; L^{\alpha + 1}(\Omega)) \cap L^{\bf p}(0,T; \overline W^{1, {\bf p}}_{\textrm{o}}(\Omega))$ for which $\varphi(0) = \bar 0 = \varphi(T)$. 
\end{lem}
\noindent Since every function in $W^{1, \alpha + 1}(0, T; L^{\alpha + 1}(\Omega))$ has a representative in $C([0,T]; L^{\alpha+1}(\Omega))$, speaking of the pointwise behavior of $\varphi$ at $0$ and $T$ makes sense. 

\vspace{1mm}
\noindent \textbf{Proof of Lemma \ref{lem:weaksol2}.} The condition \eqref{weakdef-new} implies \eqref{gggg} simply because the function class is in \eqref{weakdef-new} is bigger. For the converse assume \eqref{gggg} and let $v$ be as in \eqref{weakdef-new}. By Lemma \ref{lem:approx_with_smooth_funct}, there is a sequence $v_j \in C^\infty_\textnormal{o}(\Omega_T)$ which converges to $v$ in the norm of  $W^{1, \alpha + 1}(0, T; L^{\alpha + 1}(\Omega)) \cap L^{\bf p}(0,T; \overline W^{1, {\bf p}}_\textrm{o}(\Omega))$. By assumption, \eqref{eq:weak_form} holds with $\varphi = v_j$. Since every term in the integrals consists of factors with matching H\"older exponents, we are able to pass to the limit $j\to \infty$, replacing $v_j$ by $v$.
\qed

For solutions belonging to $C([0,T]; L^{\alpha+1}(\Omega))$, we also have the following equivalent formulation.
\begin{lem}\label{lem:3rd-equiv-formulation}
 Suppose that $u \in C([0,T]; L^{\alpha+1}(\Omega)) \cap L^{\bf p}(0,T; W^{1, {\bf p}}(\Omega))$, and let $\mathcal{A}$ and $f$ be as in Lemma \ref{lem:weaksol2}. Then \eqref{gggg} holds iff
 \begin{align}\label{eq:3rd-formulation}
  &\iint_{\Omega \times [t_1,t_2]} \mathcal{A} \cdot \nabla v - |u|^{\alpha-1} u\partial_t v \d x\d t + \Big[ \int_\Omega |u|^{\alpha-1} u v \d x \Big]^{t_2}_{t_1} = \iint_{\Omega \times [t_1,t_2]} f \varphi \d x \d t.
 \end{align}
for all $v \in W^{1, \alpha + 1}(0, T; L^{\alpha + 1}(\Omega)) \cap L^{\bf p}(0,T; \overline W^{1, {\bf p}}_{\textnormal{o}}(\Omega))$ and all $0\leq t_1 < t_2 \leq T$.
\end{lem}
\begin{proof}{}
 Let $v$ be as in the statement of the lemma. Suppose that $u$ satisfies \eqref{gggg}, and let $0 \leq t_1 < t_2 \leq T$. Define for small $\varepsilon>0$ the trapezoidal function
 \begin{align}\label{func:trapets}
 \psi_\varepsilon(t) := 
 \left\{
\begin{array}{ll}
0, & \quad t < t_1
\\[5pt]
 \varepsilon^{-1}(t - t_1),  & \quad t \in [t_1, t_1 + \varepsilon]
 \\
1,  & \quad t \in (t_1+ \varepsilon, t_2 - \varepsilon)
\\
1 - \varepsilon^{-1}(t - t_2 + \varepsilon), & \quad t \in [t_2 - \varepsilon, t_2]
\\
0, & \quad t > t_2.
\end{array}
\right.
\end{align}
Then $\psi_\varepsilon v$ belongs to $W^{1, \alpha + 1}(0, T; L^{\alpha + 1}(\Omega)) \cap L^{\bf p}(0,T; \overline W^{1, {\bf p}}_{\textnormal{o}}(\Omega))$ and $\psi_\varepsilon v$ vanishes at $0$ and $T$ so by Lemma \ref{lem:weaksol2} we have that 
\begin{align*}
&\iint_{\Omega_T} \mathcal{A}\cdot \nabla (\psi_\varepsilon v) - |u|^{\alpha-1} u\partial_t (\psi_\varepsilon v) \d x\d t= \iint_{\Omega_T} f \psi_\varepsilon v\d x \d t.
\end{align*}
Passing to the limit in the terms involving $\mathcal{A}$ and $f$ is easy, as $\psi_\varepsilon$ converges to the characteristic function of the interval $[t_1, t_2]$. For the remaining term we note that 
\begin{align*}
 \iint_{\Omega_T} |u|^{\alpha-1} u\partial_t (\psi_\varepsilon v) \d x\d t =  \iint_{\Omega_T} |u|^{\alpha-1} u \psi_\varepsilon \partial_t v \d x\d t + \iint_{\Omega_T} |u|^{\alpha-1} u  v \psi_\varepsilon'  \d x\d t.
\end{align*}
Passing to the limit in the first term is again easy. The second term becomes 
\begin{align}\label{simple-step}
 \iint_{\Omega_T} |u|^{\alpha-1} u  v \psi_\varepsilon'  \d x\d t = \frac1\varepsilon \int^{t_1 + \varepsilon}_{t_1} \int_\Omega |u|^{\alpha-1} u  v \d x \d t - \frac1\varepsilon \int^{t_2}_{t_2 - \varepsilon} \int_\Omega |u|^{\alpha-1} u  v \d x \d t.
\end{align}
Due to the time continuity of $|u|^{\alpha - 1}u$ and $v$ into $L^p$-spaces with matching exponents, we have that $|u|^{\alpha - 1}u v$ is continuous on $[0,T]$ into $L^1(\Omega)$. Therefore, passing to the limit $\varepsilon\to 0$  on the right-hand side of \eqref{simple-step} we end up with the last term on the left-hand side of \eqref{eq:3rd-formulation}. Putting everything together, we have verified \eqref{eq:3rd-formulation}. Conversely, if \eqref{eq:3rd-formulation} is satisfied for all $v$ as in the statement of the theorem, then especially it holds if $v(0) = v(T) = 0$, and in this case \eqref{eq:3rd-formulation} reduces to \eqref{weakdef-new}, which by Lemma \ref{lem:weaksol2} is equivalent to \eqref{gggg}.
\end{proof}

\end{document}